\documentclass[11pt]{amsart}
\usepackage[latin9]{inputenc}
\usepackage{amstext}
\usepackage{amsthm}
\usepackage{amssymb}
\usepackage{amstext}
\usepackage{amsthm}
\usepackage{eurosym}
\usepackage{amsfonts}

\makeatletter
\numberwithin{equation}{section}
\numberwithin{figure}{section}
\numberwithin{equation}{section}
\numberwithin{figure}{section}
\theoremstyle{plain}

\newtheorem{theorem}{Theorem}
\theoremstyle{plain}

\newtheorem{corollary}{Corollary}

\newtheorem{example}{Example}

\newtheorem{lemma}{Lemma}

\newtheorem{proposition}{Proposition}
\newtheorem{remark}{Remark}

\numberwithin{equation}{section}
\input{tcilatex}
\providecommand{\theoremname}{Theorem}
\makeatother

\begin{document}
\title[stochastic integro-Differential equations]{On the Cauchy problem for
stochastic parabolic integro-differential \ equations with radially
O-regularly varying Levy measure }

\begin{abstract}
Parabolic integro-differential nondegenerate Cauchy problem is considered in
the scale of $L_{p}$ spaces of functions whose regularity is defined by a
Levy measure with O-regulary varying radial profile. Existence and
uniqueness of a solution is proved by deriving apriori estimates. Some
probability density function estimates of the associated Levy process are
used as well.
\end{abstract}

\author{R. Mikulevicius and C. Phonsom}
\date{June 29, 2019}
\maketitle
\tableofcontents

{}

{}

{}

\section{Introduction}

Let $\sigma\in\left(0,2\right)$ and $\mathfrak{A}^{\sigma}$ be the class of
all nonnegative measures $\nu$\ on $\mathbf{R}_{0}^{d}=\mathbf{R}%
^{d}\backslash\left\{ 0\right\} $ such that $\int\left\vert y\right\vert
^{2}\wedge1d\nu<\infty$ and 
\begin{equation*}
\sigma=\inf\left\{ \alpha<2:\int_{\left\vert y\right\vert \leq1}\left\vert
y\right\vert ^{\alpha}d{\nu}<\infty\right\} . 
\end{equation*}
In addition, we assume that for $\nu\in\mathfrak{A}^{\sigma},$ 
\begin{eqnarray*}
\int_{\left\vert y\right\vert >1}\left\vert y\right\vert d\nu & < & \infty%
\text{ if }\sigma\in\left(1,2\right), \\
\int_{R<\left\vert y\right\vert \leq R^{\prime}}yd\nu & = & 0\text{ if }%
\sigma=1\text{ for all }0<R<R^{\prime}<\infty.\text{ }
\end{eqnarray*}

Let $\left(\Omega,\mathcal{F},\mathbf{P}\right)$ be a complete probability
space with a filtration of $\sigma-$algebras on $\mathbb{F}=\left(\mathcal{F}%
_{t},t\geq0\right)$ satisfying the usual conditions. Let $\mathcal{R}\left(%
\mathbb{F}\right)$ be the progressive $\sigma-$algebra on $\left[%
0,\infty\right)\times\Omega.$ Let $\left(U,\mathcal{U},\Pi\right)$ be a
measurable space with $\sigma-$finite measure $\Pi,\mathbf{R}_{0}^{d}=%
\mathbf{R}^{d}\backslash\left\{ 0\right\} .$ Let $p\left(dt,dz\right)$ be $%
\mathbb{F}-$adapted point measures on $\left(\left[0,\infty\right)\times U,%
\mathcal{B}\left(\left[0,\infty\right)\right)\otimes\mathcal{U}\right)$ with
compensator $\Pi\left(dz\right)dt.$ We denote the martingale measure $%
q\left(dt,dz\right)=p\left(dt,dz\right)-\Pi\left(dz\right)dt.$

In this paper we consider the stochastic parabolic Cauchy problem 
\begin{eqnarray}
du\left(t,x\right) & = & \left[Lu(t,x)-\lambda u\left(t,x\right)+f(t,x)%
\right]dt  \label{eq:mainEq} \\
& & +\int_{U}\Phi\left(t,x,z\right)q\left(dt,dz\right),  \notag \\
u\left(0,x\right) & = & g\left(x\right),\left(t,x\right)\in E=\left[0,T%
\right]\times\mathbf{R}^{d},  \notag
\end{eqnarray}
with $\lambda\geq0$ and integro-differential operator 
\begin{equation*}
L\varphi\left(x\right)=L^{\nu}\varphi\left(x\right)=\int\left[%
\varphi(x+y)-\varphi\left(x\right)-\chi_{\sigma}\left(y\right)y\cdot\nabla%
\varphi\left(x\right)\right]\nu\left(dy\right),\varphi\in
C_{0}^{\infty}\left(\mathbf{R}^{d}\right), 
\end{equation*}
where $\nu\in\mathfrak{A}^{\sigma},$ $\chi_{\sigma}\left(y\right)=0$ if $%
\sigma\in\left[0,1\right),\chi_{\sigma}\left(y\right)=1_{\left\{ \left\vert
y\right\vert \leq1\right\} }\left(y\right)$ if $\sigma=1$ and $%
\chi_{\sigma}\left(y\right)=1$ if $\sigma\in(1,2).$ The symbol of $L$ is 
\begin{equation*}
{\psi}\left(\xi\right)=\psi^{\nu}\left(\xi\right)=\int\left[e^{i2\pi\xi\cdot
y}-1-i2\pi\chi_{\sigma}\left(y\right)\xi\cdot y\right]\nu\left(dy\right),\xi%
\in\mathbf{R}^{d}. 
\end{equation*}
Note that $\nu\left(dy\right)=dy/\left\vert y\right\vert ^{d+\sigma}\in${$%
\mathfrak{A}$}$^{\sigma}$ and, in this case, $L=L^{\nu}=c\left(\sigma,d%
\right)\left(-\Delta\right)^{\sigma/2}$, where $\left(-\Delta\right)^{%
\sigma/2}$\ is a fractional Laplacian. The equation (\ref{eq:mainEq}) is
forward Kolmogorov equation for the Levy process associated to $\psi^{\nu}$.
We assume that $g,f$ and $\Phi$ are resp. $\mathcal{F}_{0}\otimes\mathcal{B}%
\left(\mathbf{R}^{d}\right)$- ,$\mathcal{R}\left(\mathbb{F}\right)\otimes%
\mathcal{B}\left(\mathbf{R}^{d}\right)$- , $\Phi$ is $\mathcal{R}\left(%
\mathbb{F}\right)\otimes\mathcal{B}\left(\mathbf{R}^{d}\right)\otimes%
\mathcal{U}$-measurable.

We define for $\nu\in\mathfrak{A}^{\sigma}$ its radial distribution function 
\begin{equation*}
\delta\left(r\right)=\delta_{\nu}\left(r\right)=\nu\left(x\in\mathbf{R}%
^{d}:\left\vert x\right\vert >r\right),r>0, 
\end{equation*}
and

\begin{equation*}
w\left(r\right)=w_{\nu}\left(r\right)=\delta_{\nu}\left(r\right)^{-1},r>0. 
\end{equation*}
Let $\delta$ be continuous, and $\lim_{r\rightarrow0}\delta\left(r\right)=%
\infty.$ One of our main assumptions is that $w_{\nu}\left(r\right)$ is an
O-RV function at both infinity and at zero. That is 
\begin{equation*}
r_{1}\left(x\right)=\limsup_{\epsilon\rightarrow0}\frac{w_{\nu}\left(%
\epsilon x\right)}{w_{\nu}\left(\epsilon\right)}<\infty,\hspace{1em}%
r_{2}\left(x\right)=\limsup_{\epsilon\rightarrow\infty}\frac{%
w_{\nu}\left(\epsilon x\right)}{w_{\nu}\left(\epsilon\right)}<\infty,x>0. 
\end{equation*}
The regular variation functions were introduced in \cite{k} and used in
tauberian theorems which were extended to O-RV functions as well (see \cite%
{ALAR}, \cite{bgt}, and references therein). They are very convenient for
the derivation of our main estimates.

Given $\nu\in${$\mathfrak{A}$}$^{\sigma},p\in\left[1,\infty\right),s\in%
\mathbf{R}$, we denote $H_{p}^{s}\left(E\right)=H_{p}^{\nu;s}\left(E\right)$
the closure in $L_{p}\left(E\right)$ of $C_{0}^{\infty}\left(E\right)$ with
respect to the norm 
\begin{equation*}
\left\vert f\right\vert _{H_{p}^{\nu;s}\left(E\right)}=\left\vert \mathcal{F}%
^{-1}\left(1-\func{Re}\psi^{\nu}\right)^{s}\mathcal{F}f\right\vert
_{L_{p}\left(E\right)},\, 
\end{equation*}
where $\mathcal{F}$ is the Fourier transform in space variable. We denote $%
\mathbb{H}_{p}^{s}\text{$\left(E\right)$}=\mathbb{H}_{p}^{\nu;s}\left(E%
\right)$ the corresponding space of random functions $u\left(t,x\right),%
\left(t,x\right)\in E$. In this paper, under O-RV of $w$\ and nondegeneracy
assumptions (see assumptions \textbf{A}$,$\textbf{B} below), we prove the
existence and uniqueness of solutions to (\ref{eq:mainEq}) in the scale of
spaces $\mathbb{H}_{p}^{\nu;s}$). Moreover, \ 
\begin{equation}
\left\vert u\right\vert _{\mathbb{H}_{p}^{s}\left(E\right)}\leq C\left[%
\left\vert f\right\vert _{\mathbb{H}_{p}^{s-1}\left(E\right)}+\left\vert
g\right\vert _{\mathbb{H}_{p}^{s-\frac{1}{p}}\left(\mathbf{R}%
^{d}\right)}+\left\vert \Phi\right\vert _{\mathbb{H}_{2,p}^{s-\frac{1}{2}%
}\left(E\right)}+\left\vert \Phi\right\vert _{\mathbb{B}_{p,pp}^{s-\frac{1}{p%
}}\left(E\right)}\right]  \label{4}
\end{equation}
if $p\geq2,$where $\mathbb{B}_{p,pp}^{s}$ is the Besov "counterpart" of $%
\mathbb{H}_{p}^{s}$.

This paper is a continuation of \cite{MikPh1} and \cite{MikPh2}, where (\ref%
{eq:mainEq}) with $\Phi=0$ was considered. The conditions imposed in this
paper are more natural for needed computations. Since the symbol $%
\psi^{\nu}\left(\xi\right)$ is not smooth in $\xi$, the standard Fourier
multiplier results do not apply in this case. In order to prove the estimate
involving $\Phi$ in (\ref{4}), we follow the idea of \cite{KK2}, by applying
a version of Calderon-Zygmund theorem by associating to $L^{\nu}$ a family
of balls and verifying for it the stochastic Hörmander condition (see
Theorem 2.5 in \cite{KK2}). As an example, we consider $\nu\in${$\mathfrak{A}
$}$^{\sigma}$ defined in radial and angular coordinates $r=\left\vert
y\right\vert ,z=y/r,$ as 
\begin{equation}
\nu\left(\Gamma\right)=\int_{0}^{\infty}\int_{\left\vert z\right\vert
=1}\chi_{\Gamma}\left(rz\right)a\left(r,z\right)j\left(r\right)r^{d-1}S%
\left(dz\right)dr,\Gamma\in\mathcal{B}\left(\mathbf{R}_{0}^{d}\right),
\label{5}
\end{equation}
where $S\left(dz\right)$ is a finite measure on the unit sphere on $\mathbf{R%
}^{d}$. In \cite{zh}, (\ref{eq:mainEq}) was considered in the standard
fractional Sobolev spaces using $L^{\infty}$-$BMO$ type estimate, with $g=0$%
, $\Phi=0$ and a nondegenerate $\nu$ in the form (\ref{5}) with $%
a=1,j\left(r\right)=r^{-d-\sigma}.$ In \cite{KK1}, an elliptic problem in
the whole space with $L^{\nu}$ was studied for $\nu$ in the form (\ref{5})
with $S\left(dz\right)=dz$ being a Lebesgue measure on the unit sphere in $%
\mathbf{R}^{d}$, with $0<c_{1}\leq a\leq c_{2}$, and a set of technical
assumptions on $j\left(r\right)$. A sharp function estimate based on the
solution Hölder norm estimate (following the idea in \cite{KimDong}) was
used in \cite{KK1}. The conditions imposed in this paper are more general
than those in \cite{KK1,MikPh1,MikPh2,zh} and thus our main results cover
more examples of $\nu$ (see in particular example \ref{ex: ex1}.)

The paper is organized as follows. In Section 2, the main theorem is stated,
and some examples are considered. In Section 3, some auxiliary results on
O-RV functions and probability density estimates are presented. In section
4, results on characterization of function spaces and approximation of input
functions are presented. In section 5, the main result is proved. In
Appendix, stochastic integrals driven by jump measures are constructed and a
non-degeneracy estimate presented.

\section{Notations, function spaces, main results and examples}

\subsection{Notation}

The following notation will be used in the paper.

Let $\mathbf{N}=\{1,2,\ldots\},\mathbf{N}_{0}=\left\{ 0,1,\ldots\right\} ,%
\mathbf{R}_{0}^{d}=\mathbf{R}^{d}\backslash\{0\}.$ If $x,y\in\mathbf{R}^{d},$%
\ we write 
\begin{equation*}
x\cdot y=\sum_{i=1}^{d}x_{i}y_{i},\,|x|=\sqrt{x\cdot x}. 
\end{equation*}

We denote by $C_{0}^{\infty}(\mathbf{R}^{d})$ the set of all infinitely
differentiable functions on $\mathbf{R}^{d}$ with compact support.

We denote the partial derivatives in $x$ of a function $u(t,x)$ on $\mathbf{R%
}^{d+1}$ by $\partial_{i}u=\partial u/\partial x_{i}$, $\partial_{ij}^{2}u=%
\partial^{2}u/\partial x_{i}\partial x_{j}$, etc.; $Du=\nabla
u=(\partial_{1}u,\ldots,\partial_{d}u)$ denotes the gradient of $u$ with
respect to $x$; for a multiindex $\gamma\in\mathbf{N}_{0}^{d}$ we denote 
\begin{equation*}
D_{x}^{{\scriptsize \gamma}}u(t,x)=\frac{\partial^{|{\scriptsize \gamma|}%
}u(t,x)}{\partial x_{1}^{{\scriptsize \gamma_{1}}}\ldots\partial x_{d}^{%
{\scriptsize \gamma_{d}}}}. 
\end{equation*}
For $\alpha\in(0,2]$ and a function $u(t,x)$ on $\mathbf{R}^{d+1}$, we write 
\begin{equation*}
\partial^{{\scriptsize \alpha}}u(t,x)=-\mathcal{F}^{-1}[|\xi|^{{\scriptsize %
\alpha}}\mathcal{F}u(t,\xi)](x), 
\end{equation*}
where 
\begin{equation*}
\mathcal{F}h(t,\xi)=\hat{h}\left(t,\xi\right)=\int_{\mathbf{R}^{d}}\,\mathrm{%
e}^{-i2\pi\xi\cdot x}h(t,x)dx,\mathcal{F}^{-1}h(t,\xi)=\int_{\mathbf{R}%
^{d}}\,\mathrm{e}^{i2\pi\xi\cdot x}h(t,\xi)d\xi. 
\end{equation*}

For $\nu\in${$\mathfrak{A}$}$^{\sigma}$, we denote $Z_{t}^{\nu},t\geq0,$ the
Levy process associated to $L^{\nu}$, i.e., $Z_{t}^{\nu}$ is cadlag with
independent increments and its characteristic function 
\begin{equation*}
\mathbf{E}e^{i2\pi\xi\cdot Z_{t}^{\nu}}=\exp\left\{
\psi^{\nu}\left(\xi\right)t\right\} ,\xi\in\mathbf{R}^{d},t\geq0. 
\end{equation*}
The letters $C=C(\cdot,\ldots,\cdot)$ and $c=c(\cdot,\ldots,\cdot)$ denote
constants depending only on quantities appearing in parentheses. In a given
context the same letter will (generally) be used to denote different
constants depending on the same set of arguments.

\subsection{Function spaces}

Let $V$ be a Banach space with norm $\left\vert \cdot\right\vert _{V}$. Let $%
S\left(\mathbf{R}^{d};V\right)$ be the Schwartz space of $V$-valued rapidly
decreasing functions. We use standard notation $S\left(\mathbf{R}^{d}\right)$
when $V=\mathbf{R}$.

For a $V-$valued measurable function $h$ on $\mathbf{R}^{d}$and $p\geq1$ we
denote 
\begin{equation*}
\left|h\right|_{V,p}^{p}=\int_{\mathbf{R}^{d}}\left|h\left(x\right)%
\right|_{V}^{p}dx. 
\end{equation*}

We fix $\nu\in${$\mathfrak{A}$}$^{\sigma}$. Obviously, $\text{Re}%
\psi^{\nu}=\psi^{\nu_{sym}}$, where 
\begin{equation*}
\nu_{sym}\left(dy\right)=\frac{1}{2}\left[\nu\left(dy\right)+\nu\left(-dy%
\right)\right]. 
\end{equation*}
Let 
\begin{equation*}
Jf=J_{\nu}f=(I-L^{\nu_{sym}})f=f-L^{\nu_{sym}}f,f\in\mathcal{S}\left(\mathbf{%
R}^{d},V\right). 
\end{equation*}
For $s\in\mathbf{R}$ set 
\begin{equation*}
J^{s}f=\left(I-L^{\nu_{sym}}\right)^{s}f=\mathcal{F}^{-1}[(1-\psi^{%
\nu_{sym}})^{s}\hat{f}],f\in\mathcal{S}\left(\mathbf{R}^{d},V\right). 
\end{equation*}

\begin{equation*}
L^{\nu;s}f=\mathcal{F}^{-1}\left(-\left(-\psi^{\nu_{sym}}\right)^{s}\hat{f}%
\right),f\in\mathcal{S}\left(\mathbf{R}^{d},V\right). 
\end{equation*}

Note that $L^{\nu;1}f=L^{\nu_{sym}}f,f\in\mathcal{S}\left(\mathbf{R}%
^{d}\right).$

For $p\in\left[1,\infty\right),s\in\mathbf{R,}$ we define, following \cite%
{fjs}, the Bessel potential space $H_{p}^{s}\left(\mathbf{R}%
^{d},V\right)=H_{p}^{\nu;s}\left(\mathbf{R}^{d},V\right)$ as the closure of $%
\mathcal{S}\left(\mathbf{R}^{d},V\right)$ in the norm 
\begin{eqnarray*}
\left\vert f\right\vert _{H_{p}^{s}\left(\mathbf{R}^{d},V\right)} & = &
\left\vert J^{s}f\right\vert _{L_{p}\left(\mathbf{R}^{d},V\right)}=\left%
\vert \mathcal{F}^{-1}[(1-\psi^{\nu_{sym}})^{s}\hat{f}]\right\vert
_{L_{p}\left(\mathbf{R}^{d},V\right)} \\
& = & \left\vert \left(I-L^{\nu_{sym}}\right)^{s}f\right\vert _{L_{p}\left(%
\mathbf{R}^{d},V\right)},f\in\mathcal{S}\left(\mathbf{R}^{d},V\right).
\end{eqnarray*}
According to Theorem 2.3.1 in \cite{fjs}, $H_{p}^{t}\left(\mathbf{R}%
^{d}\right)\subseteq H_{p}^{s}\left(\mathbf{R}^{d}\right)\,\ $ is
continuously embedded if $p\in\left(1,\infty\right),s<t$, $H_{p}^{0}\left(%
\mathbf{R}^{d}\right)=L_{p}\left(\mathbf{R}^{d}\right)$. For $s\geq0,p\in%
\left[1,\infty\right),$ the norm $\left\vert f\right\vert _{H_{p}^{s}\left(%
\mathbf{R}^{d}\right)}$ is equivalent to (see Theorem 2.2.7 in \cite{fjs}) 
\begin{equation*}
\left\vert f\right\vert _{H_{p}^{s}\left(\mathbf{R}^{d}\right)}=\left\vert
f\right\vert _{L_{p}\left(\mathbf{R}^{d}\right)}+\left\vert \mathcal{F}^{-1}%
\left[(-\psi^{\nu_{sym}})^{s}\mathcal{F}f\right]\right\vert _{L_{p}\left(%
\mathbf{R}^{d}\right)}. 
\end{equation*}

Further, for a characterization of our function spaces we will use the
following construction (see \cite{lofs}).

\begin{remark}
\label{rem:system}For an integer $N>1$ there exists a function $%
\phi=\phi^{N}\in C_{0}^{\infty}(\mathbf{R}^{d})$ (see Lemma 6.1.7 in \cite%
{lofs}), such that $\mathrm{supp}~\phi=\left\{ \xi:\frac{1}{N}\leq\left\vert
\xi\right\vert \leq N\right\} $ , $\phi(\xi)>0$ if $N^{-1}<|\xi|<N$ and 
\begin{equation*}
\sum_{j=-\infty}^{\infty}\phi(N^{-j}\xi)=1\quad\text{if }\xi\neq0. 
\end{equation*}
Let 
\begin{equation}
\tilde{\phi}\left(\xi\right)=\phi\left(N\xi\right)+\phi\left(\xi\right)+\phi%
\left(N^{-1}\xi\right),\xi\in\mathbf{R}^{d}.  \label{pp1}
\end{equation}
Note that supp $~\tilde{\phi}\subseteq\left\{ N^{-2}\leq\left\vert
\xi\right\vert \leq N^{2}\right\} $ and $\tilde{\phi}\phi=\phi$. Let $%
\varphi_{k}=\varphi_{k}^{N}=\mathcal{F}^{-1}\phi\left(N^{-k}\cdot\right),k%
\geq1,$ and $\varphi_{0}=\varphi_{0}^{N}\in\mathcal{S}\left(\mathbf{R}%
^{d}\right)$ is defined as 
\begin{equation*}
\varphi_{0}=\mathcal{F}^{-1}\left[1-\sum_{k=1}^{\infty}\phi\left(N^{-k}\cdot%
\right)\right]. 
\end{equation*}
Let $\phi_{0}\left(\xi\right)=\mathcal{F}\varphi_{0}\left(\xi\right),\tilde{%
\phi}_{0}\left(\xi\right)=\mathcal{F}\varphi_{0}\left(\xi\right)+\mathcal{%
F\varphi}_{1}\left(\xi\right),\xi\in\mathbf{R}^{d}\mathbf{,}\tilde{\varphi}=%
\mathcal{F}^{-1}\tilde{\phi},\varphi=\mathcal{F}^{-1}\phi,$ and 
\begin{equation*}
\tilde{\varphi}_{k}=\sum_{l=-1}^{1}\varphi_{k+l},k\geq1,\tilde{\varphi}%
_{0}=\varphi_{0}+\varphi_{1} 
\end{equation*}
that is 
\begin{eqnarray*}
\mathcal{F\tilde{\varphi}}_{k} & = &
\phi\left(N^{-k+1}\xi\right)+\phi\left(N^{-k}\xi\right)+\phi\left(N^{-k-1}%
\xi\right) \\
& = & \tilde{\phi}\left(N^{-k}\xi\right),\xi\in\mathbf{R}^{d},k\geq1.
\end{eqnarray*}
Note that $\varphi_{k}=\tilde{\varphi}_{k}\ast\varphi_{k},k\geq0$.
Obviously, $f=\sum_{k=0}^{\infty}f\ast\varphi_{k}$ in $\mathcal{S}%
^{\prime}\left(\mathbf{R}^{d}\right)$ for $f\in\mathcal{S}\left(\mathbf{R}%
^{d}\right).$
\end{remark}

Let $s\in\mathbf{R}$ and $p,q\geq1$. For $\nu\in${$\mathfrak{A}$}$^{\sigma}$%
, we introduce the Besov space $B_{pq}^{s}=B_{pq}^{\nu;s}(\mathbf{R}^{d},V)$
as the closure of $\mathcal{S}\left(\mathbf{R}^{d},V\right)$ in the norm 
\begin{equation*}
|f|_{B_{pq}^{s}(\mathbf{R}^{d},V)}=|f|_{B_{pq}^{\nu;s}(\mathbf{R}%
^{d},V)}=\left(\sum_{j=0}^{\infty}|J^{s}\varphi_{j}\ast f|_{L_{p}\left(%
\mathbf{R}^{d},V\right)}^{q}\right)^{1/q}, 
\end{equation*}
where $J=J_{\nu}=I-L^{\nu_{sym}}.$

We introduce the corresponding spaces of functions on $E=[0,T]\times\mathbf{R%
}^{d}$ $.$ The spaces $B_{pq}^{\nu;s}(E,V)$ (resp. $H_{p}^{\nu;s}(E,V)$)
consist of all measurable $B_{pq}^{\nu;s}(\mathbf{R}^{d},V)$ (resp. $%
H_{p}^{\nu;s}(\mathbf{R}^{d},V)$) -valued functions $f$ on $[0,T]$ with
finite corresponding norms: 
\begin{eqnarray}
|f|_{B_{pq}^{s}(E,V)} & = &
|f|_{B_{pq}^{\nu;s}(E,V)}=\left(\int_{0}^{T}|f(t,\cdot)|_{B_{pq}^{\nu;s}(%
\mathbf{R}^{d},V)}^{q}dt\right)^{1/q},  \notag \\
|f|_{H_{p}^{s}(E,V)} & = &
|f|_{H_{p}^{\nu;s}(E,V)}=\left(\int_{0}^{T}|f(t,\cdot)|_{H_{p}^{\nu,s}(%
\mathbf{R}^{d},V)}^{p}dt\right)^{1/p}.  \label{norm11}
\end{eqnarray}

Similarly we introduce the corresponding spaces of random functions. Let $%
\left(\Omega,\mathcal{F},\mathbf{P}\right)$ be a complete probability spaces
with a filtration of $\sigma-$algebras $\mathbb{F}=\left(\mathcal{F}%
_{t}\right)$ satisfying the usual conditions. Let $\mathcal{R}\left(\mathbb{F%
}\right)$ be the progressive $\sigma-$algebra on $\left[0,\infty\right)%
\times\Omega$.

The spaces $\mathbb{B}_{pp}^{s}\left(\mathbf{R}^{d},V\right)$ and $\mathbb{H}%
_{p}^{s}\left(\mathbf{R}^{d},V\right)$ consists of all $\mathcal{F}-$%
measurable random functions $f$ with values in $B_{pp}^{s}\left(\mathbf{R}%
^{d},V\right)$ and $H_{p}^{s}\left(\mathbf{R}^{d},V\right)$ with finite
norms 
\begin{equation*}
\left|f\right|_{\mathbb{B}_{pp}^{s}\left(\mathbf{R}^{d},V\right)}=\left\{ 
\mathbf{E}\left|f\right|_{B_{pp}^{s}\left(\mathbf{R}^{d},V\right)}^{p}\right%
\} ^{1/p} 
\end{equation*}

and 
\begin{equation*}
\left|f\right|_{\mathbb{H}_{p}^{s}\left(\mathbf{R}^{d},V\right)}=\left\{ 
\mathbf{E}\left|f\right|_{H_{p}^{s}\left(\mathbf{R}^{d},V\right)}^{p}\right%
\} ^{1/p}. 
\end{equation*}

The spaces $\mathbb{B}_{pp}^{s}\left(E,V\right)$ and $\mathbb{H}%
_{p}^{s}\left(E,V\right)$ consist of all $\mathcal{R}\left(\mathbb{F}\right)-
$measurable random functions with values in $B_{pp}^{s}\left(E,V\right)$ and 
$H_{p}^{s}\left(E,V\right)$ with finite norms 
\begin{equation*}
\left|f\right|_{\mathbb{B}_{pp}^{s}\left(E,V\right)}=\left\{ \mathbf{E}%
\left|f\right|_{B_{pp}^{s}\left(E,V\right)}^{p}\right\} ^{1/p} 
\end{equation*}

and 
\begin{equation*}
\left\vert f\right\vert _{\mathbb{H}_{p}^{s}\left(E,V\right)}=\left\{ 
\mathbf{E}\left\vert f\right\vert _{H_{p}^{s}\left(E,V\right)}^{p}\right\}
^{1/p}. 
\end{equation*}

\begin{remark}
\label{rem:embedding}(see \cite{MikPh2}) For every $\varepsilon>0$, $p>1$, $%
B_{pp}^{\nu;s+\varepsilon}\left(\mathbf{R}^{d}\right)$ is continuously
embedded into $H_{p}^{\nu;s}\left(\mathbf{R}^{d}\right)$; for $p\geq2,$ $%
H_{p}^{\nu;s}\left(\mathbf{R}^{d}\right)$ is continuously embedded into $%
B_{pp}^{\nu;s}\left(\mathbf{R}^{d}\right).$
\end{remark}

If $V_{r}=L_{r}\left(U,\mathcal{U},\Pi\right),r\geq1$, the space of $r-$%
integrable measurable functions on $U$, and $V_{0}=\mathbf{R}$, we write 
\begin{eqnarray*}
B_{r,pp}^{s}\left(A\right) & = & B_{pp}^{s}\left(A,V\right),\hspace{1em}%
\mathbb{B}_{r,pp}^{s}\left(A\right)=\mathbb{\mathbb{B}}_{pp}^{s}\left(A,V%
\right), \\
H_{r,p}^{s}\left(A\right) & = & H_{p}^{s}\left(A,V\right),\hspace{1em}%
\mathbb{H}_{r,p}^{s}\left(A\right)=\mathbb{H}_{p}^{s}\left(A,V\right),
\end{eqnarray*}
and 
\begin{equation*}
L_{r,p}\left(A\right)=H_{r,p}^{0}\left(A\right),\hspace{1em}\mathbb{L}%
_{r,p}\left(A\right)=\mathbb{H}_{r,p}^{0}\left(A\right), 
\end{equation*}
where $A=\mathbf{R}^{d}$ or $E$. For scalar functions we drop $V$ in the
notation of function spaces.

Let $U_{n}\in\mathcal{U},U_{n}\subseteq U_{n+1},n\geq1,\cup_{n}U_{n}=U$ and $%
\Pi\left(U_{n}\right)<\infty,n\geq1.$ We denote by $\mathbb{\tilde{C}}%
_{r.p}^{\infty}\left(E\right),1\leq p<\infty,$ the space of all $\mathcal{R}%
\left(\mathbb{F}\right)\otimes\mathcal{B}\left(\mathbf{R}^{d}\right)$
-measurable $V_{r}$ -valued random functions $\Phi$ on $E$ such that for
every multiindex $\gamma\in\mathbf{N}_{0}^{d}$, 
\begin{equation*}
\mathbf{E}\int_{0}^{T}\sup_{x\in\mathbf{R}^{d}}\left\vert
D^{\gamma}\Phi\left(t,x\right)\right\vert _{V_{r}}^{p}dt+\mathbf{E}\left[%
\left\vert D^{\gamma}\Phi\right\vert _{L_{p}\left(E;V_{r}\right)}^{p}\right]%
<\infty, 
\end{equation*}
and $\Phi=\Phi\raisebox{2pt}{\ensuremath{\chi}}_{U_{n}}$ for some $n$ if $%
r=2,p.$ Similarly we define the space $\mathbb{\tilde{C}}_{r.p}^{\infty}%
\left(\mathbf{R}^{d}\right)$ by replacing $\mathcal{R}\left(\mathbb{F}\right)
$ and $E$ by $\mathcal{F}$ and $\mathbf{R}^{d}$ respectively in the
definition of $\mathbb{\tilde{C}}_{r,p}^{\infty}\left(E\right)$.

\subsection{Main results}

We set for $\nu\in\mathfrak{A}^{\sigma}$ 
\begin{equation*}
\delta\left(r\right)=\delta_{\nu}\left(r\right)=\nu\left(x\in\mathbf{R}%
^{d}:\left|x\right|>r\right),r>0 
\end{equation*}

\begin{equation*}
w\left(r\right)=w_{\nu}\left(r\right)=\delta_{\nu}\left(r\right)^{-1},r>0. 
\end{equation*}

Our main assumption is that $w=w_{\nu}\left(r\right)$ is an O-RV function at
both infinity and at zero. That is $w$ is a positive, finite, measurable
function and 
\begin{equation*}
r_{1}\left(x\right)=\limsup_{\epsilon\rightarrow0}\frac{w\left(\epsilon
x\right)}{w\left(\epsilon\right)}<\infty,\hspace{1em}r_{2}\left(x\right)=%
\limsup_{\epsilon\rightarrow\infty}\frac{w\left(\epsilon x\right)}{%
w\left(\epsilon\right)}<\infty,x>0. 
\end{equation*}

By Theorem 2 in \cite{ALAR}, the following limit exist: 
\begin{equation}
p_{1}=p_{1}^{w}=\lim_{\epsilon\rightarrow0}\frac{\log
r_{1}\left(\epsilon\right)}{\log\epsilon}\leq
q_{1}=q_{1}^{w}=\lim_{\epsilon\rightarrow\infty}\frac{\log
r_{1}\left(\epsilon\right)}{\log\left(\epsilon\right)}  \label{1}
\end{equation}
and 
\begin{equation}
p_{2}=p_{2}^{w}=\lim_{\epsilon\rightarrow0}\frac{\log
r_{2}\left(\epsilon\right)}{\log\epsilon}\leq
q_{2}=q_{2}^{w}=\lim_{\epsilon\rightarrow\infty}\frac{\log
r_{2}\left(\epsilon\right)}{\log\left(\epsilon\right)}.  \label{2}
\end{equation}
Note that $p_{1}\leq\sigma\leq q_{1}$ (see \cite{MF}). We will assume
throughout this paper that $p_{1},p_{2},q_{1},q_{2}>0$. The numbers $%
p_{1},p_{2}$ are called lower indices and $q_{1},q_{2}$ are called upper
indices of O-RV\ function.

When the context is clear, for a function $f$ which is both O-RV at zero and
infinity we always denote its lower index at zero by $p_{1}$, upper index at
zero by $q_{1}$ , lower index at infinity by $p_{2}$ and upper index at
infinity by $q_{2}$. If we wish to be precise, we will write $%
p_{1}^{f},p_{2}^{f},q_{1}^{f},q_{2}^{f}$. For brevity, we sometimes say that 
$f$ is an O-RV function if it is both O-RV at zero and infinity.

The main result for (\ref{eq:mainEq}) is the following statement.

\begin{theorem}
\label{thm:main1}Let $p\in\left(1,\infty\right),s\in\mathbf{R.}$ Let $\nu\in%
\mathfrak{A}^{\sigma},$ and $w=w_{\nu}$ be continuous O-RV function at zero
and infinity with $p_{i},q_{i},i=1,2,$ defined in (\ref{1}), (\ref{2}).
Assume

\textbf{A. }for $i=1,2$ 
\begin{eqnarray*}
0 & < & p_{i}\leq q_{i}<1\text{ if }\sigma\in\left(0,1\right),0<p_{i}\leq1%
\leq q_{i}<2\text{ if }\sigma=1, \\
1 & < & p_{i}\leq q_{i}<2\text{ if $\sigma$}\in\left(1,2\right).
\end{eqnarray*}

\textbf{B. } 
\begin{equation*}
\inf_{R\in\left(0,\infty\right),\left\vert \hat{\xi}\right\vert
=1}\int_{\left\vert y\right\vert \leq1}\left\vert \hat{\xi}\cdot
y\right\vert ^{2}\tilde{\nu}_{R}\left(dy\right)>0, 
\end{equation*}
where $\tilde{\nu}_{R}\left(dy\right)=w\left(R\right)\nu\left(Rdy\right)$.
Then for each $f\in\mathbb{H}_{p}^{\nu;s}(E),g\in\mathbb{B}%
_{pp}^{\nu;s+1-1/p}\left(\mathbf{R}^{d}\right)$, $\Phi\in\mathbb{B}%
_{p,pp}^{\nu;s+1-1/p}\left(E\right)\cap\mathbb{H}_{2,p}^{\nu;s+1/2}\left(E%
\right)\hspace{1em}$ if $p\in\left[2,\infty\right)$ and $\Phi\in\mathbb{B}%
_{p,pp}^{\nu;s+1-1/p}\left(E\right)$ if\hspace{1em}$p\in\left(1,2\right)$,
there is a unique $u\in\mathbb{H}_{p}^{\nu;s+1}\left(E\right)$ solving (\ref%
{eq:mainEq}). Moreover, there is $C=C\left(d,p,\nu\right)$ such that for $%
p\in\left[2,\infty\right)$, 
\begin{eqnarray*}
& & \left\vert L^{\nu}u\right\vert _{\mathbb{H}_{p}^{\nu;s}\left(E\right)} \\
& \leq & C\left[\left\vert f\right\vert _{\mathbb{H}_{p}^{\nu;s}\left(E%
\right)}+\left\vert g\right\vert _{\mathbb{B}_{pp}^{\nu;s+1-1/p}\left(%
\mathbf{R}^{d}\right)}+\left\vert \Phi\right\vert _{\mathbb{B}%
_{p,pp}^{\nu;s+1-1/p}\left(E\right)}+\left\vert \Phi\right\vert _{\mathbb{H}%
_{2,p}^{\nu;s+1/2}\left(E\right)}\right], \\
& & \left\vert u\right\vert _{\mathbb{H}_{p}^{\nu;s}\left(E\right)} \\
& \leq & C[\rho_{\lambda}\left\vert f\right\vert _{\mathbb{H}%
_{p}^{\nu;s}\left(E\right)}+\rho_{\lambda}^{1/p}\left\vert g\right\vert _{%
\mathbb{H}_{p}^{\nu;s}\left(\mathbf{R}^{d}\right)}+\rho_{\lambda}^{1/p}\left%
\vert \Phi\right\vert _{\mathbb{H}_{p,p}^{\nu;s}\left(E\right)}+\rho_{%
\lambda}^{1/2}\left\vert \Phi\right\vert _{\mathbb{H}_{2,p}^{\nu;s}\left(E%
\right)}],
\end{eqnarray*}
and for $p\in\left(1,2\right)$, 
\begin{eqnarray*}
\left\vert L^{\nu}u\right\vert _{\mathbb{H}_{p}^{\nu;s}\left(E\right)} &
\leq & C\left[\left\vert f\right\vert _{\mathbb{H}_{p}^{\nu;s}\left(E%
\right)}+\left\vert g\right\vert _{\mathbb{B}_{pp}^{\nu;s+1-1/p}\left(%
\mathbf{R}^{d}\right)}+\left\vert \Phi\right\vert _{\mathbb{B}%
_{p,pp}^{\nu;s+1-1/p}\left(E\right)}\right], \\
\left\vert u\right\vert _{\mathbb{H}_{p}^{\nu;s}\left(E\right)} & \leq & C%
\left[\rho_{\lambda}\left\vert f\right\vert _{\mathbb{H}_{p}^{\nu;s}\left(E%
\right)}+\rho_{\lambda}^{1/p}\left\vert g\right\vert _{\mathbb{H}%
_{p}^{\nu;s}\left(\mathbf{R}^{d}\right)}+\rho_{\lambda}^{1/p}\left\vert
\Phi\right\vert _{\mathbb{H}_{p,p}^{\nu;s}\left(E\right)}\right],
\end{eqnarray*}
where $\rho_{\lambda}=\frac{1}{\lambda}\wedge T.$
\end{theorem}

Note that the assumption $\mathbf{A}$ depends on $\nu$ only through $%
w=w_{\nu}$.

The following examples are taken from \cite{MF}.

\begin{example}
\label{ex: ex1}According to \cite{DM} (pp. 70-74), any Levy measure $\nu\in%
\mathfrak{A}^{\sigma}$ can be disintegrated as 
\begin{equation*}
\nu\left(\Gamma\right)=-\int_{0}^{\infty}\int_{S_{d-1}}\chi_{\Gamma}\left(rz%
\right)\Pi\left(r,dz\right)d\delta_{\nu}\left(r\right),\Gamma\in\mathcal{B}%
\left(\mathbf{R}_{0}^{d}\right), 
\end{equation*}
where $\delta=\delta_{\nu}$, and $\Pi\left(r,dz\right),r>0$ is a measurable
family of measures on the unit sphere $S_{d-1}$ with $\Pi\left(r,S_{d-1}%
\right)=1,r>0.$ If $w_{\nu}=\delta^{-1}$ is continuous, O-RV at zero and
infinity, and satisfies assumption $\mathbf{A}$. Assume that $\left\vert
\left\{ s\in\left[0,1\right]:r_{i}\left(s\right)<1\right\} \right\vert
>0,i=1,2,$ and 
\begin{equation*}
\inf_{\left\vert \hat{\xi}\right\vert =1}\int_{S_{d-1}}\left\vert \hat{\xi}%
\cdot z\right\vert ^{2}\Pi\left(r,dz\right)\geq c_{0}>0,\hspace{1em}r>0, 
\end{equation*}
hold, then all assumptions of Theorem \ref{thm:main1} holds. (cf. Lemma \ref%
{cor:Example1} in the Appendix)
\end{example}

\begin{example}
Consider Levy measure in radial and angular coordinate in the form 
\begin{equation*}
\nu\left(B\right)=\int_{0}^{\infty}\int_{\left\vert z\right\vert
=1}1_{B}\left(rz\right)a\left(r,z\right)j\left(r\right)r^{d-1}S\left(dz%
\right)dr,B\in\mathcal{B}\left(\mathbf{R}_{0}^{d}\right), 
\end{equation*}
where $S\left(dz\right)$ is a finite measure on the unit sphere.

Assume

(i) There is $C>1,c>0,0<\delta_{1}\leq\delta_{2}<1$ such that 
\begin{equation*}
C^{-1}\phi\left(r^{-2}\right)\leq j\left(r\right)r^{d}\leq
C\phi\left(r^{-2}\right) 
\end{equation*}
and for all $0<r\leq R$, 
\begin{equation*}
c^{-1}\left(\frac{R}{r}\right)^{\delta_{1}}\leq\frac{\phi\left(R\right)}{%
\phi\left(r\right)}\leq c\left(\frac{R}{r}\right)^{\delta_{2}}. 
\end{equation*}

(ii) There is a function $\rho_{0}\left(z\right)$ defined on the unit sphere
such that $\rho_{0}\left(z\right)\leq a\left(r,z\right)\leq1,r>0,z\in S_{d-1}
$, and for all $\left\vert \hat{\xi}\right\vert =1$, 
\begin{equation*}
\int_{S^{d-1}}\left\vert \hat{\xi}\cdot z\right\vert
^{2}\rho_{0}\left(z\right)S\left(dz\right)\geq c>0. 
\end{equation*}
Under these assumptions it can be shown that $\mathbf{B}$ holds, and $w_{\nu}
$ is an O-RV function with $2\delta_{1}\leq p_{i}^{\nu}\leq
q_{i}^{\nu}\leq2\delta_{2},i=1,2.$
\end{example}

\clearpage

\section{Auxiliary results}

We start with

\subsection{Some estimates of O-RV functions}

We start with a simple but useful corollaries about functions that are O-RV
at both zero and infinity. For $\nu\in\mathfrak{A}^{\sigma},R>0$, we denote $%
\tilde{\nu}_{R}\left(dy\right)=w_{\nu}\left(R\right)\nu\left(Rdy\right)$.

First we note scaling property of functions $w:\left(0,\infty\right)%
\rightarrow\left(0,\infty\right)$ that are non-decreasing, O-RV at both zero
and infinity with strictly positive lower indices.

\begin{lemma}
\label{lem:powerEstratio}Let $\nu\in\mathfrak{A}^{\sigma},$ and $w=w_{\nu}$
be an O-RV function at zero and infinity with $p_{i},q_{i},i=1,2,$ defined
in (\ref{1}), (\ref{2}), and assumption \textbf{A }holds. Then for any $%
\alpha_{1}>q_{1}\vee q_{2}$ , $0<\alpha_{2}<p_{1}\wedge p_{2}$, there exist $%
c_{1}=c_{1}\left(\alpha_{1}\right),c_{2}=c_{2}\left(\alpha_{2}\right)>0$
such that 
\begin{equation*}
c_{1}\left(\frac{y}{x}\right)^{\alpha_{2}}\leq\frac{w\left(y\right)}{%
w\left(x\right)}\leq c_{2}\left(\frac{y}{x}\right)^{\alpha_{1}},0<x\leq
y<\infty. 
\end{equation*}
\end{lemma}

\begin{proof}
Due to similarity, we only show the right hand side of the inequality. By
Karamata characterization (see (1.7) of \cite{ALAR}) of O-RV functions,
there are $0<\eta_{1}<\eta_{2}$ such that the RHS inequality holds if either 
$x\vee y\leq\eta_{1}$ or $x\wedge y\geq\eta_{2}$. If $y\geq\eta_{2}$ and $%
x\leq\eta_{1},$ 
\begin{equation*}
\frac{w\left(y\right)}{w\left(x\right)}=\frac{w\left(y\right)}{%
w\left(\eta_{2}\right)}\frac{w\left(\eta_{2}\right)}{w\left(\eta_{1}\right)}%
\frac{w\left(\eta_{1}\right)}{w\left(x\right)}\leq c\left(\frac{y}{x}%
\right)^{\alpha_{1}}\frac{w\left(\eta_{2}\right)}{w\left(\eta_{1}\right)}, 
\end{equation*}
and similarly we consider other cases.
\end{proof}

\begin{remark}
It trivially follows from Lemma \ref{lem:powerEstratio} that $%
\lim_{r\rightarrow0}\omega\left(r\right)=0$ and $\lim_{r\rightarrow\infty}%
\omega\left(r\right)=\infty.$
\end{remark}

The following statement holds

\begin{lemma}
\label{al1}(Lemma 8 in \cite{MF}) Assume $w\left(r\right),r>0,$ is an O-RV
function at zero with lower and upper indices $p_{1},q_{1},$ that is, 
\begin{equation*}
r_{1}\left(x\right)=\overline{\lim_{\varepsilon\rightarrow0}}\frac{%
w\left(\epsilon x\right)}{w\left(\epsilon\right)}<\infty,x>0, 
\end{equation*}
and 
\begin{equation*}
p_{1}=\lim_{\epsilon\rightarrow0}\frac{\log r_{1}\left(\epsilon\right)}{%
\log\epsilon}\leq q_{1}=\lim_{\epsilon\rightarrow\infty}\frac{\log
r_{1}\left(\epsilon\right)}{\log\left(\epsilon\right)}. 
\end{equation*}
a) Let $\beta>0$ and $\tau>-\beta p_{1}.$ There is $C>0$ so that 
\begin{equation*}
\int_{0}^{x}t^{\tau}w\left(t\right)^{\beta}\frac{dt}{t}\leq
Cx^{\tau}w\left(x\right)^{\beta},x\in(0,1], 
\end{equation*}
and $\lim_{x\rightarrow0}x^{\tau}w\left(x\right)^{\beta}=0.$

\noindent b) Let $\beta>0$ and $\tau<-\beta q_{1}$. There is $C>0$ so that 
\begin{equation*}
\int_{x}^{1}t^{\tau}w\left(t\right)^{\beta}\frac{dt}{t}\leq
Cx^{\tau}w\left(x\right)^{\beta},x\in(0,1],\text{ } 
\end{equation*}
and $\lim_{x\rightarrow0}x^{\tau}w\left(x\right)^{\beta}=\infty.$

\noindent c) Let $\beta<0$ and $\tau>-\beta q_{1}$. There is $C>0$ so that 
\begin{equation*}
\int_{0}^{x}t^{\tau}w\left(t\right)^{\beta}\frac{dt}{t}\leq
Cx^{\tau}w\left(x\right)^{\beta},x\in(0,1], 
\end{equation*}
and $\lim_{x\rightarrow0}x^{\tau}w\left(x\right)^{\beta}=0.$

\noindent d) Let $\beta<0$ and $\tau<-\beta p_{1}$. There is $C>0$ so that 
\begin{equation*}
\int_{x}^{1}t^{\tau}w\left(t\right)^{\beta}\frac{dt}{t}%
=\int_{1}^{x^{-1}}t^{-\tau}w\left(\frac{1}{t}\right)^{\beta}\frac{dt}{t}\leq
Cx^{\tau}w\left(x\right)^{\beta},x\in(0,1], 
\end{equation*}
and $\lim_{x\rightarrow0}x^{\tau}w\left(x\right)^{\beta}=\infty.$
\end{lemma}

Similar statement holds for O-RV functions at infinity.

\begin{lemma}
\label{al2} Assume $w\left(r\right),r>0,$ is an O-RV function at infinity
with lower and upper indices $p_{2},q_{2},$ that is, 
\begin{equation*}
r_{2}\left(x\right)=\overline{\lim_{\varepsilon\rightarrow\infty}}\frac{%
w\left(\epsilon x\right)}{w\left(\epsilon\right)}<\infty,x>0, 
\end{equation*}
and 
\begin{equation*}
p_{2}=\lim_{\epsilon\rightarrow0}\frac{\log r_{2}\left(\epsilon\right)}{%
\log\epsilon}\leq q_{2}=\lim_{\epsilon\rightarrow\infty}\frac{\log
r_{2}\left(\epsilon\right)}{\log\left(\epsilon\right)}. 
\end{equation*}
a) Let $\beta>0$ and $-\tau>\beta q_{2}.$ There is $C>0$ so that 
\begin{equation*}
\int_{x}^{\infty}t^{\tau}w\left(t\right)^{\beta}\frac{dt}{t}\leq
Cx^{\tau}w\left(x\right)^{\beta},x\in\lbrack1,\infty), 
\end{equation*}
and $\lim_{x\rightarrow\infty}x^{\tau}w\left(x\right)^{\beta}=0.$

\noindent b) Let $\beta>0$ and $-\tau<\beta p_{2}$. There is $C>0$ so that 
\begin{equation*}
\int_{1}^{x}t^{\tau}w\left(t\right)^{\beta}\frac{dt}{t}\leq
Cx^{\tau}w\left(x\right)^{\beta},x\in\lbrack1,\infty),\text{ } 
\end{equation*}
and $\lim_{x\rightarrow\infty}x^{\tau}w\left(x\right)^{\beta}=\infty.$

\noindent c) Let $\beta<0$ and $\tau<-\beta p_{2}$. There is $C>0$ so that 
\begin{equation*}
\int_{x}^{\infty}t^{\tau}w\left(t\right)^{\beta}\frac{dt}{t}\leq
Cx^{\tau}w\left(x\right)^{\beta},x\in\lbrack1,\infty), 
\end{equation*}
and $\lim_{x\rightarrow\infty}x^{\tau}w\left(x\right)^{\beta}=0.$

\noindent d) Let $\beta<0$ and $\tau>-\beta q_{2}$. There is $C>0$ so that 
\begin{equation*}
\int_{1}^{x}t^{\tau}w\left(t\right)^{\beta}\frac{dt}{t}\leq
Cx^{\tau}w\left(x\right)^{\beta},x\in\lbrack1,\infty), 
\end{equation*}
and $\lim_{x\rightarrow\infty}x^{\tau}w\left(x\right)^{\beta}=\infty.$
\end{lemma}

\begin{proof}
The claims follow easily by Theorems 3, 4 in \cite{ALAR}. Because of the
similarities, we will prove d) only. Let $\beta<0$ and $\tau>-\beta q_{2}$.
Then 
\begin{equation*}
\overline{\lim_{t\rightarrow\infty}}\frac{w\left(\varepsilon t\right)^{\beta}%
}{w\left(t\right)^{\beta}}=\overline{\lim_{t\rightarrow\infty}}\frac{%
w\left(\varepsilon^{-1}\varepsilon t\right)^{-\beta}}{w\left(\varepsilon
t\right)^{-\beta}}=r_{2}\left(\varepsilon^{-1}\right)^{-\beta}<\infty,%
\varepsilon>0. 
\end{equation*}
Hence $w\left(t\right)^{\beta},t\geq1$, is an O-RV function at infinity with 
\begin{equation*}
p=\lim_{\varepsilon\rightarrow0}\frac{\log
r_{2}\left(\varepsilon^{-1}\right)^{-\beta}}{\log\varepsilon}=\beta
q_{2}\leq\beta p_{2}=-\lim_{\varepsilon\rightarrow0}\frac{\log
r_{2}\left(\varepsilon\right)^{-\beta}}{\log\varepsilon}=q. 
\end{equation*}
Then by Theorems 3 and 4 in \cite{ALAR}, for $\tau>-\beta q_{2},$ 
\begin{equation*}
\int_{1}^{x}t^{\tau}w\left(t\right)^{\beta}\frac{dt}{t}\leq
Cx^{\tau}w\left(x\right)^{\beta},x\geq1, 
\end{equation*}
and $\lim_{x\rightarrow\infty}x^{\tau}w\left(x\right)^{\beta}=\infty$.
\end{proof}

\begin{lemma}
\label{lem:alpha12}Let $\nu\in\mathfrak{A}^{\sigma}$,$w=w_{\nu}$ be an O-RV
function at zero and infinity with indices $p_{1},q_{1},p_{2},q_{2}$ defined
in (\ref{1}), (\ref{2}). Then for any $\alpha_{1}>q_{1}\vee q_{2}$ and $%
0<\alpha_{2}<p_{1}\wedge p_{2}$, there is $C=C\left(w,\alpha_{1},\alpha_{2}%
\right)>0$ such that 
\begin{equation*}
\int_{\left\vert y\right\vert \leq1}\left\vert y\right\vert ^{\alpha_{1}}%
\tilde{\nu}_{R}\left(dy\right)+\int_{\left\vert y\right\vert >1}\left\vert
y\right\vert ^{\alpha_{2}}\tilde{\nu}_{R}\left(dy\right)\leq C,R>0. 
\end{equation*}
\end{lemma}

\begin{proof}
First,

\begin{align*}
\int_{\left\vert y\right\vert \leq1}\left\vert y\right\vert ^{\alpha_{1}}%
\tilde{\nu}_{r}\left(dy\right) &
=\delta\left(r\right)^{-1}r^{-\alpha_{1}}\int_{\left\vert y\right\vert \leq
r}\left\vert y\right\vert ^{\alpha_{1}}\nu\left(dy\right) \\
&
=\delta\left(r\right)^{-1}r^{-\alpha_{1}}\alpha_{1}\int_{0}^{r}s^{\alpha_{1}}
\left[\delta\left(s\right)-\delta\left(r\right)\right]\frac{ds}{s} \\
&
=\delta\left(r\right)^{-1}r^{-\alpha_{1}}\alpha_{1}\int_{0}^{r}\delta\left(s%
\right)s^{\alpha_{1}-1}ds-1,
\end{align*}
and similarly,

\begin{eqnarray*}
\int_{\left\vert y\right\vert >1}\left\vert y\right\vert ^{\alpha_{2}}\tilde{%
\nu}_{r}\left(dy\right) & = &
w\left(r\right)r^{-\alpha_{2}}\alpha_{2}\int_{0}^{\infty}\delta\left(s\vee
r\right)s^{\alpha_{2}}\frac{ds}{s} \\
& = &
1+w\left(r\right)r^{-\alpha_{2}}\alpha_{2}\int_{r}^{\infty}\delta\left(s%
\right)s^{\alpha_{2}}\frac{ds}{s}.
\end{eqnarray*}
Thus 
\begin{eqnarray*}
& & \int_{\left\vert y\right\vert \leq1}\left\vert y\right\vert ^{\alpha_{1}}%
\tilde{\pi}_{r}\left(dy\right)+\int_{\left\vert y\right\vert >1}\left\vert
y\right\vert ^{\alpha_{2}}\tilde{\pi}_{r}\left(dy\right) \\
& = &
w\left(r\right)r^{-\alpha_{1}}\alpha_{1}\int_{0}^{r}w\left(s\right)^{-1}s^{%
\alpha_{1}}\frac{ds}{s}+w\left(r\right)r^{-\alpha_{2}}\alpha_{2}\int_{r}^{%
\infty}w\left(s\right)^{-1}s^{\alpha_{2}}\frac{ds}{s} \\
& = & I_{1}+I_{2}
\end{eqnarray*}
By Lemma \ref{al2}, there is $C$ so that 
\begin{equation*}
I_{2}=w\left(r\right)r^{-\alpha_{2}}\alpha_{2}\int_{r}^{\infty}w\left(s%
\right)^{-1}s^{\alpha_{2}}\frac{ds}{s}\leq C,r\geq1. 
\end{equation*}
By Lemma \ref{al1}, there is $C$ so that 
\begin{equation*}
w\left(r\right)r^{-\alpha_{2}}\alpha_{2}\int_{r}^{1}w\left(s\right)^{-1}s^{%
\alpha_{2}}\frac{ds}{s}\leq C,r\in\left(0,1\right). 
\end{equation*}
Hence there is $C$ so that $I_{2}\leq C$ for all $r>0.$ Similarly, using
Lemmas \ref{al1}-\ref{al2}, we estimate $I_{1}.$
\end{proof}

\begin{remark}
\label{rem:alpha_choice}In particular, if $w=w_{\nu}$ satisfies assumption 
\textbf{A}, then we may choose in Lemma \ref{lem:alpha12} $%
\alpha_{1},\alpha_{2}$ so that

(i) $\alpha_{1,}\alpha_{2}\in\left(0,1\right)$ if $\sigma\in\left(0,1\right)$%
; (ii) $\alpha_{1},\alpha_{2}\in\left(1,2\right)$ if $\sigma\in\left(1,2%
\right)$; (iii) $\alpha_{1}\in(1,2]$ and $\alpha_{2}\in\lbrack0,1)$ if $%
\sigma=1.$
\end{remark}

\begin{lemma}
\label{al3}Let $\nu\in\mathfrak{A}^{\sigma}$,$w=w_{\nu}\left(r\right),r>0,$
be a continuous O-RV function at zero and infinity with indices $%
p_{1},q_{1},p_{2},q_{2}$ defined in (\ref{1}), (\ref{2}), and $p_{1},p_{2}>0.
$ Let 
\begin{equation*}
a\left(r\right)=\inf\left\{ t>0:w\left(t\right)\geq r\right\} ,r>0. 
\end{equation*}
Then

(i) $w\left(a\left(t\right)\right)=t,t>0,$ and 
\begin{eqnarray*}
a\left(w\left(t\right)-\right) & \leq & t\leq
a\left(w\left(t\right)+\right),t>0.
\end{eqnarray*}

(ii) $a$ is O-RV at zero and infinity with lower indices $p,\bar{p}$ and
upper indices $q,\bar{q}$ respectively so that 
\begin{eqnarray*}
\frac{1}{q_{1}} & \leq & p\leq q\leq\frac{1}{p_{1}}, \\
\frac{1}{q_{2}} & \leq & \bar{p}\leq\bar{q}\leq\frac{1}{p_{2}}.
\end{eqnarray*}
\end{lemma}

\begin{proof}
(i) follow easily from the definitions,

(ii) By Theorem 3 in \cite{ALAR}, for any $\alpha\in\left(0,p_{1}\right)$
there is $C>0$ so that 
\begin{equation*}
\frac{w\left(x\right)}{x^{\alpha}}\leq C\frac{w\left(y\right)}{y^{\alpha}}%
,0<x\leq y\leq1. 
\end{equation*}
Hence 
\begin{equation*}
\frac{a\left(y\right)}{y^{\frac{1}{\alpha}}}\leq C\frac{a\left(x\right)}{x^{%
\frac{1}{\alpha}}},0<x\leq y\leq1, 
\end{equation*}
and, by Karamata characterization (1.7) in \cite{ALAR} and Theorem 3 in \cite%
{ALAR}, $a$ is O-RV at zero with upper index, 
\begin{equation*}
q\leq\frac{1}{p_{1}}. 
\end{equation*}
Similarly we find that the lower index at zero $p\geq\frac{1}{q_{1}}$, and
determine that $a$ is O-RV at infinity with indices $\bar{p}\leq\bar{q}$ so
that 
\begin{equation*}
\frac{1}{q_{2}}\leq\bar{p}\leq\bar{q}\leq\frac{1}{p_{2}}. 
\end{equation*}
\end{proof}

The following claim is an obvious consequence of Lemmas \ref{al1}-\ref{al3}.

\begin{corollary}
\label{cor:aymp_integral}Let $\nu\in\mathfrak{A}^{\sigma}$,$%
w=w_{\nu}\left(r\right),r>0,$ be a continuous O-RV function at zero and
infinity with indices $p_{1},q_{1},p_{2},q_{2}$ defined in (\ref{1}), (\ref%
{2}), and $p_{1},p_{2}>0.$ Let 
\begin{equation*}
a\left(r\right)=\inf\left\{ t>0:w\left(t\right)\geq r\right\} ,r>0. 
\end{equation*}
(i) For any $\beta>0$ and $\tau<\frac{\beta}{q_{1}}\wedge\frac{\beta}{q_{2}}$
there is $C>0$ such that 
\begin{eqnarray*}
\int_{0}^{r}t^{-\tau}a\left(t\right)^{\beta}\frac{dt}{t} & \leq &
Cr^{-\tau}a\left(r\right)^{\beta},r>0, \\
\lim_{r\rightarrow0}r^{-\tau}a\left(r\right)^{\beta} & = &
0,\lim_{r\rightarrow\infty}r^{-\tau}a\left(r\right)^{\beta}=\infty,
\end{eqnarray*}
and for any $\beta<0,\tau>\left(-\frac{\beta}{p_{1}}\right)\vee\left(-\frac{%
\beta}{p_{2}}\right)$ there is $C>0$ such that 
\begin{eqnarray*}
\int_{0}^{r}t^{\tau}a\left(t\right)^{\beta}\frac{dt}{t} & \leq &
Cr^{\tau}a\left(r\right)^{\beta},r>0, \\
\lim_{r\rightarrow0}r^{\tau}a\left(r\right)^{\beta} & = &
0,\lim_{r\rightarrow\infty}r^{\tau}a\left(r\right)^{\beta}=\infty.
\end{eqnarray*}
(ii) For any $\gamma>0$ and $\tau>\frac{\gamma}{p_{1}}\vee\frac{\gamma}{p_{2}%
}$ there is $C>0$ such that 
\begin{eqnarray*}
\int_{r}^{\infty}t^{-\tau}a\left(t\right)^{\gamma}\frac{dt}{t} & \leq &
Cr^{-\tau}a\left(r\right)^{\gamma},r>0, \\
\lim_{r\rightarrow0}r^{-\tau}a\left(r\right)^{\gamma} & = &
\infty,\lim_{r\rightarrow\infty}r^{-\tau}a\left(r\right)^{\gamma}=0,
\end{eqnarray*}
and for any $\gamma<0$ and $\tau<\left(-\frac{\gamma}{q_{1}}%
\right)\wedge\left(-\frac{\gamma}{q_{2}}\right)$ there is $C>0$ such that 
\begin{eqnarray*}
\int_{r}^{\infty}t^{\tau}a\left(t\right)^{\gamma}\frac{dt}{t} & \leq &
Cr^{\tau}a\left(r\right)^{\gamma},r>0, \\
\lim_{r\rightarrow0}r^{\tau}a\left(r\right)^{\gamma} & = &
\infty,\lim_{r\rightarrow\infty}r^{\tau}a\left(r\right)^{\gamma}=0.
\end{eqnarray*}
\end{corollary}

\subsection{Estimates for probability density}

In this section we derive some estimates of probability density of $Z_{t}^{%
\tilde{\nu}_{R}}$ (see below for detailed description), these preliminary
estimates will be used in verifying Hörmander condition (see \cite{MikPh1})
and stochastic Hörmander condition (see \cite{KK2}) to derive apriori
estimates.

Let $\nu\in\mathfrak{A}^{\sigma}$ and $p\left(dt,dy\right)$ be a Poisson
point measure on $[0,\infty)\times\mathbf{R}_{0}^{d}$ such that $\mathbf{E}%
p\left(dt,dy\right)=\nu\left(dy\right)dt.$ Let $q(dt,dy)=p\left(dt,dy%
\right)-\nu\left(dy\right)dt$. We associate to $L^{\nu}$ the stochastic
process with independent increments 
\begin{equation}
Z_{t}=Z_{t}^{\nu}=\int_{0}^{t}\int\chi_{\sigma}(y)yq(ds,dy)+\int_{0}^{t}%
\int(1-\chi_{\sigma}(y))yp(ds,dy),t\geq0.  \label{f10}
\end{equation}
By Ito formula, 
\begin{equation}
\mathbf{E}e^{i2\pi\xi\cdot Z_{t}^{\nu}}=\exp\left\{
\psi^{\nu}\left(\xi\right)t\right\} ,t\geq0,\xi\in\mathbf{R}^{d},
\label{f12}
\end{equation}
where 
\begin{equation*}
\psi^{\nu}(\xi):=\int\left[e^{i2\pi\xi\cdot
y}-1-i2\pi\chi_{\sigma}\left(y\right)y\cdot\xi\right]\nu(dy). 
\end{equation*}

For $R>0$, let $Z_{t}^{R}=Z_{t}^{\nu,R},t>0$ be the stochastic process with
independent increments associated with $\tilde{\nu}_{R}=w_{\nu}\left(R%
\right)\nu_{R}$, i.e., 
\begin{equation*}
\mathbf{E}e^{i2\pi\xi\cdot Z_{t}^{R}}=\exp\left\{ \psi^{\tilde{\nu}%
_{R}}\left(\xi\right)t\right\} 
\end{equation*}
with 
\begin{equation*}
\psi^{\tilde{\nu}_{R}}\left(\xi\right)=\int\left[e^{i2\pi\xi\cdot
y}-1-i2\pi\chi_{\sigma}\left(y\right)y\cdot\xi\right]d\tilde{\nu}_{R},\xi\in%
\mathbf{R}^{d}. 
\end{equation*}
Note $Z_{t}^{R}=Z_{t}^{\nu,R}$ and $R^{-1}Z_{\omega_{\nu}\left(R\right)t}^{%
\nu},t>0,$ have the same distribution. For $R>0$, consider Levy measures $%
\nu^{R,0}\left(dy\right)=\chi_{\left\{ \left\vert y\right\vert \leq1\right\}
}\tilde{\nu}_{R}\left(dy\right)$, i.e., 
\begin{equation*}
\int\chi_{\Gamma}\left(y\right)\nu^{R,0}\left(dy\right)=\int_{\left\vert
y\right\vert \leq1}\chi_{\Gamma}\left(y\right)\tilde{\nu}_{R}\left(dy%
\right)=w_{\nu}\left(R\right)\int_{\left\vert y\right\vert \leq
R}\chi_{\Gamma}\left(y/R\right)\nu\left(dy\right),\Gamma\in\mathcal{B}%
_{0}\left(\mathbf{R}^{d}\right). 
\end{equation*}
Let $\eta=\eta_{0}^{R}$ be a random variable with characteristic function $%
\exp\left\{ \psi^{R,0}\left(\xi\right)\right\} $ given below, denote $\hat{%
\xi}=\xi/\left\vert \xi\right\vert ,\xi\in\mathbf{R}_{0}^{d}$, 
\begin{eqnarray}
\psi^{R,0}\left(\xi\right) & = & \int_{\left\vert y\right\vert \leq1}\left[%
e^{i2\pi\xi\cdot y}-1-i\chi_{\alpha}\left(y\right)\xi\cdot y\right]\tilde{\nu%
}_{R}\left(dy\right)  \notag \\
& = & \int_{\left\vert y\right\vert \leq1}\left[e^{i2\pi\left\vert
\xi\right\vert \hat{\xi}\cdot y}-1-i\chi_{\alpha}\left(y\right)\hat{\xi}%
\cdot y\left\vert \xi\right\vert \right]\tilde{\nu}_{R}\left(dy\right) 
\notag \\
& = & \frac{w\left(R\right)}{w\left(R\left\vert \xi\right\vert ^{-1}\right)}%
\int_{\left\vert y\right\vert \leq\left\vert \xi\right\vert }\left[e^{i2\pi%
\hat{\xi}\cdot y}-1-i\chi_{\alpha}\left(y\right)\hat{\xi}\cdot y\right]%
\tilde{\nu}_{R\left\vert \xi\right\vert ^{-1}}\left(dy\right).
\label{eq:cutSymb}
\end{eqnarray}

\begin{lemma}
\label{lem:symbol}Let $\nu\in\mathfrak{A}^{\sigma},w=w_{\nu}$ be an O-RV
function and \textbf{A, B} hold. Then

(i) 
\begin{equation}
\func{Re}\psi^{R,0}\left(\xi\right)\leq-c\left\vert \xi\right\vert
^{\kappa},\left\vert \xi\right\vert \geq1,  \label{eq:eq:mainExp}
\end{equation}
with some $c,\kappa>0$ independent of $R$.

(ii) $\eta=\eta_{0}^{R}$ has a pdf $p_{R,0}\left(x\right),x\in\mathbf{R}^{d},
$ such that for any multiindex $\beta\in\mathbf{N}_{0}^{d}$, and a positive
integer $n\geq0$, then there exists $C=C\left(\beta,\nu\right)>0$ such that 
\begin{equation*}
\sup_{x}\left\vert \partial^{\beta}p_{R,0}\left(x\right)\right\vert
+\int(1+\left\vert x\right\vert ^{2})^{n}\left\vert
\partial^{\beta}p_{R,0}\left(x\right)\right\vert dx\leq C,R>0. 
\end{equation*}
\end{lemma}

\begin{proof}
(i) Let $\alpha_{2}\in\left(0,p_{1}^{\omega_{\nu}}\wedge
p_{2}^{\omega_{\nu}}\right)$. By Lemma \ref{lem:powerEstratio}, there is $C>0
$ so that 
\begin{equation*}
\frac{w\left(R\right)}{w\left(R\left\vert \xi\right\vert ^{-1}\right)}\geq
C\left\vert \xi\right\vert ^{\alpha_{2}},R>0,\left\vert \xi\right\vert
\geq1. 
\end{equation*}
Hence, according to (\ref{eq:cutSymb}), for $\left\vert \xi\right\vert \geq1,%
\hat{\xi}=\xi/\left\vert \xi\right\vert $, 
\begin{align*}
\func{Re}\psi^{R,0}\left(\xi\right) & \leq c\left\vert \xi\right\vert
^{\alpha_{2}}\int_{\left\vert y\right\vert \leq1/8}\left[\cos\left(2\pi\hat{%
\xi}\cdot y\right)-1\right]\tilde{\nu}_{R\left\vert \xi\right\vert
^{-1}}\left(dy\right) \\
& \leq-c\left\vert \xi\right\vert
^{\alpha_{2}}\inf_{R\in\left(0,\infty\right),\left\vert \hat{\xi}\right\vert
=1}\int_{\left\vert y\right\vert \leq1}\left\vert \hat{\xi}\cdot
y\right\vert ^{2}\tilde{\nu}_{R}\left(dy\right)=-c_{0}\left\vert
\xi\right\vert ^{\alpha_{2}},
\end{align*}
and 
\begin{equation}
\int\exp\left\{ \func{Re}\psi^{R,0}\left(\xi\right)\right\} d\xi<\infty.
\label{3}
\end{equation}

(ii) Since (\ref{3}) holds, by Proposition I.2.5 in \cite{sato}, $%
\eta=\eta_{0}^{R}$ has a continuous bounded density 
\begin{equation}
p_{R,0}\left(x\right)=\int e^{-i2\pi x\cdot\xi}\exp\left\{
\psi^{R,0}\left(\xi\right)\right\} d\xi,x\in\mathbf{R}^{d}.  \label{d2-1}
\end{equation}
Moreover, by (i), for any multiindex $\beta\in\mathbf{N}_{0}^{d},$ 
\begin{equation*}
\partial^{\beta}p_{R,0}\left(x\right)=\int e^{-i2\pi
x\cdot\xi}\left(-i2\pi\xi\right)^{\beta}\exp\left\{
\psi^{R,0}\left(\xi\right)\right\} d\xi,x\in\mathbf{R}^{d}, 
\end{equation*}
is a bounded continuous function. The function $\left(1+\left\vert
x\right\vert ^{2}\right)^{n}$ $\partial^{\beta}p_{R,0}$ is integrable if 
\begin{eqnarray}
& & (-i2\pi x_{j})^{l}\left(-i2\pi
x_{k}\right)^{2n}\partial^{\beta}p_{R,0}\left(x\right)  \label{d3-1} \\
& = & \int\partial_{\xi_{j}}^{l}\partial_{\xi_{k}}^{2n}[e^{-i2\pi
x\cdot\xi}]\left(-i2\pi\xi\right)^{\beta}\exp\left\{
\psi^{R,0}\left(\xi\right)\right\} d\xi  \notag \\
& = & \left(-1\right)^{l+2n}\int e^{-i2\pi
x\cdot\xi}\partial_{\xi_{j}}^{l}\partial_{\xi_{k}}^{2n}[\left(-i2\pi\xi%
\right)^{\beta}\exp\left\{ \psi^{R,0}\left(\xi\right)\right\} ]d\xi  \notag
\end{eqnarray}
is bounded for all $j,k,$ and for $l\leq d+1$. Since $\partial^{\mu}%
\psi^{R,0}\left(\xi\right)$ is bounded for $\left\vert \mu\right\vert \geq2$
and 
\begin{equation*}
|\nabla\psi^{R,0}\left(\xi\right)|\leq
C\left(1+\zeta\left(\xi\right)\right),\xi\in\mathbf{R}^{d}, 
\end{equation*}
with $\zeta\left(\xi\right)=\int_{\left\vert y\right\vert
\leq1}\chi_{\sigma}\left(y\right)\left\vert y\right\vert \left[%
\left(\left\vert \xi\right\vert \left\vert y\right\vert \right)\wedge1\right]%
\tilde{\nu}_{R}\left(dy\right)$. The boundedness follows from (i) and Lemma %
\ref{lem:alpha12} (see Remark \ref{rem:alpha_choice}).
\end{proof}

\begin{lemma}
\label{Lem: density_one}Let $\nu\in\mathfrak{A}^{\sigma},w=w_{\nu}$ is an
O-RV function and \textbf{A, B} hold.

Then for $R>0$, $Z_{1}^{R}$ has a bounded continuous probability density of
the form 
\begin{equation*}
p^{R}\left(1,x\right)=\int p_{R,0}\left(x-y\right)P_{R}\left(dy\right),x\in%
\mathbf{R}^{d}, 
\end{equation*}
where $P_{R}$ is a probability distribution. Moreover, for any $%
0<\alpha_{2}<p_{1}^{w_{\nu}}\wedge p_{2}^{w_{\nu}}$ and multiindex $\beta\in%
\mathbf{N}_{0}^{d}$, there is $C=C\left(\beta,\alpha_{2},\nu\right)>0$ such
that 
\begin{equation}
\sup_{x}\left|\partial^{\beta}p^{R}\left(1,x\right)\right|+\int\left(1+\left%
\vert x\right\vert ^{\alpha_{2}}\right)\left\vert
D^{\beta}p^{R}\left(1,x\right)\right\vert dx\leq C.  \label{eq:densityEst}
\end{equation}
\end{lemma}

\begin{proof}
We have 
\begin{equation}
\tilde{\nu}_{R}=\tilde{\nu}_{R,0}+\tilde{\nu}_{R,2},  \label{d50}
\end{equation}
where $\tilde{\nu}_{R,0}\left(dy\right)=\chi_{\left\vert y\right\vert \leq1}%
\tilde{\nu}_{R}\left(dy\right)$ and hence, 
\begin{equation*}
\psi^{\tilde{\nu}_{R}}\left(\xi\right)=\psi^{\tilde{\nu}_{R,0}}\left(\xi%
\right)+\psi^{\tilde{\nu}_{R,2}}\left(\xi\right),\xi\in\mathbf{R}^{d}. 
\end{equation*}

Denote $\eta_{0}^{R}$ and $\eta_{2}^{R}$ to be independent random variables
with characteristic exponent $\exp\left\{ \psi^{\tilde{\nu}%
_{R,0}}\left(\xi\right)\right\} $ and $\exp\left\{ \psi^{\tilde{\nu}%
_{R,2}}\left(\xi\right)\right\} $ respectively. Obviously the distribution
of $Z_{1}^{R}$ coincides with the distribution of the sum $%
\eta_{0}^{R}+\eta_{2}^{R}$. Therefore, 
\begin{equation}
p^{R}\left(1,x\right)=\int p_{R,0}\left(x-y\right)P_{R}\left(dy\right)
\label{eq:compositionDensity}
\end{equation}
where $P_{R}\left(dy\right)$ is the probability distribution of $\eta_{2}^{R}
$. It is clear from (\ref{eq:compositionDensity}) that for any multiindex $%
\beta\in\mathbf{N}_{0}^{d}$, 
\begin{equation*}
\partial^{\beta}p^{R}\left(1,x\right)=\int\partial^{\beta}p_{R,0}\left(x-y%
\right)P_{R}\left(dy\right) 
\end{equation*}
and from Lemma \ref{lem:symbol}, 
\begin{equation*}
\int\left\vert \partial^{\beta}p^{R}\left(1,x\right)\right\vert dx\leq C, 
\end{equation*}
and 
\begin{equation*}
\sup_{x}\left\vert \partial^{\beta}p^{R}\left(1,x\right)\right\vert \leq C 
\end{equation*}

Moreover, using Lemma \ref{lem:alpha12}, and following the argument of Lemma
10 of \cite{MikPh2}, one derives easily the estimate (\ref{eq:densityEst}).
\end{proof}

We write $\pi\in\mathfrak{A}_{sign}^{\sigma}=\mathfrak{A}^{\sigma}-\mathfrak{%
A}^{\sigma}$ if $\pi=\mu-\eta$ with $\mu,\eta\in\mathfrak{A}^{\sigma}$, and $%
L^{\pi}=L^{\mu}-L^{\eta}$. Given $\pi\in\mathfrak{A}_{sign}^{\sigma}$, we
denote $\left\vert \pi\right\vert $ its variation measure. Obviously, $%
\left\vert \pi\right\vert \in\mathfrak{A}^{\sigma}.$ Based on Lemma \ref%
{Lem: density_one}, we derive the following estimates(Lemma \ref%
{lem:density_with_L} and Lemma \ref{lem:main_density}) similar to the ones
in Lemma 12, and Lemma 13 of \cite{MikPh2}. Since the proof carry over to
our setting in obvious ways we omit the proof.

\begin{lemma}
(see Corollary 4 of \cite{MikPh2})\label{lem:density_with_L}Let $\nu\in%
\mathfrak{A}^{\sigma},w=w_{\nu}$ be an O-RV function and \textbf{A, B} hold.
Let 
\begin{equation*}
\alpha_{1}>q_{1}^{w}\vee q_{2}^{w},0<\alpha_{2}<p_{1}^{w}\wedge p_{2}^{w}, 
\end{equation*}
and 
\begin{eqnarray*}
\alpha_{2} & > & 1\text{ if }\sigma\in(1,2), \\
\alpha_{1} & \leq & 1\text{ if }\sigma\in(0,1),\alpha_{1}\leq2\text{ if }%
\sigma\in\lbrack1,2).
\end{eqnarray*}
Let $\pi\in\mathfrak{A}_{sign}^{\sigma}$ and assume that 
\begin{equation*}
\int_{\left\vert y\right\vert \leq1}\left\vert y\right\vert ^{\alpha_{1}}d%
\widetilde{\left\vert \pi\right\vert }_{R}+\int_{\left\vert y\right\vert
>1}\left\vert y\right\vert ^{\alpha_{2}}d\widetilde{\left\vert
\pi\right\vert }_{R}\leq M,R>0. 
\end{equation*}
Then for any multiindex $k\in\mathbf{N}_{0}^{d}$ there is $%
C=C\left(k,\nu\right)>0$ such that

\begin{eqnarray*}
\int\left(1+\left\vert x\right\vert ^{\alpha_{2}}\right)\left\vert D^{k}L^{%
\tilde{\pi}_{R}}p^{R}\left(1,x\right)\right\vert dx & \leq & CM,
\end{eqnarray*}
and there is $C=C\left(\nu\right)>0$ such that 
\begin{equation*}
\int(1+\left\vert x\right\vert ^{\alpha_{2}})\left\vert L^{\tilde{\pi}%
_{R}}L^{\tilde{\nu}_{R}}p^{R}\left(1,x\right)\right\vert dx\leq CM. 
\end{equation*}
\end{lemma}

In the next lemma, we denote $p^{\nu}\left(t,\cdot\right)$ the probability
density of the process $Z_{t}^{\nu}$, $t>0$.

\begin{lemma}
\label{lem:main_density}Let $\nu\in\mathfrak{A}^{\sigma},w=w_{\nu}$ be a
continuous O-RV function and \textbf{A, B} hold. Let 
\begin{equation*}
\alpha_{1}>q_{1}^{w}\vee q_{2}^{w},0<\alpha_{2}<p_{1}^{w}\wedge p_{2}^{w}, 
\end{equation*}
and 
\begin{eqnarray*}
\alpha_{2} & > & 1\text{ if }\sigma\in(1,2), \\
\alpha_{1} & \leq & 1\text{ if }\sigma\in(0,1),\alpha_{1}\leq2\text{ if }%
\sigma\in\lbrack1,2).
\end{eqnarray*}
Let $\pi\in\mathfrak{A}_{sign}^{\sigma}$ and assume that 
\begin{equation*}
\int_{\left\vert y\right\vert \leq1}\left\vert y\right\vert ^{\alpha_{1}}d%
\widetilde{\left\vert \pi\right\vert }_{R}+\int_{\left\vert y\right\vert
>1}\left\vert y\right\vert ^{\alpha_{2}}d\widetilde{\left\vert
\pi\right\vert }_{R}\leq M,R>0. 
\end{equation*}
Denote $a\left(t\right)=\inf\left\{ r:w_{\nu}\left(r\right)\geq t\right\}
,t>0.$

\noindent (i) For any multiindex $k\in\mathbf{N}_{0}^{d}$ and $\beta\in\left[%
0,\alpha_{2}\right]$, there is $C=C\left(k,\nu,\beta\right)>0$ such that 
\begin{eqnarray*}
\int_{\left\vert z\right\vert >c}\left\vert
L^{\pi}D^{k}p^{\nu}\left(t,z\right)\right\vert dz & \leq &
CMt^{-1}a\left(t\right)^{\beta-\left\vert k\right\vert }c^{-\beta}, \\
\int\left\vert L^{\pi}D^{k}p^{\nu}\left(t,z\right)\right\vert dz & \leq &
CMt^{-1}a\left(t\right)^{-\left\vert k\right\vert },
\end{eqnarray*}
(ii) There is $C=C\left(\nu\right)>0$ such that 
\begin{equation*}
\int_{\mathbf{R}^{d}}\left\vert
L^{\pi}p^{\nu}\left(t,x-y\right)-L^{\pi}p^{\nu}\left(t,x\right)\right\vert
dx\leq CM\frac{\left\vert y\right\vert }{ta\left(t\right)},t>0,y\in\mathbf{R}%
^{d}, 
\end{equation*}
(iii) There is $C=C\left(\nu\right)>0$ such that 
\begin{align*}
& \int_{2b}^{\infty}\int\left\vert
L^{\pi}p^{\nu}\left(t-s,x\right)-L^{\pi}p^{\nu}\left(t,x\right)\right\vert
dxdt \\
& \leq CM,\left\vert s\right\vert \leq b<\infty.
\end{align*}
\end{lemma}

\subsection{Representation of fractional operator and some density estimates}

In order to deal with the stochastic term we need some finer estimates of
probability density. The fractional operator is defined in the following
way. Let $\mu\in\mathfrak{A}_{sym}^{\sigma}=\left\{ \eta\in\mathfrak{A}%
^{\sigma}:\eta\text{ is symmetric, }\eta=\eta_{sym}\right\} .$ Then for $%
\delta\in(0,1)$ and $f\in\mathcal{S}\left(\mathbf{R}^{d}\right)$, we have 
\begin{eqnarray*}
& & -\left(-\psi^{\mu}\left(\xi\right)\right)^{\delta}\hat{f}\left(\xi\right)
\\
& = & c_{\delta}\int_{0}^{\infty}t^{-\delta}\left[\exp\left(\psi^{\mu}\left(%
\xi\right)t\right)-1\right]\frac{dt}{t}\hat{f}\left(\xi\right),\xi\in\mathbf{%
R}^{d},
\end{eqnarray*}
and define 
\begin{eqnarray}
L^{\mu;\delta}f\left(x\right) & := & \mathcal{F}^{-1}\left[%
-\left(-\psi^{\mu}\right)^{\delta}\hat{f}\right]\left(x\right)  \label{ff01}
\\
& = & c_{\delta}\mathbf{E}\int_{0}^{\infty}t^{-\delta}\left[%
f\left(x+Z_{t}^{\mu}\right)-f\left(x\right)\right]\frac{dt}{t},x\in\mathbf{R}%
^{d}.  \notag
\end{eqnarray}
For $\mu\in\mathfrak{A}^{\sigma},\delta\in\left(0,1\right)$ we let $\mu_{%
\text{sym}}=\frac{\mu\left(dy\right)+\mu\left(-dy\right)}{2}$ and define $%
L^{\mu;\delta}=-\left(-L^{\mu_{\text{sym}}}\right)^{\delta}$.

\begin{lemma}
\label{lem:fracdensity}Let $\mu\in\mathfrak{A}_{sym}^{\sigma},\delta\in%
\left(0,1\right).$ Let $\nu\in\mathfrak{A}^{\sigma},w=w_{\nu}$ be a
continuous O-RV function and \textbf{A, B} hold. Assume 
\begin{equation*}
\alpha_{1}>q_{1}^{w}\vee q_{2}^{w},0<\alpha_{2}<p_{1}^{w}\wedge p_{2}^{w}, 
\end{equation*}
\begin{eqnarray*}
\alpha_{2} & > & 1\text{ if }\sigma\in(1,2), \\
\alpha_{1} & \leq & 1\text{ if }\sigma\in(0,1),\alpha_{1}\leq2\text{ if }%
\sigma\in\lbrack1,2),
\end{eqnarray*}
and 
\begin{equation*}
\int_{\left\vert y\right\vert \leq1}\left\vert y\right\vert ^{\alpha_{1}}d%
\widetilde{\mu}_{R}+\int_{\left\vert y\right\vert >1}\left\vert y\right\vert
^{\alpha_{2}}d\widetilde{\mu}_{R}\leq M,R>0. 
\end{equation*}
Let $p^{\nu}\left(t,x\right),x\in\mathbf{R}^{d}$, be the pdf of $%
Z_{t}^{\nu},t>0$ and $a\left(t\right)=\inf\left\{ r:w\left(r\right)\geq
t\right\} ,t>0.$

\noindent (i) For any $p\geq1,$ and $\varepsilon>0$ there is $C_{\epsilon}>0$
such that 
\begin{equation*}
\left\vert L^{\mu;\delta}f\right\vert _{L_{p}\left(\mathbf{R}%
^{d}\right)}\leq\varepsilon\left\vert L^{\mu}f\right\vert _{L_{p}\left(%
\mathbf{R}^{d}\right)}+C_{\epsilon}\left\vert f\right\vert _{L_{p}\left(%
\mathbf{R}^{d}\right)},f\in\mathcal{S}\left(\mathbf{R}^{d}\right). 
\end{equation*}
(ii) There is $C=C\left(\nu\right)>0$ such that 
\begin{equation*}
\int_{\mathbf{R}^{d}}\left\vert
L^{\mu;\delta}p^{\nu}\left(t,x-y\right)-L^{\mu;\delta}p^{\nu}\left(t,x%
\right)\right\vert dx\leq CM\frac{\left\vert y\right\vert }{%
t^{\delta}a\left(t\right)},t>0,y\in\mathbf{R}^{d}, 
\end{equation*}
(iii) There is $C=C\left(\nu\right)>0$ such that 
\begin{align}
& \int_{2b}^{\infty}\left(\int\left\vert L^{\mu;\frac{1}{2}%
}p^{\nu}\left(t-s,x\right)-L^{\mu;\frac{1}{2}}p^{\nu}\left(t,x\right)\right%
\vert dx\right)^{2}dt  \label{eq:MVTtime-1} \\
& \leq CM,\left\vert s\right\vert \leq b<\infty.  \notag
\end{align}
(iv) For any multiindex $k\in\mathbf{N}_{0}^{d}$ and $\beta\in\lbrack0,%
\delta\alpha_{2})$ there is $C=C\left(k,\nu,\beta\right)>0$ such that 
\begin{eqnarray*}
\int_{\left\vert x\right\vert >c}\left\vert
L^{\mu;\delta}D^{k}p^{\nu}\left(t,x\right)\right\vert dx & \leq &
CMt^{-\delta}a\left(t\right)^{\beta-\left\vert k\right\vert }c^{-\beta}, \\
\int\left\vert L^{\mu;\delta}D^{k}p^{\nu}\left(t,x\right)\right\vert dx &
\leq & CMt^{-\delta}a\left(t\right)^{-\left\vert k\right\vert }.
\end{eqnarray*}
\end{lemma}

For the proof of Lemma \ref{lem:fracdensity}, we will need the following
Lemmas \ref{al1-1}, \ref{lem:le0}.

\begin{lemma}
\label{al1-1}Let $\delta\in\left(0,1\right),$ $\mu\in\mathfrak{A}%
_{sym}^{\sigma}$ and $\pi\in\mathfrak{A}_{sign}^{\sigma}.$ Let $\nu\in%
\mathfrak{A}^{\sigma},w=w_{\nu}$ be a continuous O-RV function and \textbf{%
A, B} hold. Then 
\begin{eqnarray*}
p^{\nu}\left(t,x\right) & = & a\left(t\right)^{-d}p^{\tilde{\nu}%
_{a\left(t\right)}}\left(1,xa\left(t\right)^{-1}\right),x\in\mathbf{R}%
^{d},t>0, \\
L^{\mu;\delta}p^{\nu}\left(t,x\right) & = & \frac{1}{t^{\delta}}%
a\left(t\right)^{-d}(L^{\tilde{\mu}_{a\left(t\right)};\delta}p^{\tilde{\nu}%
_{a\left(t\right)}})\left(1,xa\left(t\right)^{-1}\right),x\in\mathbf{R}%
^{d},t>0, \\
L^{\pi}L^{\mu;\delta}p^{\nu}\left(t,x\right) & = & \frac{1}{t^{1+\delta}}%
a\left(t\right)^{-d}(L^{\tilde{\pi}_{a\left(t\right)}}L^{\tilde{\mu}%
_{a\left(t\right)};\delta}p^{\tilde{\nu}_{a\left(t\right)}})\left(1,xa%
\left(t\right)^{-1}\right),x\in\mathbf{R}^{d},t>0,
\end{eqnarray*}
where $a\left(t\right)=\inf\left\{ r\geq0:\omega\left(r\right)\geq t\right\}
,t>0.$
\end{lemma}

\begin{proof}
Indeed, by Lemma \ref{Lem: density_one}, for each $t>0$ and $r>0$, the
density $p^{\tilde{\nu}_{a\left(t\right)}}\left(r,x\right),x\in\mathbf{R}%
^{d},$ is infinitely differentiable in $x$, all derivatives are bounded and
integrable. Obviously, 
\begin{equation*}
\exp\left\{ \psi^{\nu}\left(\xi\right)t\right\} =\exp\left\{ \psi^{\tilde{\nu%
}_{a\left(t\right)}}\left(a\left(t\right)\xi\right)\right\} ,t>0,\xi\in%
\mathbf{R}^{d}, 
\end{equation*}
and 
\begin{eqnarray*}
& & \left(-\psi^{\mu}\left(\xi\right)\right)^{\delta}\exp\left\{
\psi^{\nu}\left(\xi\right)t\right\} \\
& = & \frac{1}{t^{\delta}}\left(-\psi^{\tilde{\mu}_{a\left(t\right)}}\left(a%
\left(t\right)\xi\right)\right)^{\delta}\exp\left\{ \psi^{\tilde{\nu}%
_{a\left(t\right)}}\left(a\left(t\right)\xi\right)\right\} ,t>0,\xi\in%
\mathbf{R}^{d}.
\end{eqnarray*}
We derive the first two equalities by taking Fourier inverse. Similarly, the
third equality can be derived.
\end{proof}

\begin{lemma}
\label{lem:le0}Let $\mu$, $\nu\in\mathfrak{A}^{\sigma}$ satisfy assumptions
of Lemma \ref{lem:fracdensity}, $\delta\in\left(0,1\right).$

a) For any $p\geq1,$ and $\varepsilon>0$ there is $C_{\epsilon}>0$ such that 
\begin{equation*}
\left\vert L^{\mu;\delta}f\right\vert _{L_{p}\left(\mathbf{R}%
^{d}\right)}\leq\varepsilon\left\vert L^{\mu}f\right\vert _{L_{p}\left(%
\mathbf{R}^{d}\right)}+C_{\epsilon}\left\vert f\right\vert _{L_{p}\left(%
\mathbf{R}^{d}\right)},f\in\mathcal{S}\left(\mathbf{R}^{d}\right). 
\end{equation*}

b) Let $p^{R}\left(t,x\right)=p^{\tilde{\nu}_{R}}\left(t,x\right),x\in%
\mathbf{R}^{d}$, be the pdf of $Z_{t}^{\tilde{\nu}_{R}},t>0,R>0.$ Then for
each multiindex $k\in\mathbf{N}_{0}^{d}$, and $\beta\in\lbrack0,\delta%
\alpha_{2})$, there is $C=C\left(k,\nu,\beta\right)>0$ such that, 
\begin{eqnarray*}
\int\left(1+\left\vert x\right\vert ^{\beta}\right)\left\vert D^{k}L^{\tilde{%
\mu}_{R};\delta}p^{R}\left(1,x\right)\right\vert dx & \leq & CM, \\
\int(1+\left\vert x\right\vert ^{\beta})\left\vert L^{\tilde{\mu}%
_{R};\delta}L^{\tilde{\nu}_{R}}p^{R}\left(1,x\right)\right\vert dx & \leq &
CM.
\end{eqnarray*}
\end{lemma}

\begin{proof}
Indeed for any $a>0,f\in\mathcal{S}\left(\mathbf{R}^{d}\right),x\in\mathbf{R}%
^{d},$ by Ito formula and (\ref{ff01}), 
\begin{eqnarray}
L^{\mu;\delta}f\left(x\right) & = & c\mathbf{E}\int_{0}^{a}t^{-\delta}%
\int_{0}^{t}L^{\mu}f\left(x+Z_{r}^{\mu}\right)dr\frac{dt}{t}  \label{f00} \\
& & +c\mathbf{E}\int_{a}^{\infty}t^{-\delta}\left[f\left(x+Z_{t}^{\mu}%
\right)-f\left(x\right)\right]\frac{dt}{t}.  \notag
\end{eqnarray}
The statement a) follows by Minkowski inequality.

By (\ref{f00}), for any multiindex $k$, 
\begin{eqnarray*}
& & \int\left(1+\left\vert x\right\vert ^{\beta}\right)\left\vert L^{\tilde{%
\mu}_{R};\delta}D^{k}p^{R}\left(1,x\right)\right\vert dx \\
& \leq & C\mathbf{E}\int_{0}^{1}t^{-\delta}\int_{0}^{t}\int\left(1+\left%
\vert x\right\vert ^{\beta}\right)\left\vert L^{\tilde{\mu}%
_{R}}D^{k}p^{R}\left(1,x+Z_{r}^{\mu}\right)\right\vert dxdr\frac{dt}{t} \\
& & +C\mathbf{E}\int_{1}^{\infty}t^{-\delta}\int\left(1+\left\vert
x\right\vert ^{\beta}\right)[\left\vert
D^{k}p^{R}\left(1,x+Z_{t}^{\mu}\right)\right\vert +\left\vert
D^{k}p^{R}\left(1,x\right)\right\vert ]dx\frac{dt}{t} \\
& = & A_{1}+A_{2}.
\end{eqnarray*}
Now, by Lemma \ref{lem:density_with_L}, and Lemma 17 in \cite{MikPh2}, 
\begin{eqnarray*}
A_{1} & \leq & C\int_{0}^{1}t^{-\delta}\int\left(1+\left\vert x\right\vert
^{\beta}\right)\left\vert L^{\tilde{\mu}_{R}}D^{k}p^{R}\left(1,x\right)%
\right\vert dxdt \\
& & +C\int_{0}^{1}t^{-\delta}\int_{0}^{t}\mathbf{E}\left(\left\vert
Z_{r}^{\mu}\right\vert ^{\beta}\right)dr\int\left\vert L^{\tilde{\mu}%
_{R}}D^{k}p^{R}\left(1,x\right)\right\vert dx\frac{dt}{t} \\
& \leq & C,
\end{eqnarray*}
and 
\begin{eqnarray*}
A_{2} & \leq & C\int_{1}^{\infty}t^{-\delta}\int\left(1+\left\vert
x\right\vert ^{\beta}\right)\left\vert D^{k}p^{R}\left(1,x\right)\right\vert
]dx\frac{dt}{t} \\
& & +C\int_{1}^{\infty}t^{-\delta}\mathbf{E}\left(\left\vert
Z_{t}^{\mu}\right\vert ^{\beta}\right)\int\left\vert
D^{k}p^{R}\left(1,x\right)\right\vert dx\frac{dt}{t} \\
& \leq & C\left(1+\int_{1}^{\infty}t^{-\delta}t^{\frac{\beta}{\alpha_{2}}}%
\frac{dt}{t}\right)\leq C.
\end{eqnarray*}
Similarly, the second inequality of part b) is proved.
\end{proof}

\subsubsection{Proof of Lemma \protect\ref{lem:fracdensity}}

(i) is proved in Lemma \ref{lem:le0} (a).

(ii) By Lemma \ref{al1-1} and Lemma \ref{lem:le0}, 
\begin{eqnarray*}
& & \int_{\mathbf{R}^{d}}\left\vert
L^{\mu;\delta}p^{\nu}\left(t,x-y\right)-L^{\mu;\delta}p^{\nu}\left(t,x%
\right)\right\vert dx \\
& = & \frac{1}{t^{\delta}}\int\left\vert L^{\tilde{\mu}_{a\left(t\right)};%
\delta}p^{\tilde{\nu}_{a\left(t\right)}}\left(1,x-\frac{y}{a\left(t\right)}%
\right)-L^{\tilde{\mu}_{a\left(t\right)};\delta}p^{\tilde{\nu}%
_{a\left(t\right)}}\left(1,x\right)\right\vert dx \\
& \leq & \frac{1}{t^{\delta}}\int_{0}^{1}\int\left\vert \nabla L^{\tilde{\mu}%
_{a\left(t\right)};\delta}p^{\tilde{\nu}_{a\left(t\right)}}\left(1,x-s\frac{y%
}{a\left(t\right)}\right)\right\vert \frac{\left\vert y\right\vert }{%
a\left(t\right)}dxds \\
& \leq & C\frac{\left\vert y\right\vert }{t^{\delta}a\left(t\right)}%
\int\left\vert L^{\tilde{\mu}_{a\left(t\right)};\delta}\nabla p^{\tilde{\nu}%
_{a(t)}}\left(1,x\right)\right\vert dx\leq CM\frac{\left\vert y\right\vert }{%
t^{\delta}a\left(t\right)}.
\end{eqnarray*}

(iii) By Lemma \ref{al1-1} and Lemma \ref{lem:le0}, 
\begin{eqnarray*}
& & \int_{2a}^{\infty}\left(\int\left\vert L^{\mu;\frac{1}{2}%
}p^{\nu}\left(t-s,x\right)-L^{\mu;\frac{1}{2}}p^{\nu}\left(t,x\right)\right%
\vert dx\right)^{2}dt \\
& \leq & \left\vert s\right\vert
^{2}\int_{2a}^{\infty}\left(\int_{0}^{1}\int\left\vert L^{\mu;\frac{1}{2}%
}L^{\nu}p^{\nu}\left(t-rs,x\right)\right\vert dxdr\right)^{2}dt \\
& \leq & C\left\vert s\right\vert ^{2}\int_{2a}^{\infty}\left(\int_{0}^{1}%
\frac{dr}{\left(t-rs\right)^{1+\frac{1}{2}}}\right)^{2}dt\leq
C\int_{2a}^{\infty}\left\vert \frac{1}{\left(t-s\right)^{\frac{1}{2}}}-\frac{%
1}{t^{\frac{1}{2}}}\right\vert ^{2}dt \\
& \leq & C\left\vert s\right\vert \int_{2a}^{\infty}\frac{dt}{%
\left(t-s\right)t}=C\int_{2a}^{\infty}\frac{1}{\left(\frac{t}{s}-1\right)}%
\frac{dt}{t}\leq C.
\end{eqnarray*}

(iv) Indeed, by Lemma \ref{al1-1}, Chebyshev inequality, and Lemma \ref%
{lem:le0}, 
\begin{eqnarray*}
& & \int_{\left\vert x\right\vert >c}\left\vert
L^{\mu;\delta}D^{k}p^{\nu}\left(t,x\right)\right\vert dx \\
& = & \frac{1}{t^{\delta}}a\left(t\right)^{-d-k}\int_{\left\vert
x\right\vert >c}\left\vert L^{\tilde{\mu}_{a\left(t\right)};\delta}D^{k}p^{%
\tilde{\nu}_{a\left(t\right)}}\left(1,\frac{x}{a\left(t\right)}%
\right)\right\vert dx \\
& \leq & \frac{a\left(t\right)^{\beta-k}c^{-\beta}}{t^{\delta}}%
\int\left\vert x\right\vert ^{\beta}\left\vert L^{\tilde{\mu}%
_{a\left(t\right)};\delta}D^{k}p^{\tilde{\nu}_{a\left(t\right)}}\left(1,x%
\right)\right\vert dx\leq CM\frac{a\left(t\right)^{\beta-k}c^{-\beta}}{%
t^{\delta}}.
\end{eqnarray*}
Similarly, we derive the second estimate of the claim.

\section{Equivalent norms of function spaces}

In this section, we discuss a characterization of Bessel potential spaces
and Besov spaces.

\begin{lemma}
\label{Lem: symbol_omega}Let $w$ be a non-decreasing O-RV function at zero
and infinity, with $p_{1}^{w},p_{2}^{w}>0$. Let $\pi\in\mathfrak{A}^{\sigma}$
and define $\tilde{\pi}_{R}\left(dy\right)=w\left(R\right)\pi\left(Rdy\right)
$, $R>0.$

\noindent a) Assume there is $N_{2}>0$ so that 
\begin{eqnarray}
\int\left(\left\vert y\right\vert \wedge1\right)\tilde{\pi}%
_{R}\left(dy\right) & \leq & N_{2}\text{ if }\sigma\in(0,1),  \label{fa0} \\
\int\left(\left\vert y\right\vert ^{2}\wedge1\right)\tilde{\pi}%
_{R}\left(dy\right) & \leq & N_{2}\text{ if }\sigma=1,  \notag \\
\int_{\left\vert y\right\vert \leq1}\left\vert y\right\vert ^{2}\tilde{\pi}%
_{R}\left(dy\right)+\int_{\left\vert y\right\vert >1}\left\vert y\right\vert 
\tilde{\pi}_{R}\left(dy\right) & \leq & N_{2}\text{ if }\sigma\in\left(1,2%
\right)  \notag
\end{eqnarray}
for any $R>0.$ Then there is a constant $C_{1}$ so that for all $\xi\in%
\mathbf{R}^{d},$ 
\begin{eqnarray*}
\int\left[1-\cos\left(2\pi\xi y\right)\right]\pi\left(dy\right) & \leq &
C_{1}N_{2}w\left(\left\vert \xi\right\vert ^{-1}\right)^{-1}, \\
\int\left\vert \sin\left(2\pi\xi\cdot
y\right)-2\pi\chi_{\sigma}\left(y\right)\xi\cdot y\right\vert
\pi\left(dy\right) & \leq & C_{1}N_{2}w\left(\left\vert \xi\right\vert
^{-1}\right)^{-1},
\end{eqnarray*}
assuming $w\left(\left\vert \xi\right\vert ^{-1}\right)^{-1}=0$ if $\xi=0.$

\noindent b) Let 
\begin{equation}
\inf_{R\in\left(0,\infty\right),\left\vert \hat{\xi}\right\vert
=1}\int_{\left\vert y\right\vert \leq1}\left\vert \hat{\xi}\cdot
y\right\vert ^{2}\tilde{\pi}_{R}\left(dy\right)=c_{0}>0.  \label{fa00}
\end{equation}
Then there is a constant $c_{2}=c_{2}\left(w,c_{0}\right)>0$ such that 
\begin{equation*}
\int\left[1-\cos\left(2\pi\xi y\right)\right]\pi\left(dy\right)\geq
c_{2}w\left(\left\vert \xi\right\vert ^{-1}\right)^{-1} 
\end{equation*}
for all $\xi\in\mathbf{R}^{d},$ assuming $w\left(\left\vert \xi\right\vert
^{-1}\right)^{-1}=0$ if $\xi=0.$
\end{lemma}

\begin{proof}
The following simple estimates hold: 
\begin{eqnarray}
\left\vert \sin x-x\right\vert & \leq & \frac{\left\vert x\right\vert ^{3}}{6%
},1-\cos x\leq\frac{1}{2}x^{2},x\in\mathbf{R,}  \label{fa1} \\
1-\cos x & \geq & \frac{x^{2}}{\pi}\text{ if }\left\vert x\right\vert
\leq\pi/2.  \notag
\end{eqnarray}
Part a) was proved in Lemma 7 of \cite{MikPh1}.

b) By (\ref{fa1}), for all $\xi\in\mathbf{R}^{d},$ 
\begin{eqnarray*}
& & \int[1-\cos\left(2\pi\xi\cdot y\right)]\pi\left(dy\right) \\
& = & \int[1-\cos\left(2\pi\hat{\xi}\cdot y\right)]\pi_{\left\vert
\xi\right\vert ^{-1}}\left(dy\right)\geq\int_{\left\vert y\right\vert \leq%
\frac{1}{4}}4\pi\left\vert \hat{\xi}\cdot y\right\vert ^{2}\pi_{\left\vert
\xi\right\vert ^{-1}}\left(dy\right) \\
& = & 4^{-1}\int_{\left\vert 4y\right\vert \leq1}\pi\left\vert \hat{\xi}%
\cdot4y\right\vert ^{2}\pi_{\left\vert \xi\right\vert
^{-1}}\left(dy\right)=4^{-1}w\left(\left\vert 4\xi\right\vert
^{-1}\right)^{-1}\int_{\left\vert y\right\vert \leq1}\pi\left\vert \hat{\xi}%
\cdot y\right\vert ^{2}\tilde{\pi}_{\left\vert 4\xi\right\vert
^{-1}}\left(dy\right) \\
& \geq & cw\left(\left\vert 4\xi\right\vert ^{-1}\right)^{-1}\geq
cw\left(\left\vert \xi\right\vert ^{-1}\right)^{-1}.
\end{eqnarray*}
\end{proof}

\begin{lemma}
\label{Lem:expEst}Let $\nu\in\mathfrak{A}^{\sigma},w=w_{\nu}$ be an O-RV
function and \textbf{A, B} hold. Let $\zeta,\zeta_{0}\in C_{0}^{\infty}\left(%
\mathbf{R}^{d}\right)$ be such that $0\notin\text{supp}\left(\zeta\right)$ .
Let $\tilde{\zeta}=\mathcal{F}^{-1}\zeta,\tilde{\zeta}_{0}=\mathcal{F}%
^{-1}\zeta_{0}$, and for $R>0$, 
\begin{eqnarray*}
H^{R}\left(t,x\right) & = & \mathbf{E}\tilde{\zeta}\left(x+Z_{t}^{R}%
\right),t\geq0,x\in\mathbf{R}^{d}, \\
H_{0}^{R}\left(t,x\right) & = & \mathbf{E}\tilde{\zeta}_{0}\left(x+Z_{t}^{R}%
\right),t\geq0,x\in\mathbf{R}^{d}.
\end{eqnarray*}
(i) For any $0<\alpha_{2}<p_{1}^{\omega_{\nu}}\wedge p_{2}^{\omega_{\nu}}$,
there are constants $C_{0},C_{1},C_{2}>0$ independent of $R$ such that 
\begin{eqnarray*}
\int\left(1+\left\vert x\right\vert ^{\alpha_{2}}\right)\left\vert
H^{R}\left(t,x\right)\right\vert dx & \leq & C_{1}e^{-C_{2}t},t\geq0, \\
\int\left\vert x\right\vert ^{\alpha_{2}}\left\vert
H_{0}^{R}\left(t,x\right)\right\vert dx & \leq &
C_{0}\left(1+t\right),t\geq0, \\
\int\left\vert H_{0}^{R}\left(t,x\right)\right\vert dx & \leq & C_{0},t\geq0.
\end{eqnarray*}
(ii) There are constants $C_{1},C_{2}>0$ independent of $R$ so that for $y\in%
\mathbf{R}^{d},$ 
\begin{eqnarray*}
\int\left\vert H^{R}\left(t,x+y\right)-H^{R}\left(t,x\right)\right\vert dx &
\leq & C_{1}\left\vert y\right\vert e^{-C_{2}t}, \\
\int\left\vert
H_{0}^{R}\left(t,x+y\right)-H_{0}^{R}\left(t,x\right)\right\vert dx & \leq &
\left\vert y\right\vert \int\left\vert \nabla\tilde{\zeta}%
_{0}\left(x\right)\right\vert dx.
\end{eqnarray*}
\end{lemma}

\begin{proof}
(i) Note that 
\begin{equation*}
\mathcal{F}H^{R}(t,\xi)=\exp\left\{ \psi^{\tilde{\nu}_{R}}\left(\xi\right)t%
\right\} \zeta\left(\xi\right),\xi\in\mathbf{R}^{d}. 
\end{equation*}

Following Lemma 2 of \cite{MikPh2}(see also Corollary 2 of \cite{MF}). We
choose $\epsilon>0$ so that $\text{supp}\left(\zeta\right)\subset\left\{
\xi:\left\vert \xi\right\vert \leq\epsilon^{-1}\right\} $. Let $\tilde{\nu}%
_{R,\epsilon}\left(dy\right)=\chi_{\left\{ \left\vert y\right\vert
\leq\epsilon\right\} }\tilde{\nu}_{R}\left(dy\right),R\in\left(0,1\right]$.
Then for $\xi\in\text{supp}\left(\zeta\right)$, $\left\vert y\right\vert
\leq\epsilon$ and $R\in(0,1]$,

\begin{equation*}
1-\cos\left(\xi\cdot y\right)\geq\frac{1}{\pi}\left\vert \xi\cdot
y\right\vert ^{2}=\frac{\left\vert \xi\right\vert ^{2}}{\pi}\left\vert \hat{%
\xi}\cdot y\right\vert ^{2},\hspace{1em}\hat{\xi}=\xi/\left\vert
\xi\right\vert , 
\end{equation*}
and for some $c_{0}=c_{0}\left(\varepsilon\right),$ applying Lemma \ref%
{lem:powerEstratio},

\begin{align}
-\text{Re}\psi^{\tilde{\nu}_{R,\epsilon}}\left(\xi\right) &
=\int_{\left\vert y\right\vert \leq\epsilon}\left[1-\cos\left(\xi\cdot
y\right)\right]\tilde{\nu}_{R,\epsilon}\left(dy\right)  \notag \\
& \geq\frac{\left\vert \xi\right\vert ^{2}}{\pi}\int_{\left\vert
y\right\vert \leq\epsilon}\left\vert \hat{\xi}\cdot y\right\vert ^{2}\tilde{%
\nu}_{R}\left(dy\right)  \label{eq:symb2} \\
& =\frac{\left\vert \xi\right\vert ^{2}}{\pi}\int_{\left\vert y\right\vert
\leq1}\epsilon^{2}\left\vert \hat{\xi}\cdot y\right\vert ^{2}\frac{%
w\left(R\right)}{w\left(R\epsilon\right)}\tilde{\nu}_{\epsilon
R}\left(dy\right)\geq c_{0}\left\vert \xi\right\vert ^{2}  \notag
\end{align}
for all $\xi\in\mathbf{R}^{d}.$ Hence,

\begin{equation*}
H^{R}(t,\cdot)=F\left(t,\cdot\right)\ast P_{t}, 
\end{equation*}
where 
\begin{equation*}
F\left(t,x\right)=\mathcal{F}^{-1}\left[\exp\left\{ \psi^{\tilde{\nu}%
_{R,\epsilon}}t\right\} \zeta\right]\left(x\right)=\mathbf{E}\tilde{\zeta}%
\left(x+Z_{t}^{\tilde{\nu}_{R,\epsilon}}\right),t\geq0,x\in\mathbf{R}^{d}, 
\end{equation*}
and $P_{t}\left(dy\right)$ is the distribution of $Z_{t}^{\tilde{\nu}_{R}-%
\tilde{\nu}_{R,\epsilon}}$. By Plancherel for any multiindex $\gamma\in%
\mathbf{N}_{0}^{d}$ and (\ref{eq:symb2}), 
\begin{eqnarray*}
\int\left\vert x^{\gamma}F\left(t,x\right)\right\vert ^{2}dx & \leq &
C\int\left\vert D^{\gamma}\left[\zeta\left(\xi\right)\exp\left\{ \psi^{%
\tilde{\nu}_{R,\epsilon}}\left(\xi\right)t\right\} \right]\right\vert
^{2}d\xi \\
& \leq & C_{1}e^{-C_{2}t},t\geq0.
\end{eqnarray*}
By Cauchy-Schwarz inequality, with $d_{0}=\left[\frac{d}{2}\right]+1,$ 
\begin{eqnarray*}
& & \int\left(1+\left\vert x\right\vert ^{2}\right)\left\vert
F\left(t,x\right)\right\vert dx \\
& = & \int\left(1+\left\vert x\right\vert ^{2}\right)\left(1+\left\vert
x\right\vert \right)^{-d_{0}}\left\vert F\left(t,x\right)\right\vert
\left(1+\left\vert x\right\vert \right)^{d_{0}}dx \\
& \leq & \left(\int\left(1+\left\vert x\right\vert
\right)^{-2d_{0}}dx\right)^{1/2}\left(\int(1+\left\vert x\right\vert
)^{4}\left\vert F\left(t,x\right)\right\vert ^{2}\left(1+\left\vert
x\right\vert \right)^{2d_{0}}dx\right)^{1/2} \\
& \leq & C\left(\int F\left(t,x\right)^{2}\left(1+\left\vert x\right\vert
^{2}\right)^{d_{0}+2}dx\right)^{1/2}\leq C_{1}\exp\left\{ -C_{2}t\right\}
,t\geq0.
\end{eqnarray*}
Let $0<\alpha_{2}<p_{1}\wedge p_{2}$. By Lemma \ref{lem:alpha12}, and Lemma
17 of \cite{MikPh2}, there is $C>0$ so that 
\begin{equation*}
\mathbf{E}\left[\left\vert Z_{t}^{\tilde{\nu}_{R}-\tilde{\nu}%
_{R,\epsilon}}\right\vert ^{\alpha_{2}}\right]=\int\left\vert y\right\vert
^{\alpha_{2}}P_{t}\left(dy\right)\leq C\left(1+t\right),t\geq0. 
\end{equation*}
Hence there are constants $C_{1},C_{2}$ so that 
\begin{eqnarray*}
& & \int\left(1+\left\vert x\right\vert ^{\alpha_{2}}\right)\left\vert
H^{R}\left(t,x\right)\right\vert dx=\int\left(1+\left\vert x\right\vert
^{\alpha_{2}}\right)\left\vert \int
F\left(t,x-y\right)P_{t}\left(dy\right)\right\vert dx \\
& \leq & \int\int\left(1+\left\vert x-y\right\vert
^{\alpha_{2}}\right)\left\vert F\left(t,x-y\right)\right\vert
P_{t}\left(dy\right)dx \\
& & +\int\int\left\vert y\right\vert ^{\alpha_{2}}\left\vert
F\left(t,x-y\right)\right\vert P_{t}\left(dy\right)dx \\
& \leq & C_{1}e^{-C_{2}t},t\geq0,
\end{eqnarray*}
and 
\begin{eqnarray*}
\int\left\vert x\right\vert ^{\alpha_{2}}\left\vert
H_{0}^{R}\left(t,x\right)\right\vert dx & = & \int\left\vert x\right\vert
^{\alpha_{2}}\left\vert \mathbf{E}\tilde{\zeta}_{0}\left(x+Z_{t}^{R}\right)%
\right\vert dx \\
& \leq & \mathbf{E}\int\left\vert x+Z_{t}^{R}\right\vert
^{\alpha_{2}}\left\vert \tilde{\zeta}_{0}\left(x+Z_{t}^{R}\right)\right\vert
dx+\mathbf{E}\left[\left\vert Z_{t}^{R}\right\vert ^{\alpha_{2}}\right]%
\int\left\vert \tilde{\zeta}_{0}\left(x\right)\right\vert dx \\
& \leq & C\left(1+t\right).
\end{eqnarray*}
The last inequality is trivial.

\noindent (ii) Similarly as in part (i), for $y\in\mathbf{R}^{d},$ 
\begin{eqnarray*}
& & \int\left\vert H^{R}\left(t,x+y\right)-H^{R}\left(t,x\right)\right\vert
dx \\
& = & \int\left\vert \int\int_{0}^{1}\nabla F\left(t,x+sy-z\right)\cdot
ydsP_{t}\left(dz\right)\right\vert dx \\
& \leq & \left\vert y\right\vert \int\left\vert
DF\left(t,x\right)\right\vert dx\leq C_{1}\left\vert y\right\vert
e^{-C_{2}t},t>0,
\end{eqnarray*}
and directly 
\begin{equation*}
\int\left\vert
H_{0}^{R}\left(t,x+y\right)-H_{0}^{R}\left(t,x\right)\right\vert
dx\leq\left\vert y\right\vert \int\left\vert \nabla\tilde{\zeta}%
_{0}\left(x\right)\right\vert dx. 
\end{equation*}
\end{proof}

Let $V_{r}=L_{r}\left(U,\mathcal{U},\Pi\right),r\geq1$, the space of $r-$%
integrable measurable functions on $U$, and $V_{0}=\mathbf{R}$. For brevity
of notation we write 
\begin{equation*}
B_{r,pp}^{s}\left(A\right)=B_{pp}^{s}\left(A;V_{r}\right),\hspace{1em}%
\mathbb{B}_{r,pp}^{s}\left(A\right)=\mathbb{\mathbb{B}}_{pp}^{s}%
\left(A;V_{r}\right), 
\end{equation*}
\begin{eqnarray*}
H_{r,p}^{s}\left(A\right) & = & H_{p}^{s}\left(A;V_{r}\right),\hspace{1em}%
\mathbb{H}_{r,p}^{s}\left(A\right)=\mathbb{H}_{p}^{s}\left(A;V_{r}\right), \\
L_{r,p}\left(A\right) & = & H_{r,p}^{0}\left(A\right),\hspace{1em}\mathbb{L}%
_{r,p}\left(A\right)=\mathbb{H}_{r,p}^{0}\left(A\right),
\end{eqnarray*}
where $A=\mathbf{R}^{d}$ or $E$. We use the following equivalent norms of
Besov spaces $B_{r,pp}^{s}\left(\mathbf{R}^{d}\right),r=0,p,$ (see
Proposition \ref{prop-FuncSpace} below) 
\begin{equation*}
|v|_{\tilde{B}_{r,pp}^{s}(\mathbf{R}^{d})}=\left(\sum_{j=0}^{\infty}w%
\left(N^{-j}\right)^{-sp}\int|\varphi_{j}\ast v|_{V_{r}\
}^{p}dx\right)^{1/p}, 
\end{equation*}
where $\varphi_{j}=\varphi_{j}^{N},j\geq0,$ is the system of functions
defined in Remark \ref{rem:system}. The equivalent norms of $%
H_{r,p}^{s}\left(\mathbf{R}^{d}\right),r=0,2,$ (see Proposition \ref%
{prop-FuncSpace} below) are defined by 
\begin{equation}
|v|_{\tilde{H}_{r,p}^{s}(\mathbf{R}^{d})}=|v|_{\tilde{H}_{r,p}^{\nu,N;s}(%
\mathbf{R}^{d})}=\left\vert \left(\sum_{j=0}^{\infty}\left\vert
w\left(N^{-j}\right)^{-s}\varphi_{j}\ast v\right\vert
_{V_{r}}^{2}\right)^{1/2}\right\vert _{L_{p}\left(\mathbf{R}^{d}\right)}.
\end{equation}
We define the equivalent norms of functions on $E$ as well: 
\begin{eqnarray*}
|v|_{\tilde{B}_{p,pp}^{s}(E)} & = & \left(\int_{0}^{T}\left\vert
v\left(t\right)\right\vert _{\tilde{B}_{p,pp}^{s}(\mathbf{R}%
^{d})}^{p}dt\right)^{1/p} \\
|v|_{\tilde{H}_{2,p}^{s}(E)} & = & \left(\int_{0}^{T}\left\vert
v\left(t\right)\right\vert _{\tilde{H}_{2,pp}^{s}(\mathbf{R}%
^{d})}^{p}dt\right)^{1/p}.
\end{eqnarray*}
For $D=D_{r}\left(A\right)=B_{r,pp}^{s}\left(A\right)$ or $%
H_{r,p}^{s}\left(A\right),A=\mathbf{R}^{d},E$, we consider corresponding
equivalent norms on random function spaces $\mathbb{D}=\mathbb{D}%
_{r}\left(A\right)=\mathbb{B}_{r,pp}^{s}\left(A\right)$ or $\mathbb{H}%
_{r,p}^{s}\left(A\right):$ 
\begin{equation*}
\left\vert v\right\vert _{\mathbb{\tilde{D}}}=\left\{ \mathbf{E}%
\left(\left\vert v\right\vert _{\tilde{D}}^{p}\right)\right\} ^{1/p}. 
\end{equation*}

Lemmas \ref{Lem: symbol_omega}, \ref{Lem:expEst}, and \ref{lem:alpha12} are
key ingredients in proving the following characterization of our Bessel
potential spaces and Besov spaces. The proof follows almost verbatim with
obvious changes the corresponding proof in \cite{MikPh2}. Therefore we omit
it, and simply state the result.

\begin{proposition}
\label{prop-FuncSpace}Let $\nu\in\mathfrak{A}^{\sigma},w=w_{\nu}$ be an O-RV
function and \textbf{A, B} hold. Let $s\in\mathbf{R},p,q\in\left(1,\infty%
\right)$.

(i) $\tilde{B}_{pq}^{\nu;s}\left(\mathbf{R}^{d}\right)=B_{pq}^{\nu;s}\left(%
\mathbf{R}^{d}\right)$ and the norms are equivalent

(ii) $\tilde{H}_{p}^{\nu;s}\left(\mathbf{R}^{d}\right)=H_{p}^{\nu;s}\left(%
\mathbf{R}^{d}\right)$ and the norms are equivalent

(iii) $\tilde{H}_{p}^{\nu;s}\left(\mathbf{R}^{d};l_{2}\right)=H_{p}^{\nu;s}%
\left(\mathbf{R}^{d};l_{2}\right)$ and the norms are equivalent.

Moreover, for all $s,s^{\prime}\in\mathbf{R}$, $J^{s}:A^{s^{\prime}}%
\rightarrow A^{s^{\prime}-s}$ is an isomorphism where $A^{s}=B_{pq}^{\nu;s}%
\left(\mathbf{R}^{d}\right),H_{p}^{\nu;s}\left(\mathbf{R}^{d}\right)$ or $%
H_{p}^{\nu;s}\left(\mathbf{R}^{d};l_{2}\right).$
\end{proposition}

Let $U_{n}\in\mathcal{U},U_{n}\subseteq U_{n+1},n\geq1,\cup_{n}U_{n}=U$ and $%
\Pi\left(U_{n}\right)<\infty,n\geq1.$ We denote by $\mathbb{\tilde{C}}%
_{r.p}^{\infty}\left(E\right),1\leq p<\infty,$ the space of all $\mathcal{R}%
\left(\mathbb{F}\right)\otimes\mathcal{B}\left(\mathbf{R}^{d}\right)-$%
measurable $V_{r}$ -valued random functions $\Phi$ on $E$ such that for
every $\gamma\in\mathbf{N}_{0}^{d}$, 
\begin{equation*}
\mathbf{E}\int_{0}^{T}\sup_{x\in\mathbf{R}^{d}}\left\vert
D^{\gamma}\Phi\left(t,x\right)\right\vert _{V_{r}}^{p}dt+\mathbf{E}\left[%
\left\vert D^{\gamma}\Phi\right\vert _{L_{p}\left(E;V_{r}\right)}^{p}\right]%
<\infty, 
\end{equation*}
and $\Phi=\Phi\raisebox{2pt}{\ensuremath{\chi}}_{U_{n}}$ for some $n$ if $%
r=2,p.$ Similarly we define the space $\mathbb{\tilde{C}}_{r.p}^{\infty}%
\left(\mathbf{R}^{d}\right)$ by replacing $\mathcal{R}\left(\mathbb{F}\right)
$ and $E$ by $\mathcal{F}$ and $\mathbf{R}^{d}$ respectively in the
definition of $\mathbb{\tilde{C}}_{r,p}^{\infty}\left(E\right)$.

Let $\nu\in\mathfrak{A}^{\sigma},N>1,w=w_{\nu}$ is a continuous O-RV
function and \textbf{A, B} hold. We now discuss an approximating sequence of
the input $\Phi$, based on Proposition \ref{prop-FuncSpace}.

\begin{lemma}
\label{lem-approximation1} Let $\nu\in\mathfrak{A}^{\sigma},w=w_{\nu}$ be an
O-RV function and \textbf{A, B} hold. Let $U_{n}\in\mathcal{U}%
,U_{n}\subseteq U_{n+1},n\geq1,\cup_{n}U_{n}=U$ and $\Pi\left(U_{n}\right)<%
\infty,n\geq1.$ Let $s\in\mathbf{R},p\in\left(1,\infty\right)$, $\Phi\in%
\mathbb{D}_{r,p}$, where $\mathbb{D}_{r,p}=\mathbb{D}_{r,p}\left(A\right)=%
\mathbb{B}_{r,pp}^{s}\left(A\right)$ with $r=0,p,$ or $\mathbb{D}_{r,p}=%
\mathbb{H}_{r,p}^{s}\left(A\right)$ with $r=0,2$, $A=\mathbf{R}^{d}$ or $E$.
For $\Phi\in\mathbb{D}_{r,p}$ we set 
\begin{equation*}
\Phi_{n}=\sum_{j=0}^{n}\Phi\ast\varphi_{j}\raisebox{2pt}{\ensuremath{\chi}}%
_{U_{n}},\text{ if }r=2,p,~\Phi_{n}=\sum_{j=0}^{n}\Phi\ast\varphi_{j},\text{
if }r=0. 
\end{equation*}
Then there is $C>0$ so that 
\begin{equation*}
\left\vert \Phi_{n}\right\vert _{\mathbb{D}_{r,p}}\leq C\left\vert
\Phi\right\vert _{\mathbb{D}_{r,p}},\Phi\in\mathbb{D}_{r,p},n\geq1\text{,} 
\end{equation*}
and $\left\vert \Phi_{n}-\Phi\right\vert _{\mathbb{D}_{r,p}}\rightarrow0$ as 
$n\rightarrow\infty.$ Moreover, for $r=0,2,p,$ every $n$ and multiindex $%
\gamma\in\mathbf{N}_{0}^{d}$, 
\begin{eqnarray*}
\mathbf{E}\int_{0}^{T}\sup_{x}\left\vert D^{\gamma}\Phi_{n}\right\vert
_{V_{r}}^{p}dt+\left\vert D^{\gamma}\Phi_{n}\right\vert _{\mathbb{L}%
_{r,p}\left(E\right)}^{p} & < & \infty\text{ if }A=E, \\
\mathbf{E[}\sup_{x}\left\vert D^{\gamma}\Phi_{n}\right\vert
_{V_{r}}^{p}]+\left\vert D^{\gamma}\Phi_{n}\right\vert _{\mathbb{L}%
_{r,p}\left(\mathbf{R}^{d}\right)}^{p} & < & \infty\text{ if }A=\mathbf{R}%
^{d},
\end{eqnarray*}
\end{lemma}

\begin{proof}
Let $\tilde{\Phi}_{n}=\Phi\raisebox{2pt}{\ensuremath{\chi}}_{U_{n}},n\geq1.$
Since 
\begin{equation*}
\varphi_{k}=\sum_{l=-1}^{1}\varphi_{k+l}\ast\varphi_{k},k\geq1,\varphi_{0}=%
\left(\varphi_{0}+\varphi_{1}\right)\ast\varphi_{0}, 
\end{equation*}
we have for $n>1,$ 
\begin{eqnarray*}
\left(\tilde{\Phi}_{n}-\Phi_{n}\right)\ast\varphi_{k} & = & 0,k<n, \\
\left(\tilde{\Phi}_{n}-\Phi_{n}\right)\ast\varphi_{k} & = & \left(\tilde{\Phi%
}_{n}\ast\varphi_{k-1}+\tilde{\Phi}_{n}\ast\varphi_{k}+\tilde{\Phi}%
_{n}\ast\varphi_{k+1}\right)\ast\varphi_{k},k>n+1, \\
\left(\tilde{\Phi}_{n}-\Phi_{n}\right)\ast\varphi_{n} & = & \left(\tilde{\Phi%
}_{n}\ast\varphi_{n+1}\right)\ast\varphi_{n}, \\
\left(\tilde{\Phi}_{n}-\Phi_{n}\right)\ast\varphi_{n+1} & = & \left(\tilde{%
\Phi}_{n}\ast\varphi_{n+1}+\tilde{\Phi}_{n}\ast\varphi_{n+2}\right)\ast%
\varphi_{n+1}.
\end{eqnarray*}
Let $V_{r}=L_{r}\left(U,\mathcal{U},\Pi\right),r=2,p$ and $V_{0}=\mathbf{R}$%
. By Corollary 2 in \cite{MikPh2}, there is a constant $C$ independent of $%
\Phi\in\mathbb{H}_{2,p}^{s}\left(E\right)$ so that 
\begin{eqnarray*}
& & \left\vert \left(\sum_{j=0}^{\infty}\left\vert
w\left(N^{-j}\right)^{-s}\left(\tilde{\Phi}_{n}-\Phi_{n}\right)\ast%
\varphi_{j}\right\vert _{V_{r}}^{2}\right)^{1/2}\right\vert
_{L_{p}\left(E\right)} \\
& \leq & C\left\vert \left(\sum_{j=n}^{\infty}\left\vert
w\left(N^{-j}\right)^{-s}\Phi\ast\varphi_{j}\right\vert
_{V_{r}}^{2}\right)^{1/2}\right\vert
_{L_{p}\left(E\right)}\rightarrow0,r=0,2,
\end{eqnarray*}
as $n\rightarrow\infty.$ Obviously 
\begin{eqnarray*}
\left\vert \left(\tilde{\Phi}_{n}-\Phi_{n}\right)\ast\varphi_{j}\right\vert
_{L_{r,p}\left(E\right)} & \leq & C\sum_{k=j-1}^{j+1}\left\vert \tilde{\Phi}%
_{n}\ast\varphi_{k}\right\vert _{L_{r,p}\left(E\right)},j\geq n, \\
\left\vert \left(\tilde{\Phi}_{n}-\Phi_{n}\right)\ast\varphi_{j}\right\vert
_{L_{r,p}\left(E\right)} & = & 0,j<n,r=0,p,
\end{eqnarray*}
and 
\begin{equation*}
\left\vert \Phi_{n}\right\vert _{\mathbb{D}_{r,p}\left(E\right)}\leq
C\left\vert \Phi\right\vert _{\mathbb{D}_{r,p}\left(E\right)},\Phi\in\mathbb{%
D}_{r,p},n\geq1,r=0,2,p\text{.} 
\end{equation*}
Thus $\left\vert \Phi_{n}-\Phi\right\vert _{\mathbb{D}_{r,p}\left(E\right)}%
\rightarrow0$ as $n\rightarrow\infty,r=0,2,p.$

Let $r=0,2,p$, $\Phi\in\mathbb{D}_{r,p}\left(E\right)$. Obviously, for any $%
k\geq0,$ 
\begin{equation*}
\mathbf{E}\int_{E}\left\vert \Phi\ast\varphi_{k}\right\vert
_{V_{r}}^{p}dxdt<\infty. 
\end{equation*}
Since for any multiindex $\gamma\in\mathbf{N}_{0}^{d},$ 
\begin{equation*}
\Phi\ast\varphi_{k}=\Phi\ast\varphi_{k}\ast\tilde{\varphi}%
_{k},D^{\gamma}\Phi\ast\varphi_{k}=\Phi\ast\varphi_{k}\ast D^{\gamma}\tilde{%
\varphi}_{k}, 
\end{equation*}
and $\mathbf{P}$-a.s. for all $s,x,$ with $\frac{1}{q}+\frac{1}{p}=1,$ 
\begin{eqnarray*}
\left\vert D^{\gamma}\Phi\ast\varphi_{k}\left(s,x\right)\right\vert _{V_{r}}
& \leq & \int\left\vert \Phi\ast\varphi_{k}\left(s,x-y\right)\right\vert
_{V_{r}}\left\vert D^{\gamma}\tilde{\varphi}_{k}\left(y\right)\right\vert dy,
\\
\sup_{x}\left\vert D^{\gamma}\Phi\ast\varphi_{k}\left(s,x\right)\right\vert
_{V_{r}} & \leq & \left(\int\left\vert
\Phi\ast\varphi_{k}\left(s,\cdot\right)\right\vert
_{V_{r}}^{p}dx\right)^{1/p}\left\vert D^{\gamma}\tilde{\varphi}%
_{k}\right\vert _{L_{q}\left(\mathbf{R}^{d}\right)},
\end{eqnarray*}
we have for any multiindex $\gamma,$ 
\begin{equation*}
\left\vert D^{\gamma}\Phi\ast\varphi_{k}\right\vert _{\mathbb{L}%
_{r,p}\left(E\right)}<\infty, 
\end{equation*}
and 
\begin{equation*}
\mathbf{E}\int_{0}^{T}\sup_{x}\left\vert
D^{\gamma}\Phi\ast\varphi_{k}\right\vert _{V_{r}}^{p}dt<\infty,r=0,2,p. 
\end{equation*}
The proof for the case of $A=\mathbf{R}^{d}$ is a repeat with obvious changes%
$.$ The statement follows.
\end{proof}

\begin{corollary}
\label{cor:approx1}The space $C_{0}^{\infty}\left(\mathbf{R}^{d};V_{r}\right)
$ of $V_{r}$-valued infinitely differentiable functions with compact support
is dense in $D_{r}\left(\mathbf{R}^{d}\right),r=0,2,p$.
\end{corollary}

\begin{proof}
In the view of Lemma \ref{lem-approximation1}, it suffices to show that for
any $V=V_{r}$-valued function $f$ such that for all multiindex $\gamma\in%
\mathbf{N}_{0}^{d}$, 
\begin{equation*}
\sup_{x}\left\vert D^{\gamma}f\left(x\right)\right\vert _{V_{r}}+\left\vert
D^{\gamma}f\right\vert _{L_{p}\left(\mathbf{R}^{d};V_{r}\right)}<\infty, 
\end{equation*}
there exists $f_{n}\in C_{0}^{\infty}\left(\mathbf{R}^{d},V_{r}\right)$ so
that $f_{n}\rightarrow f$ in $D_{r}\left(\mathbf{R}^{d}\right)$. Let $g\in
C_{0}^{\infty}\left(\mathbf{R}^{d}\right)$ with $0\leq
g\left(x\right)\leq1,x\in\mathbf{R}^{d}$, $g\left(x\right)=1$ for $%
\left\vert x\right\vert \leq1$, and $g\left(x\right)=0$ for $\left\vert
x\right\vert \geq2$. Let 
\begin{equation*}
f_{n}\left(x\right):=f\left(x\right)g\left(x/n\right),x\in\mathbf{R}^{d}. 
\end{equation*}
Obviously $f_{n}\in C_{0}^{\infty}\left(\mathbf{R}^{d},V_{r}\right)$, and
for any multiindex $\beta,$

\begin{eqnarray*}
& & D^{\beta}f_{n}\left(x\right) \\
& = &
D^{\beta}f\left(x\right)g\left(x/n\right)+\sum_{\beta_{1}+\beta_{2}=\beta,%
\left\vert \beta_{2}\right\vert \geq1}n^{-\left\vert \beta_{2}\right\vert
}D^{\beta_{1}}f\left(x\right)\left(D^{\beta_{2}}g\right)\left(x/n\right),x\in%
\mathbf{R}^{d}, \\
& & \left\vert D^{\beta}f_{n}\right\vert _{Lp\left(\mathbf{R}%
^{d};V_{r}\right)} \\
& \leq & C\left(\left\vert \beta\right\vert
\right)\sup_{\beta^{\prime}\leq\left\vert \beta\right\vert }\left\vert
D^{\beta^{\prime}}f\right\vert _{Lp\left(\mathbf{R}^{d};V_{r}\right)},
\end{eqnarray*}
and $\left\vert D^{\beta}f_{n}-D^{\beta}f\right\vert _{Lp\left(\mathbf{R}%
^{d};V_{r}\right)}\rightarrow0$. Since for any multiindex $\beta$ we have$%
\int y^{\beta}\varphi_{j}\left(y\right)dy=0,$ it follows for $m>0,j\geq1$,
by Taylor remainder theorem, for $x\in\mathbf{R}^{d},$

\begin{eqnarray*}
& & f_{n}\ast\varphi_{j}\left(x\right) \\
& = & \int\varphi_{j}\left(y\right)\left\{
f_{n}\left(x-y\right)-\sum_{\beta:\left\vert \beta\right\vert \leq m}\frac{%
D^{\beta}f_{n}\left(x\right)}{\beta!}\left(-y\right)^{\beta}\right\} dy \\
& = & \int\varphi_{j}\left(y\right)\sum_{\beta:\left\vert \beta\right\vert
=m+1}\int_{0}^{1}\frac{\left(1-t\right)^{m+1}}{\left(m+1\right)!}%
\left(D^{\beta}f_{n}\right)\left(x-ty\right)\left(-y\right)^{\beta}dtdy \\
& = & N^{-j(m+1)}\sum_{\beta:\left\vert \beta\right\vert
=m+1}\int\varphi\left(y\right)\int_{0}^{1}\frac{\left(1-t\right)^{m+1}}{%
\left(m+1\right)!}\left(D^{\beta}f_{n}\right)\left(x-tN^{-j}y\right)dt%
\left(-y\right)^{\beta}dy.
\end{eqnarray*}
By Lemma \ref{lem:powerEstratio}, there exists $\sigma^{\prime}>0$ such that 
$w\left(N^{-j}\right)^{-s}\leq CN^{js\sigma^{\prime}}$. Let $m>1$ be such
that $t=N^{\sigma^{\prime}s}N^{-m}<1$. Hence there is a constant $%
C=C\left(m,w_{\nu}\right)$ (independent of $n$) so that 
\begin{equation*}
w\left(N^{-j}\right)^{-s}\left\vert f_{n}\ast\varphi_{j}\right\vert
_{L_{p}\left(\mathbf{R}^{d};V_{r}\right)}\leq
C\left(m,w_{\nu}\right)t^{j},j\geq0. 
\end{equation*}
Now, for $r=0,2$ and any $k\geq0,$

\begin{eqnarray*}
& & \left\vert \left(\sum_{j=k}^{\infty}\left\vert
w\left(N^{-j}\right)^{-s}f_{n}\ast\varphi_{j}\left(x\right)\right\vert
_{V_{r}}^{2}\right)^{1/2}\right\vert _{L_{p}\left(\mathbf{R}^{d}\right)} \\
& \leq & \left\vert \sum_{j=k}^{\infty}\left\vert
w\left(N^{-j}\right)^{-s}f_{n}\ast\varphi_{j}\left(x\right)\right\vert
_{V_{r}}\right\vert _{L_{p}\left(\mathbf{R}^{d}\right)}\leq\sum_{j\geq
k}\left\vert
w\left(N^{-j}\right)^{-s}f_{n}\ast\varphi_{j}\left(x\right)\right\vert
_{L_{p}\left(\mathbf{R}^{d};V_{r}\right)} \\
& \leq & C\left(m,w_{\nu}\right)\sum_{j\geq k}t^{j}.
\end{eqnarray*}

Since the same estimate holds for $f,$ 
\begin{equation*}
\left\vert \left(\sum_{j\geq k}\left\vert
w\left(N^{-j}\right)^{-s}f\ast\varphi_{j}\left(x\right)\right\vert
_{V_{r}}^{2}\right)^{1/2}\right\vert _{L_{p}\left(\mathbf{R}^{d}\right)}\leq
C\left(m,w_{\nu}\right)\sum_{j\geq k}t^{j}, 
\end{equation*}
and 
\begin{equation*}
\left\vert \left(\sum_{j<k}\left\vert
w\left(N^{-j}\right)^{-s}(f-f_{n})\ast\varphi_{j}\left(x\right)\right\vert
_{V_{r}}^{2}\right)^{1/2}\right\vert _{L_{p}\left(\mathbf{R}%
^{d}\right)}\rightarrow0, 
\end{equation*}
it follows that 
\begin{equation*}
\left\vert \left(\sum_{j}\left\vert
w\left(N^{-j}\right)^{-s}(f-f_{n})\ast\varphi_{j}\left(x\right)\right\vert
_{V_{r}}^{2}\right)^{1/2}\right\vert _{L_{p}\left(\mathbf{R}%
^{d}\right)}\rightarrow0 
\end{equation*}
as $n\rightarrow\infty$.

Likewise, for $r=0,p,$ 
\begin{equation*}
\lim_{n\rightarrow\infty}\left\vert f_{n}-f\right\vert
_{B_{r,pp}^{s}}=\lim_{n\rightarrow\infty}\left(\sum_{j=0}^{\infty}\left\vert
w\left(N^{-j}\right)^{-s}\left(f_{n}-f\right)\ast\varphi_{j}\right\vert
_{L_{p}\left(\mathbf{R}^{d};V_{r}\right)}^{p}\right)^{1/p}=0. 
\end{equation*}
\end{proof}

An obvious consequence of Lemma \ref{lem-approximation1} (the form of the
approximating sequence is identical for different $V_{r}$) is the following

\begin{lemma}
\label{lem-denseSubspace} Let $p\geq1$ and $s,s^{\prime}\in\mathbf{R}$. Then
the set $\tilde{\mathcal{\mathbb{C}}}_{p,p}^{\infty}\left(E\right)$ is a
dense subset in $\mathbb{B}_{p,pp}^{s^{\prime}}\left(E\right),\tilde{%
\mathcal{\mathbb{C}}}_{0,p}^{\infty}\left(\mathbf{R}^{d}\right)$ is a dense
subset of $\mathbb{B}_{pp}^{s^{\prime}}\left(\mathbf{R}^{d}\right),$ and $%
\tilde{\mathcal{\mathbb{C}}}_{r,p}^{\infty}\left(E\right)$ is dense in $%
\mathbb{H}_{r,p}^{s}\left(E\right),r=0,2.$ Moreover, the set $\tilde{%
\mathcal{\mathbb{C}}}_{2,p}^{\infty}\left(E\right)\cap\tilde{\mathcal{%
\mathbb{C}}}_{p,p}^{\infty}\left(E\right)$ is a dense subset of $\mathbb{B}%
_{p,pp}^{s^{\prime}}\left(E\right)\cap\mathbb{H}_{2,p}^{s}\left(E\right)$.
\end{lemma}

\section{Proof of the main results}

\subsection{Existence and uniqueness of solution for smooth input functions}

Let $\nu\in\mathfrak{A}^{\sigma}$, and $Z_{t}=Z_{t}^{\nu},t\geq0,$ be the
Levy process associated to it. Let $P_{t}\left(dy\right)$ be the
distribution of $Z_{t}^{\nu},t>0$, and for a measurable $f\geq0,$ 
\begin{equation*}
T_{t}f\left(x\right)=\int
f\left(x+y\right)P_{t}\left(dy\right),\left(t,x\right)\in E. 
\end{equation*}

We represent the solution to (\ref{eq:mainEq}) with smooth input functions
using the following operators:

\begin{align*}
T_{t}^{\lambda}g\left(x\right) & =e^{-\lambda t}\int
g\left(x+y\right)P_{t}\left(dy\right),\left(t,x\right)\in E,\hspace{1em}g\in%
\tilde{\mathcal{\mathbb{C}}}_{0,p}^{\infty}\left(\mathbf{R}^{d}\right),p>1,
\\
R_{\lambda}f\left(t,x\right) & =\int_{0}^{t}e^{-\lambda\left(t-s\right)}\int
f\left(s,x+y\right)P_{t-s}\left(dy\right)ds,\left(t,x\right)\in E,\hspace{1em%
}f\in\tilde{\mathcal{\mathbb{C}}}_{0,p}^{\infty}\left(E\right),p>1,
\end{align*}
and 
\begin{eqnarray*}
\tilde{R}_{\lambda}\Phi\left(t,x\right) & = &
\int_{0}^{t}e^{-\lambda\left(t-s\right)}\int_{U}\int\Phi\left(s,x+y,z%
\right)P_{t-s}\left(dy\right)q\left(ds,dz\right),\left(t,x\right)\in E,%
\hspace{1em} \\
\Phi & \in & \tilde{\mathcal{\mathbb{C}}}_{2,p}^{\infty}\left(E\right)\cap%
\tilde{\mathcal{\mathbb{C}}}_{p,p}^{\infty}\left(E\right)\text{ if }%
p\geq2,\Phi\in\tilde{\mathcal{\mathbb{C}}}_{p,p}^{\infty}\left(E\right)\text{
if }p\in\left(1,2\right).
\end{eqnarray*}

We remind readers that here $q\left(ds,dz\right)$ is the martingale measure
as given in (\ref{eq:mainEq}).

\begin{lemma}
\label{lem:est-smooth}Let $f\in\mathcal{\mathbb{\tilde{C}}}%
_{0,p}^{\infty}\left(E\right),g\in\tilde{\mathcal{\mathbb{C}}}%
_{0,p}^{\infty}\left(\mathbf{R}^{d}\right),\Phi\in\mathbb{\tilde{\mathcal{%
\mathbb{C}}}}_{2,p}^{\infty}\left(E\right)\cap{\tilde{\mathcal{\mathbb{C}}}}%
_{p,p}^{\infty}\left(E\right)$ for $p\in\left[2,\infty\right)$ and $\Phi\in{%
\tilde{\mathcal{\mathbb{C}}}}_{p,p}^{\infty}\left(E\right)$ for $p\in(1,2)$,
then the following estimates hold for any multiindex $\gamma\in\mathbf{N}%
_{0}^{d}$:

(i) $\mathbf{P}$-a.s. 
\begin{eqnarray*}
\left\vert D^{\gamma}T^{\lambda}g\right\vert _{L_{p}\left(E\right)} & \leq &
\rho_{\lambda}^{\frac{1}{p}}\left\vert D^{\gamma}g\right\vert _{L_{p}\left(%
\mathbf{R}^{d}\right)},\hspace{1em}g\in\tilde{\mathbb{C}}_{0,p}^{\infty}%
\left(\mathbf{R}^{d}\right),p\geq1, \\
\left\vert D^{\gamma}R_{\lambda}f\right\vert _{L_{p}\left(E\right)} & \leq &
\rho_{\lambda}\left\vert D^{\gamma}f\right\vert _{L_{p}\left(E\right)},%
\hspace{1em}f\in\mathbb{\tilde{C}}_{0,p}^{\infty}\left(E\right),p\geq1,
\end{eqnarray*}
and 
\begin{eqnarray*}
\left\vert D^{\gamma}R_{\lambda}f\left(t,\cdot\right)\right\vert
_{L_{p}\left(\mathbf{R}^{d}\right)} & \leq & \int_{0}^{t}\left\vert
D^{\gamma}f\left(s,\cdot\right)\right\vert _{L_{p}\left(\mathbf{R}%
^{d}\right)}ds,t\geq0, \\
\left\vert T_{t}^{\lambda}g\right\vert _{L_{p}\left(\mathbf{R}^{d}\right)} &
\leq & e^{-\lambda t}\left\vert g\right\vert _{L_{p}\left(\mathbf{R}%
^{d}\right)},t\geq0,p\geq1;
\end{eqnarray*}

(ii) For each $p\geq2,$ 
\begin{eqnarray*}
& & \left\vert D^{\gamma}\tilde{R}_{\lambda}\Phi\right\vert _{\mathbb{L}%
_{p}\left(E\right)}^{p} \\
& \leq & C\left[\rho_{\lambda}^{\frac{p}{2}}\mathbf{E}\int_{0}^{T}\left\vert
D^{\gamma}\Phi\left(s,\cdot\right)\right\vert _{L_{2,p}\left(\mathbf{R}%
^{d}\right)}^{p}ds+\rho_{\lambda}\left\vert D^{\gamma}\Phi\right\vert _{%
\mathbb{L}_{p,p}\left(E\right)}^{p}\right], \\
\Phi & \in & \tilde{\mathbb{C}}_{2,p}^{\infty}\left(E\right)\cap\tilde{%
\mathbb{C}}_{p,p}^{\infty}\left(E\right),
\end{eqnarray*}
and for each \,$p\in(1,2),$ 
\begin{equation*}
\left\vert D^{\gamma}\tilde{R}_{\lambda}\Phi\right\vert _{\mathbb{L}%
_{p}\left(E\right)}^{p}\leq C\rho_{\lambda}\left\vert
D^{\gamma}\Phi\right\vert _{\mathbb{L}_{p,p}\left(E\right)}^{p},\Phi\in%
\tilde{\mathbb{C}}_{p,p}^{\infty}\left(E\right), 
\end{equation*}
where $\rho_{\lambda}=T\wedge\frac{1}{\lambda}.$ Moreover, 
\begin{eqnarray*}
& & \left\vert D^{\gamma}\tilde{R}_{\lambda}\Phi\left(t,\cdot\right)\right%
\vert _{\mathbb{L}_{p}\left(\mathbf{R}^{d}\right)}^{p} \\
& \leq & C\left\{ \mathbf{E}\left[\left(\int_{0}^{t}\left\vert
D^{\gamma}\Phi\left(s,\cdot\right)\right\vert _{L_{2,p}\left(\mathbf{R}%
^{d}\right)}^{2}ds\right)^{p/2}\right]+\mathbf{E}\int_{0}^{t}\left\vert
D^{\gamma}\Phi\left(s,\cdot\right)\right\vert _{L_{p,p}\left(\mathbf{R}%
^{d}\right)}^{p}ds\right\} ,
\end{eqnarray*}
if $p\geq2$, and 
\begin{equation*}
\left\vert D^{\gamma}\tilde{R}_{\lambda}\Phi\left(t,\cdot\right)\right\vert
_{\mathbb{L}_{p}\left(\mathbf{R}^{d}\right)}^{p}\leq C\mathbf{E}%
\int_{0}^{t}\left\vert D^{\gamma}\Phi\left(s,\cdot\right)\right\vert
_{L_{p,p}\left(\mathbf{R}^{d}\right)}^{p}ds,t\geq0, 
\end{equation*}
if $p\in\left(1,2\right).$
\end{lemma}

\begin{proof}
The estimates (i) follow from Lemma 15 in \cite{MikPh2} and Lemma 8 in \cite%
{MikPh1}. Due to similarity, we only prove (ii) for the case $p\geq2.$ Let $%
\Phi\in\tilde{\mathbb{C}}_{2,p}^{\infty}\left(E\right)\cap\tilde{\mathbb{C}}%
_{p,p}^{\infty}\left(E\right)$, recall that $\Phi=\Phi\raisebox{2pt}{%
\ensuremath{\chi}}_{U_{n}}$ for some $U_{n}\in\mathcal{U}$ with $%
\Pi\left(U_{n}\right)<\infty.$ Obviously, for any multiindex $\gamma$, $%
\left(t,x\right)\in E,$ 
\begin{eqnarray*}
D^{\gamma}\tilde{R}_{\lambda}\Phi\left(t,x\right) & = &
\int_{0}^{t}e^{-\lambda\left(t-s\right)}\int_{U}\int
D^{\gamma}\Phi\left(s,x+y,z\right)P_{t-s}\left(dy\right)q\left(ds,dz\right)
\\
& = & \int_{0}^{t}e^{-\lambda\left(t-s\right)}\int_{U}\int
D^{\gamma}\Phi\left(s,x+y,z\right)P_{t-s}\left(dy\right)p\left(ds,dz\right)
\\
& & -\int_{0}^{t}e^{-\lambda\left(t-s\right)}\int_{U}\int
D^{\gamma}\Phi\left(s,x+y,z\right)P_{t-s}\left(dy\right)\Pi\left(dz\right)ds.
\end{eqnarray*}
By Kunita's inequality (see Theorem 2.11 of \cite{ku}), for $t>0,$ 
\begin{eqnarray*}
& & \mathbf{E}\int\left\vert D^{\gamma}\tilde{R}_{\lambda}\Phi\left(t,x%
\right)\right\vert ^{p}dx \\
& \leq & C\mathbf{E}\int\left(\int_{0}^{t}e^{-2\lambda\left(t-s\right)}\int%
\left(\int_{U}\left\vert D^{\gamma}\Phi\left(s,x+y,z\right)\right\vert
^{2}\Pi\left(dz\right)\right)P_{t-s}\left(dy\right)ds\right)^{p/2}dx \\
& & +C\mathbf{E}\int\int_{0}^{t}e^{-p\lambda\left(t-s\right)}\left\vert
T_{t-s}D^{\gamma}\Phi\left(s,x,z\right)\right\vert ^{p}\Pi\left(dz\right)dsdx
\\
& = & B\left(t\right)+D\left(t\right).
\end{eqnarray*}
By Fubini theorem and Minkowski inequality, 
\begin{equation*}
D\left(t\right)\leq C\mathbf{E}\int_{0}^{t}e^{-p\lambda\left(t-s\right)}%
\left\vert D^{\gamma}\Phi\left(s,\cdot\right)\right\vert _{L_{p,p}\left(%
\mathbf{R}^{d}\right)}^{p}ds,t>0, 
\end{equation*}

and 
\begin{eqnarray*}
B\left(t\right) & \leq & C\mathbf{E}\left[\left(\int_{0}^{t}e^{-2\lambda%
\left(t-s\right)}\left\vert D^{\gamma}\Phi\left(s,\cdot\right)\right\vert
_{L_{2,p}\left(\mathbf{R}^{d}\right)}^{2}ds\right)^{p/2}\right] \\
& \leq & C\left(\frac{1}{\lambda}\right)^{p/2}\mathbf{E}\left[%
\int_{0}^{t}2\lambda e^{-2\lambda\left(t-s\right)}\left\vert
D^{\gamma}\Phi\left(s,\cdot\right)\right\vert _{L_{2,p}\left(\mathbf{R}%
^{d}\right)}^{p}ds\right].
\end{eqnarray*}
Now, 
\begin{equation*}
\int_{0}^{T}D\left(t\right)dt\leq C\rho_{\lambda}\mathbf{E}%
\int_{0}^{T}\left\vert D^{\gamma}\Phi\left(s,\cdot\right)\right\vert
_{L_{p,p}\left(\mathbf{R}^{d}\right)}^{p}ds 
\end{equation*}
and 
\begin{equation*}
\int_{0}^{T}B\left(t\right)dt\leq C\rho_{\lambda}^{\frac{p}{2}}\mathbf{E}%
\int_{0}^{T}\left\vert D^{\gamma}\Phi\left(s,\cdot\right)\right\vert
_{L_{2,p}\left(\mathbf{R}^{d}\right)}^{p}ds. 
\end{equation*}
\end{proof}

\begin{lemma}
\label{lem:well-posedness-smooth}For $\nu\in\mathfrak{A}^{\sigma}$, let $f\in%
\mathcal{\mathbb{\tilde{C}}}_{0,p}^{\infty}\left(E\right),g\in\tilde{%
\mathcal{\mathbb{C}}}_{0,p}^{\infty}\left(\mathbf{R}^{d}\right),\Phi\in%
\mathbb{\tilde{\mathcal{\mathbb{C}}}}_{2,p}^{\infty}\left(E\right)\cap{%
\tilde{\mathcal{\mathbb{C}}}}_{p,p}^{\infty}\left(E\right)$ for $p\in\left[%
2,\infty\right)$ and $\Phi\in{\tilde{\mathcal{\mathbb{C}}}}%
_{p,p}^{\infty}\left(E\right)$ for $p\in\left(1,2\right)$, then there is
unique $u\in\tilde{\mathcal{\mathbb{C}}}_{0,p}^{\infty}\left(E\right)$
solving (\ref{eq:mainEq}). Moreover,

\begin{equation*}
u\left(t,x\right)=T_{t}^{\lambda}g\left(x\right)+R_{\lambda}f\left(t,x%
\right)+\tilde{R}_{\lambda}\Phi\left(t,x\right),\left(t,x\right)\in E, 
\end{equation*}
and $u_{1}\left(t,x\right)=T_{t}^{\lambda}g\left(x\right),\left(t,x\right)%
\in E,$ solves (\ref{eq:mainEq}) with $f=0,\Phi=0$, $u_{2}=R_{\lambda}f$
solves (\ref{eq:mainEq}) with $g=0,\Phi=0$, and $u_{3}=\tilde{R}%
_{\lambda}\Phi$ solves \ (\ref{eq:mainEq}) with $g=0,f=0.$
\end{lemma}

\begin{proof}
Uniqueness is a simple repeat of the proof of Lemma 8 in \cite{MikPh1}. We
prove that $u_{1},u_{2}$ solve the corresponding equations by repeating the
proofs of Lemma 8 in \cite{MikPh1} and Lemma 15\ in \cite{MikPh2}. Let $%
\Phi\in\mathbb{\tilde{\mathcal{\mathbb{C}}}}_{2,p}^{\infty}\left(E\right)\cap%
{\tilde{\mathcal{\mathbb{C}}}}_{p,p}^{\infty}\left(E\right)$ if $p\in\left[%
2,\infty\right)$ or $\Phi\in{\tilde{\mathcal{\mathbb{C}}}}%
_{p,p}^{\infty}\left(E\right)$ if $p\in(1,2]$. Recall that $\Phi=\Phi%
\raisebox{2pt}{\ensuremath{\chi}}_{U_{n}}$ for some $U_{n}\in\mathcal{U}$
with $\Pi\left(U_{n}\right)<\infty$. A simple application of Ito formula and
Fubini theorem show that $\mathbf{P}$-a.s. 
\begin{eqnarray*}
& & u_{3}\left(t,x\right) \\
& = &
\int_{0}^{t}e^{-\lambda\left(t-s\right)}\int_{U}\int\Phi\left(s,x+y,z%
\right)P_{t-s}\left(dy\right)q\left(ds,dz\right) \\
& = &
\int_{0}^{t}\int_{U}\Phi\left(s,x,z\right)q\left(ds,dz\right)+\int_{0}^{t}%
\int_{s}^{t}e^{-\lambda\left(r-s\right)}\times \\
& & \int_{U}\int[L^{\nu}\Phi\left(s,x+y,z\right)-\lambda\Phi\left(s,x+y,z%
\right)]P_{r-s}\left(dy\right)drq\left(ds,dz\right) \\
& = &
\int_{0}^{t}\int_{U}\Phi\left(s,x,z\right)q\left(ds,dz\right)+\int_{0}^{t}
\left[L^{\nu}u_{3}\left(s,x\right)-\lambda u_{3}\left(s.x\right)\right]%
ds,\left(t,x\right)\in E.
\end{eqnarray*}
\end{proof}

\subsection{Estimate of $T_{\protect\lambda}g$}

The solution associated to the initial value function is given explicitly by 
\begin{equation*}
T_{\lambda}g=G_{t}^{\lambda}\ast g,\hspace{1em}g\in\mathbb{\tilde{C}}%
_{0,p}^{\infty}\left(\mathbf{R}^{d}\right), 
\end{equation*}
where $G_{t}^{\lambda}\left(x\right):=\exp\left(-\lambda
t\right)p^{\nu^{\ast}}\left(t,x\right)$, $\nu^{\ast}\left(dy\right)=\nu%
\left(-dy\right).$

We prove that there is $C=C\left(\nu,d,p\right)$ so that a.s. 
\begin{equation}
\left\vert L^{\nu}T_{\lambda}g\right\vert _{H_{p}^{\nu;s}\left(E\right)}\leq
C\left\vert g\right\vert _{B_{pp}^{\nu;s+1-1/p}\left(\mathbf{R}^{d}\right)}.
\label{h80}
\end{equation}
Since by Proposition \ref{prop-FuncSpace}, $J_{\nu}^{t}:H_{p}^{\nu;s}\left(%
\mathbf{R}^{d}\right)\rightarrow H_{p}^{\nu;s-t}\left(\mathbf{R}^{d}\right)$
and $J_{\nu}^{t}:B_{pp}^{\nu;s}\left(\mathbf{R}^{d}\right)\rightarrow
B_{pp}^{\nu;s-t}\left(\mathbf{R}^{d}\right)$ are isomorphisms for any $s,t\in%
\mathbf{R}$, it is enough to derive the estimate for $s=0.$

\begin{lemma}
\label{lem:est_initial}Let $\nu\in\mathfrak{A}^{\sigma},w=w_{\nu}$ is a
continuous O-RV function and \textbf{A, B} hold. Then there is $%
C=C\left(\nu,d,p\right)$ so that 
\begin{equation}
\left\vert L^{\nu}T_{\lambda}g\right\vert _{L_{p}\left(E\right)}\leq
C\left\vert g\right\vert _{B_{pp}^{\nu;1-1/p}\left(\mathbf{R}%
^{d}\right)},g\in\mathbb{\tilde{C}}_{0,p}^{\infty}\left(\mathbf{R}%
^{d}\right).  \label{eq:initialEst}
\end{equation}
\end{lemma}

\begin{proof}
We will use an equivalent norm(see Proposition \ref{prop-FuncSpace}.) Let $%
N>1$ be an integer. There exists a function $\phi\in C_{0}^{\infty}(\mathbf{R%
}^{d})$ such that $\mathrm{supp}\,\phi=\{\xi:\frac{1}{N}\leqslant|\xi|%
\leqslant N\}$, $\phi(\xi)>0$ if $N^{-1}<|\xi|<N$ and 
\begin{equation*}
\sum_{j=-\infty}^{\infty}\phi(N^{-j}\xi)=1\quad\text{if }\xi\neq0. 
\end{equation*}
Let 
\begin{equation*}
\tilde{\phi}\left(\xi\right)=\phi\left(N\xi\right)+\phi\left(\xi\right)+\phi%
\left(N^{-1}\xi\right),\xi\in\mathbf{R}^{d}. 
\end{equation*}
Note that supp$~\tilde{\phi}\subseteq\left\{ N^{-2}\leq\left\vert
\xi\right\vert \leq N^{2}\right\} $ and $\tilde{\phi}\phi=\phi$. Let $%
\varphi_{k}=\mathcal{F}^{-1}\phi\left(N^{-k}\cdot\right),k\geq1,$ and $%
\varphi_{0}\in\mathcal{S}\left(\mathbf{R}^{d}\right)$ is defined as 
\begin{equation*}
\varphi_{0}=\mathcal{F}^{-1}\left[1-\sum_{k=1}^{\infty}\phi\left(N^{-k}\cdot%
\right)\right]. 
\end{equation*}
Let $\phi_{0}\left(\xi\right)=\mathcal{F}\varphi_{0}\left(\xi\right),\tilde{%
\phi}_{0}\left(\xi\right)=\mathcal{F}\varphi_{0}\left(\xi\right)+\mathcal{%
F\varphi}_{1}\left(\xi\right),\xi\in\mathbf{R}^{d}\mathbf{,}\tilde{\varphi}=%
\mathcal{F}^{-1}\tilde{\phi},\varphi=\mathcal{F}^{-1}\phi$. Let 
\begin{equation*}
\tilde{\varphi}_{k}=\sum_{l=-1}^{1}\varphi_{k+l},k\geq1,\tilde{\varphi}%
_{0}=\varphi_{0}+\varphi_{1}. 
\end{equation*}
Note that $\varphi_{k}=\tilde{\varphi}_{k}\ast\varphi_{k},k\geq0$. For $%
j\geq1,$ 
\begin{eqnarray*}
& & \mathcal{F}\left[L^{\nu}T_{\lambda}g\left(t,\cdot\right)\ast\varphi_{j}%
\right] \\
& = & w\left(N^{-j}\right)^{-1}\psi^{\tilde{\nu}_{N^{-j}}}\left(N^{-j}\xi%
\right)\exp\left\{ w\left(N^{-j}\right)^{-1}\psi^{\tilde{\nu}%
_{N^{-j}}}\left(N^{-j}\xi\right)t-\lambda t\right\} \\
& & \times\tilde{\phi}\left(N^{-j}\xi\right)\hat{g}_{j}\left(\xi\right),
\end{eqnarray*}
and 
\begin{equation*}
\mathcal{F}\left[L^{\nu}T_{\lambda}g\left(t,\cdot\right)\ast\varphi_{0}%
\right]=\psi^{\nu}\left(\xi\right)\exp\left\{
\psi^{\nu}\left(\xi\right)t-\lambda t\right\} \tilde{\phi}%
_{0}\left(\xi\right)\hat{g}_{0}\left(\xi\right), 
\end{equation*}
where $g_{j}=g\ast\varphi_{j},j\geq0.$

Let $Z^{j}=Z^{\tilde{\nu}_{N^{-j}}},j\geq1.$ Let $\bar{\phi}\in
C_{0}^{\infty}\left(\mathbf{R}^{d}\right),0\notin$supp$\left(\bar{\phi}%
\right)$ and $\bar{\phi}\tilde{\phi}=\tilde{\phi}$, $\bar{\eta}=\mathcal{F}%
^{-1}\bar{\phi}$. Denoting $\eta=\mathcal{F}^{-1}\tilde{\phi}_{0},$ we have 
\begin{eqnarray}
L^{\nu}T_{\lambda}g\left(t,\cdot\right)\ast\varphi_{j} & = &
w\left(N^{-j}\right)^{-1}\bar{H}_{t}^{\lambda,j}\ast g_{j},j\geq1,
\label{for0} \\
L^{\nu}T_{\lambda}g\left(t,\cdot\right)\ast\varphi_{0} & = & \bar{H}%
_{t}^{\lambda,0}\ast g_{0},t>0,  \notag
\end{eqnarray}
where for $j\geq1,$ 
\begin{eqnarray*}
\bar{H}_{t}^{\lambda,j}\left(x\right) & = &
N^{jd}H_{w\left(N^{-j}\right)^{-1}t}^{\lambda,j}\left(N^{j}x\right),%
\left(t,x\right)\in E, \\
H_{t}^{\lambda,j} & = & e^{-\lambda w\left(N^{-j}\right)t}(L^{\tilde{\nu}%
_{N^{-j}}}\bar{\eta})\ast\mathbf{E}\tilde{\varphi}\left(\cdot+Z_{t}^{j}%
\right),t>0\mathbf{,}
\end{eqnarray*}
and 
\begin{equation*}
\bar{H}_{t}^{\lambda,0}\left(x\right)=e^{-\lambda t}L^{\nu}\mathbf{E}%
\eta\left(\cdot+Z_{t}^{\nu}\right),\left(t,x\right)\in E. 
\end{equation*}

By Lemma 5 in \cite{MF}, 
\begin{equation*}
\sup_{j}\int\left\vert L^{\tilde{\nu}_{N^{-j}}}\bar{\eta}\right\vert
dx<\infty\text{.} 
\end{equation*}
Hence by Lemma \ref{Lem:expEst}, 
\begin{eqnarray*}
\int\left\vert H_{t}^{\lambda,j}\right\vert dx & \leq & \int\left\vert L^{%
\tilde{\nu}_{N^{-j}}}\bar{\eta}\right\vert dx\int\left\vert \mathbf{E}\tilde{%
\varphi}\left(\cdot+Z_{t}^{j}\right)\right\vert dx \\
& \leq & Ce^{-ct},t>0,j\geq1,
\end{eqnarray*}
and hence, 
\begin{equation}
\int\left\vert \bar{H}_{t}^{\lambda,j}\right\vert dx\leq C\exp\left\{
-cw\left(N^{-j}\right)^{-1}t\right\} ,t>0,j\geq1.  \label{for1}
\end{equation}
and by Lemma \ref{lem:main_density}, 
\begin{equation}
\int\left\vert \bar{H}_{t}^{\lambda,0}\right\vert dx\leq C\left(\frac{1}{t}%
\wedge1\right),t>0.  \label{h11}
\end{equation}

It follows by Proposition \ref{prop-FuncSpace} and (\ref{for0}) that 
\begin{eqnarray*}
\left\vert L^{\nu}T_{\lambda}g\left(t\right)\right\vert _{L_{p}\left(\mathbf{%
R}^{d}\right)}^{p} & \leq & C\left\vert \left(\sum_{j=1}^{\infty}\left\vert
w\left(N^{-j}\right)^{-1}\bar{H}_{t}^{\lambda,j}\ast g_{j}\right\vert
^{2}\right)^{1/2}\right\vert _{L_{p}\left(\mathbf{R}^{d}\right)}^{p} \\
& & +C\int\left\vert \bar{H}_{t}^{\lambda,0}\ast g_{0}\right\vert ^{p}dx.
\end{eqnarray*}

Hence 
\begin{equation*}
\left\vert L^{\nu}T_{\lambda}g\left(t\right)\right\vert _{L_{p}\left(\mathbf{%
R}^{d}\right)}^{p}\leq C\sum_{j=0}^{\infty}\left\vert
w\left(N^{-j}\right)^{-1}\bar{H}_{t}^{\lambda,j}\ast g_{j}\right\vert
_{L_{p}\left(\mathbf{R}^{d}\right)}^{p}\text{ if }p\in(1,2], 
\end{equation*}
and, by Minkowski inequality, 
\begin{eqnarray*}
\left\vert L^{\nu}T_{\lambda}g\left(t\right)\right\vert _{L_{p}\left(\mathbf{%
R}^{d}\right)}^{p} & \leq & C\left(\sum_{j=1}^{\infty}\left(\int\left\vert
w\left(N^{-j}\right)^{-1}\bar{H}_{t}^{\lambda,j}\ast g_{j}\right\vert
^{p}dx\right)^{2/p}\right)^{p/2} \\
& & +C\int\left\vert \bar{H}_{t}^{\lambda,0}\ast g_{0}\right\vert ^{p}dx
\end{eqnarray*}
if $p>2$. Now, by (\ref{for1}), 
\begin{eqnarray}
& & \int\left\vert w\left(N^{-j}\right)^{-1}\bar{H}_{t}^{\lambda,j}\ast
g_{j}\right\vert ^{p}dx  \label{h12} \\
& \leq & \left(\int\left\vert \bar{H}_{t}^{\lambda,j}\right\vert
dx\right)^{p}\int\left\vert w\left(N^{-j}\right)^{-1}g_{j}\right\vert ^{p}dx
\notag \\
& \leq & Cw\left(N^{-j}\right)^{-p}\exp\left\{
-cw\left(N^{-j}\right)^{-1}t\right\} \left\vert g_{j}\right\vert
_{L_{p}\left(\mathbf{R}^{d}\right)}^{p}\text{ if }j\geq1,  \notag
\end{eqnarray}
and, by (\ref{h11}), 
\begin{equation}
\int\left\vert \bar{H}_{t}^{\lambda,0}\ast g_{0}\right\vert ^{p}dx\leq
C\left(\frac{1}{t}\wedge1\right)^{p}\int\left\vert g_{0}\right\vert ^{p}dx.
\label{h13}
\end{equation}
Therefore for $p\in(1,2],$ 
\begin{equation*}
\int_{0}^{\infty}\left\vert L^{\nu}T_{\lambda}g\left(t\right)\right\vert
_{L_{p}\left(\mathbf{R}^{d}\right)}^{p}dt\leq C\sum_{j=0}^{\infty}\left\vert
w\left(N^{-j}\right)^{-(1-1/p)}\left\vert g_{j}\right\vert _{L_{p}\left(%
\mathbf{R}^{d}\right)}\right\vert ^{p}, 
\end{equation*}
and (\ref{eq:initialEst}) follows by Proposition \ref{prop-FuncSpace}.

Let $p>2.$ In this case, 
\begin{equation*}
\int_{0}^{\infty}\left\vert L^{\nu}T_{\lambda}g\left(t\right)\right\vert
_{L_{p}\left(\mathbf{R}^{d}\right)}^{p}dt\leq C[G+\left\vert
g_{0}\right\vert _{L_{p}\left(\mathbf{R}^{d}\right)}^{p}], 
\end{equation*}
where 
\begin{equation*}
G=\int_{0}^{\infty}\left(\sum_{j=1}^{\infty}\exp\left\{
-cw\left(N^{-j}\right)^{-1}t\right\} k_{j}^{2}\right)^{p/2}dt 
\end{equation*}
with $c>0$ and 
\begin{equation*}
k_{j}=w\left(N^{-j}\right)^{-1}\left\vert g_{j}\right\vert _{L_{p}\left(%
\mathbf{R}^{d}\right)},j\geq1. 
\end{equation*}
Now, let $B=\left\{ j:a\left(t\right)\leq N^{-j}\right\} $ where $%
a\left(t\right)=\inf\left\{ r:w\left(r\right)\geq t\right\} ,t>0$. Then 
\begin{equation*}
\sum_{j=1}^{\infty}e^{-cw\left(N^{-j}\right)^{-1}t}k_{j}^{2}=\sum_{j\in
B}...+\sum_{j\notin B}...=D\left(t\right)+E\left(t\right),t>0. 
\end{equation*}
Let $0<\frac{\beta p}{2}<p_{1}\wedge p_{2}$, by Hölder inequality,

\begin{eqnarray*}
& & D\left(t\right) \\
& \leq & C\sum_{j=1}^{\infty}\chi_{\left\{ j:a\left(t\right)\leq
N^{-j}\right\} }\left(\frac{a\left(t\right)}{N^{-j}}\right)^{\beta}\left(%
\frac{a\left(t\right)}{N^{-j}}\right)^{-\beta}k_{j}^{2} \\
& \leq & C\left(\sum_{j=1}^{\infty}\chi_{\left\{ j:a\left(t\right)\leq
N^{-j}\right\} }\left(\frac{a\left(t\right)}{N^{-j}}\right)^{\beta\frac{p}{%
p-2}}\right)^{1-\frac{2}{p}}\left(\sum_{j=1}^{\infty}\chi_{\left\{
j:j:a\left(t\right)\leq N^{-j}\right\} }\left(\frac{a\left(t\right)}{N^{-j}}%
\right)^{-\beta\frac{p}{2}}k_{j}^{p}\right)^{\frac{2}{p}} \\
& = & CD_{1}^{\frac{p-2}{p}}D_{2}^{\frac{2}{p}}.
\end{eqnarray*}

Denoting $\beta^{\prime}=\beta p/\left(p-2\right),$ we have for $t>0,$ 
\begin{eqnarray*}
D_{1}\left(t\right) & = & \sum_{j=1}^{\infty}\chi_{\left\{ j:\frac{%
a\left(t\right)}{N^{-j}}\leq1\right\} }\left(\frac{a\left(t\right)}{N^{-j}}%
\right)^{\beta^{\prime}} \\
& \leq & C\int_{0}^{\infty}\chi_{\left\{ \frac{a\left(t\right)}{N^{-x}}%
\leq1\right\} }\left(\frac{a\left(t\right)}{N^{-x}}\right)^{\beta^{%
\prime}}dx\leq C\int_{0}^{1}y^{\beta^{\prime}}\frac{dy}{y}<\infty.
\end{eqnarray*}
Applying Corollary \ref{cor:aymp_integral} with $1>\beta\cdot\frac{p}{2}%
\cdot\left(\frac{1}{p_{1}}\vee\frac{1}{p_{2}}\right)$, and arbitrary $\rho>1,
$ 
\begin{eqnarray*}
& & \int_{0}^{\infty}D_{2}^{p/2}dt\leq
C\sum_{j=1}^{\infty}\int_{0}^{\infty}\chi_{\left\{ j:\frac{a\left(t\right)}{%
N^{-j}}\leq1\right\} }\left(\frac{a\left(t\right)}{N^{-j}}\right)^{-\beta%
\frac{p}{2}}k_{j}^{p}dt \\
& \leq & C\sum_{j=1}^{\infty}\int_{0}^{w\left(N^{-j}\right)\rho}\left(\frac{%
a\left(t\right)}{N^{-j}}\right)^{-\beta\frac{p}{2}}dtk_{j}^{p}\leq
C\sum_{j=1}^{\infty}w\left(N^{-j}\right)\left(\frac{a\left(w\left(N^{-j}%
\right)\rho\right)}{N^{-j}}\right)^{-\beta\frac{p}{2}}k_{j}^{p}.
\end{eqnarray*}
Hence 
\begin{equation*}
\int_{0}^{\infty}D_{2}^{p/2}dt\leq
C\sum_{j=1}^{\infty}w\left(N^{-j}\right)k_{j}^{p}. 
\end{equation*}

Now, we estimate the second term \,$E\left(t\right),t>0$. By Hölder
inequality, for $t>0$, 
\begin{eqnarray*}
E\left(t\right) & = &
\sum_{a\left(t\right)>N^{-j}}e^{-cw\left(N^{-j}\right)^{-1}t}w\left(N^{-j}%
\right)^{-2}\left\vert g_{j}\right\vert _{L_{p}}^{2} \\
& \leq & \left(\sum_{a\left(t\right)\geq
N^{-j}}e^{-cw\left(N^{-j}\right)^{-1}t}\right)^{\frac{p-2}{p}%
}\left(\sum_{a\left(t\right)\geq
N^{-j}}e^{-cw\left(N^{-j}\right)^{-1}t}k_{j}^{p}\right)^{\frac{2}{p}}.
\end{eqnarray*}

Changing the variable of integration ($y=\frac{1}{a\left(t\right)N^{x}}$)$,$
we have, by Lemma \ref{lem:powerEstratio}, for some $l>0$, 
\begin{eqnarray*}
& & \sum_{a\left(t\right)\geq
N^{-j}}e^{-cw\left(N^{-j}\right)^{-1}t}\leq\sum_{a\left(t\right)\geq
N^{-j}}\exp\left\{ -c\frac{w\left(a\left(t\right)\right)}{%
w\left(N^{-j}\right)}\right\} \\
& \leq & \sum_{a\left(t\right)\geq N^{-j}}\exp\left\{ -c\left(\frac{%
a\left(t\right)}{N^{-j}}\right)^{l}\right\}
=\sum_{N^{j}a\left(t\right)\geq1}\exp\left\{
-c\left(N^{j}a\left(t\right)\right)^{l}\right\} \\
& \leq & C\int_{N^{x}a\left(t\right)\geq1}\exp\left\{
-c\left(N^{x}a\left(t\right)\right)^{l}\right\}
dx=C\int_{1}^{\infty}\exp\left\{ -cy^{l}\right\} \frac{dy}{y}
\end{eqnarray*}
Hence 
\begin{eqnarray*}
\int_{0}^{\infty}E\left(t\right)^{p/2}dt & \leq &
C\int_{0}^{\infty}\sum_{w\left(N^{-j}\right)^{-1}t\geq1}e^{-cw\left(N^{-j}%
\right)^{-1}t}k_{j}^{p}dt \\
& \leq & C\sum_{j}w\left(N^{-j}\right)k_{j}^{p}\text{.}
\end{eqnarray*}
The estimate (\ref{eq:initialEst}) is proved.
\end{proof}

\subsection{Estimates of $R_{\protect\lambda}f,\tilde{R}_{\protect\lambda%
}\Phi$, verification of Hörmander conditions}

First we show that for each $p>1$ there is $C>0$ so that 
\begin{equation}
\left\vert L^{\nu}R_{\lambda}f\right\vert _{\mathbb{H}_{p}^{\nu;s}\left(E%
\right)}\leq C\left\vert f\right\vert _{\mathbb{H}_{p}^{\nu;s}\left(E%
\right)},\hspace{1em}f\in\tilde{\mathcal{\mathbb{C}}}_{0,p}^{\infty}\left(E%
\right).  \label{f0}
\end{equation}

According to Lemma \ref{lem:powerEstratio}, 
\begin{equation*}
\frac{w\left(\epsilon r\right)}{w\left(r\right)}\leq
C\left(\epsilon^{\alpha_{2}}\vee\epsilon^{\alpha_{1}}\right),r,\epsilon>0. 
\end{equation*}
Thus $w$ is a scaling function with scaling factor $\epsilon^{\alpha_{2}}%
\vee\epsilon^{\alpha_{1}}$ as defined in \cite{MikPh1} and \cite{MikPh2}.
Therefore, we may apply Calderon-Zygmund theorem (see Theorem 5 in \cite%
{MikPh1}) to derive (\ref{f0}). For $p=2,$ the estimate (\ref{f0}) follows
by Plancherel identity (see \cite{MikPh2}). We remind some simple facts that
for a non-decreasing continuous function $w:\left(0,\infty\right)\rightarrow%
\left(0,\infty\right)$ such that $\lim_{\epsilon\rightarrow0}w\left(\epsilon%
\right)=0$ and $\lim_{\epsilon\rightarrow\infty}w\left(\epsilon\right)=\infty
$, denote $a$ to be its generalized inverse, then $a$ is left-hand
continuous and 
\begin{equation*}
a\left(w\left(t\right)\right)\leq t\leq a\left(w\left(t\right)+\right) 
\end{equation*}
and thus, for any $\rho\geq1$, 
\begin{equation*}
a\left(w\left(t\right)\right)\leq t\leq a\left(\rho w\left(t\right)\right) 
\end{equation*}

Moreover, if $a^{-1}$ is a generalized inverse of $a$, i.e. $%
a^{-1}\left(r\right)=\inf\left\{ t>0:a\left(t\right)\geq r\right\} ,r>0$, we
have $a^{-1}\leq w$.

In the next Lemma we show that a version a Hörmander condition, (\ref%
{hormander}), holds and thus a Calderon-Zygmund theorem applies (see Theorem
5 in \cite{MikPh1}); as a result, (\ref{f0}) holds.

\begin{lemma}
\label{Lem:cal1}Let $\nu\in\mathfrak{A}^{\sigma},w=w_{\nu}$ be a continuous
O-RV function and \textbf{A, B} hold. Let $\alpha_{1}>q_{1}\vee q_{2}$ , $%
0<\alpha_{2}<p_{1}\wedge p_{2}$, and 
\begin{eqnarray*}
\alpha_{2} & > & 1\text{ if }\sigma\in(1,2), \\
\alpha_{1} & \leq & 1\text{ if }\sigma\in(0,1),\alpha_{1}\leq2\text{ if }%
\sigma\in\lbrack1,2).
\end{eqnarray*}
Let $\pi\in\mathfrak{A}_{sign}^{\sigma},$ and assume that 
\begin{equation*}
\int_{\left\vert y\right\vert \leq1}\left\vert y\right\vert ^{\alpha_{1}}d%
\widetilde{\left\vert \pi\right\vert }_{R}+\int_{\left\vert y\right\vert
>1}\left\vert y\right\vert ^{\alpha_{2}}d\widetilde{\left\vert
\pi\right\vert }_{R}\leq M,R>0. 
\end{equation*}
Let $\epsilon>0$ and 
\begin{equation*}
K_{\lambda}^{\epsilon}\left(t,x\right)=e^{-\lambda
t}L^{\pi}p^{\nu^{\ast}}\left(t,x\right)\chi_{\left[\epsilon,\infty\right]%
}\left(t\right),t>0,x\in\mathbf{R}^{d}, 
\end{equation*}
where $\nu^{\ast}\left(dy\right)=\nu\left(-dy\right)$. There exist $C_{0}>1$
and $C$ so that 
\begin{equation}
\mathcal{I}=\int\chi_{Q_{C_{0}\eta}\left(0\right)^{c}}\left(t,x\right)\left%
\vert
K_{\lambda}^{\epsilon}\left(t-s,x-y\right)-K_{\lambda}^{\epsilon}\left(t,x%
\right)\right\vert dxdt\leq CM  \label{hormander}
\end{equation}
for all $\left\vert s\right\vert \leq w\left(\eta\right),\left\vert
y\right\vert \leq\eta,\eta>0$, where $Q_{C_{0}\eta}\left(0\right)=\left(-w%
\left(C_{0}\eta\right),w\left(C_{0}\eta\right)\right)\times\left\{
x:\left\vert x\right\vert \leq C_{0}\eta\right\} .$
\end{lemma}

\begin{proof}
Fix $\rho>1$ throughout the proof. By Lemma \ref{lem:powerEstratio} we
choose $C_{0}>3$ such that $w\left(C_{0}\eta\right)>3w\left(\eta\right),%
\eta>0$. We split 
\begin{equation*}
\mathcal{I}=\int_{-\infty}^{2\left\vert s\right\vert
}\int...+\int_{2\left\vert s\right\vert }^{\infty}\int...=\mathcal{I}_{1}+%
\mathcal{I}_{2} 
\end{equation*}

Since $w\left(C_{0}\eta\right)>3w\left(\eta\right),\eta>0,$ it follows by
Lemma \ref{lem:main_density}, Corollary \ref{cor:aymp_integral} and Lemma %
\ref{lem:powerEstratio} for $a$, denoting $k_{0}=C_{0}-1,$ 
\begin{eqnarray*}
\mathcal{I}_{1} & \leq & C\int_{0}^{3\left\vert s\right\vert
}\int_{\left\vert x\right\vert >k_{0}a\left(\left\vert s\right\vert
\right)}\left\vert L^{\pi}p^{\nu^{\ast}}\left(t,x\right)\right\vert dxdt \\
& \leq & CM\frac{1}{a\left(\left\vert s\right\vert \right)^{\alpha_{2}}}%
\int_{0}^{3\left\vert s\right\vert }t^{-1}a\left(t\right)^{\alpha_{2}}dt\leq
CM\frac{a\left(3\left\vert s\right\vert \right)^{\alpha_{2}}}{%
a\left(\left\vert s\right\vert \right)^{\alpha_{2}}} \\
& = & CM.
\end{eqnarray*}

Now,

\begin{eqnarray*}
\mathcal{I}_{2} & \leq & \int_{2\left\vert s\right\vert
}^{\infty}\int\chi_{Q_{C_{0}\eta}^{c}\left(0\right)}\left\vert
L^{\pi}p^{\nu^{\ast}}\left(t-s,x-y\right)-L^{\pi}p^{\nu^{\ast}}\left(t-s,x%
\right)\right\vert dxdt \\
& & +\int_{2\left\vert s\right\vert
}^{\infty}\int\chi_{Q_{C_{0}\eta}^{c}\left(0\right)}|\chi_{\left[%
\epsilon,\infty\right]}\left(t-s\right)L^{\pi}p^{\nu^{\ast}}\left(t-s,x%
\right) \\
& & -\chi_{\left[\epsilon,\infty\right]}\left(t\right)L^{\pi}p^{\nu^{\ast}}%
\left(t,x\right)|dxdt \\
& = & \mathcal{\mathcal{I}}_{2,1}+\mathcal{I}_{2,2}.
\end{eqnarray*}

We split the estimate of $\mathcal{I}_{2,1}$ into two cases.

\emph{Case 1}. Assume $\left\vert y\right\vert \leq a\left(2\left\vert
s\right\vert +\right).$ Then, by Lemma \ref{lem:main_density} , Corollary %
\ref{cor:aymp_integral} , and Lemma \ref{lem:powerEstratio},

\begin{eqnarray*}
\mathcal{I}_{2,1} & \leq & CM\left\vert y\right\vert \int_{2\left\vert
s\right\vert }^{\infty}\left(t-s\right)^{-1}a\left(\left\vert t-s\right\vert
\right)^{-1}dt \\
& = & CM\left\vert y\right\vert a\left(2\left\vert s\right\vert
-s\right)^{-1}\leq CM\left\vert y\right\vert a\left(\left\vert s\right\vert
\right)^{-1} \\
& \leq & CM\frac{\left\vert y\right\vert }{a\left(2\left\vert s\right\vert
+\right)}\frac{a\left(3\left\vert s\right\vert \right)}{a\left(\left\vert
s\right\vert \right)}\leq CM.
\end{eqnarray*}

\emph{Case 2}. Assume $\left\vert y\right\vert >a\left(2\left\vert
s\right\vert +\right)$, i.e. $\eta\geq\left\vert y\right\vert
>a\left(2\left\vert s\right\vert +\right)$ and $a^{-1}\left(\eta\right)\geq
a^{-1}\left(\left\vert y\right\vert \right)\geq2\left\vert s\right\vert $.
We split 
\begin{equation*}
\mathcal{I}_{2,1}=\int_{2\left\vert s\right\vert }^{2\left\vert s\right\vert
+a^{-1}\left(\left\vert y\right\vert \right)}\int...+\int_{2\left\vert
s\right\vert +a^{-1}\left(\left\vert y\right\vert \right)}^{\infty}\int...=%
\mathcal{I}_{2,1,1}+\mathcal{I}_{2,1,2}. 
\end{equation*}

If $2\left\vert s\right\vert \leq t\leq2\left\vert s\right\vert
+a^{-1}\left(\left\vert y\right\vert \right)$, then $0\leq
t\leq3a^{-1}\left(\eta\right)\leq3w\left(\eta\right)\leq
w\left(C_{0}\eta\right)$. Hence $\left\vert x\right\vert >C_{0}\eta\geq
a\left(2\left\vert s\right\vert \right)+\left\vert y\right\vert $ and 
\begin{eqnarray*}
\left\vert x-y\right\vert & \geq & \left(C_{0}-1\right)\eta=k_{0}\eta\geq%
\frac{k_{0}}{2}[a\left(2\left\vert s\right\vert \right)+\left\vert
y\right\vert ] \\
& \geq & a\left(2\left\vert s\right\vert \right)+\left\vert y\right\vert 
\text{ if }\left(t,x\right)\notin Q_{C_{0}\eta}\left(0\right).
\end{eqnarray*}

Also, 
\begin{equation}
2\geq\frac{2\left\vert s\right\vert +a^{-1}\left(\left\vert y\right\vert
\right)}{2\left\vert s\right\vert +a^{-1}\left(\left\vert y\right\vert
\right)-s}\geq\frac{2}{3},  \label{af1}
\end{equation}
and, by Lemma \ref{lem:powerEstratio},

\begin{eqnarray}
& & \frac{a\left(3\left\vert s\right\vert +a^{-1}(\left\vert y\right\vert
)\right)}{a\left(2\left\vert s\right\vert \right)+\left\vert y\right\vert }
\label{af2} \\
& \leq & \frac{a\left(\frac{5}{2}a^{-1}(\left\vert y\right\vert )\right)}{%
a\left(2\left\vert s\right\vert \right)+\left\vert y\right\vert }\leq\frac{%
a\left(\frac{5}{2}a^{-1}\left(\left\vert y\right\vert \right)\right)}{%
a\left(a^{-1}\left(\left\vert y\right\vert \right)\right)}\frac{%
a\left(a^{-1}\left(\left\vert y\right\vert \right)\right)}{\left\vert
y\right\vert }\leq C.  \notag
\end{eqnarray}

By Lemma \ref{lem:main_density}, Corollary \ref{cor:aymp_integral} and (\ref%
{af2}), 
\begin{eqnarray*}
\mathcal{I}_{2,1,1} & \leq & C\int_{2\left\vert s\right\vert }^{2\left\vert
s\right\vert +a^{-1}\left(\left\vert y\right\vert \right)}\int_{\left\vert
x\right\vert >a\left(2\left\vert s\right\vert \right)+\left\vert
y\right\vert }\left\vert L^{\pi}p^{\nu^{\ast}}\left(t-s,x\right)\right\vert
dtdx \\
& \leq & \frac{CM}{[a\left(2\left\vert s\right\vert \right)+\left\vert
y\right\vert ]^{\alpha_{2}}}\int_{0}^{2\left\vert s\right\vert
+a^{-1}\left(\left\vert y\right\vert
\right)}\left(t-s\right)^{-1}a\left(\left\vert t-s\right\vert
\right)^{\alpha_{2}}dt \\
& \leq & CM\frac{a\left(3\left\vert s\right\vert +a^{-1}\left(\left\vert
y\right\vert \right)\right)^{\alpha_{2}}}{[a\left(2\left\vert s\right\vert
\right)+\left\vert y\right\vert ]^{\alpha_{2}}}\leq CM
\end{eqnarray*}

Then, by Lemma \ref{lem:main_density}, Corollary \ref{cor:aymp_integral} and
(\ref{af1})

\begin{eqnarray*}
& & \mathcal{I}_{2,1,2} \\
& \leq & \int_{2\left\vert s\right\vert +a^{-1}\left(\left\vert y\right\vert
\right)}^{\infty}\left[\int_{\mathbf{R}^{d}}\left\vert
L^{\pi}p^{\nu^{\ast}}\left(t-s,x-y\right)-L^{\pi}p^{\nu^{\ast}}\left(t-s,x%
\right)\right\vert dx\right]dt \\
& = & CM\left\vert y\right\vert \int_{2\left\vert s\right\vert
+a^{-1}\left(\left\vert y\right\vert
\right)}^{\infty}\left(t-s\right)^{-1}a\left(\left\vert t-s\right\vert
\right)^{-1}dr \\
& \leq & CM\left\vert y\right\vert a\left(2\left\vert s\right\vert
+a^{-1}\left(\left\vert y\right\vert \right)-s\right)^{-1}\leq CM\left\vert
y\right\vert a\left(a^{-1}\left(\left\vert y\right\vert \right)+\left\vert
s\right\vert \right)^{-1} \\
& \leq & CM\left\vert y\right\vert \left\vert y\right\vert ^{-1}\leq CM
\end{eqnarray*}

Hence, $\mathcal{I}_{2,1}\leq C$.

Finally, by Lemma \ref{lem:main_density}, 
\begin{eqnarray*}
\mathcal{I}_{2,2} & \leq & \int_{2\left\vert s\right\vert
}^{\infty}\int\chi_{Q_{C_{0}\eta}^{c}\left(0\right)}\left\vert
L^{\pi}p^{\nu^{\ast}}\left(t-s,x\right)-L^{\pi}p^{\nu^{\ast}}\left(t,x%
\right)\right\vert dxdt \\
& + & \int_{\varepsilon\vee\left\vert s\right\vert }^{\varepsilon+\left\vert
s\right\vert }\int\left\vert
L^{\pi}p^{\nu^{\ast}}\left(t,x\right)\right\vert dxdt \\
& \leq & CM.
\end{eqnarray*}

The proof is complete.
\end{proof}

\begin{remark}
Although we write Lemma \ref{Lem:cal1} with general $\pi$, in this paper we
only need the result for $\pi=\nu$.
\end{remark}

Now we will show that for $p\in\left[2,\infty\right)$, there is $C$ so that
for all $\Phi\in\mathbb{\tilde{\mathcal{\mathbb{C}}}}_{2,p}^{\infty}\left(E%
\right)\cap{\tilde{\mathcal{\mathbb{C}}}}_{p,p}^{\infty}\left(E\right)$, 
\begin{eqnarray}
& & \left\vert L^{\nu}\tilde{R}^{\lambda}\Phi\right\vert _{\mathbb{H}%
_{p}^{\nu;s}\left(E\right)}\leq C\left[\left\vert \Phi\right\vert _{\mathbb{B%
}_{p,pp}^{\nu;s+1-1/p}\left(E\right)}+\left\vert \Phi\right\vert _{\mathbb{H}%
_{2,p}^{\nu;s+1/2}\left(E\right)}\right],  \label{f1}
\end{eqnarray}
and for $p\in\left(1,2\right)$, there is $C$ so that for all $\Phi\in{\tilde{%
\mathcal{\mathbb{C}}}}_{p,p}^{\infty}\left(E\right)$, 
\begin{eqnarray}
\left\vert L^{\nu}\tilde{R}^{\lambda}\Phi\right\vert _{\mathbb{H}%
_{p}^{\nu;s}\left(E\right)} & \leq & C\left\vert \Phi\right\vert _{\mathbb{B}%
_{p,pp}^{\nu;s+1-1/p}\left(E\right)}.  \label{f2}
\end{eqnarray}

Due to Proposition \ref{prop-FuncSpace}, it is enough to consider the case $%
s=0$. Let $\varepsilon>0,$ 
\begin{equation*}
G_{s,t}^{\lambda,\varepsilon}\left(x\right)=\exp\left(-\lambda\left(t-s%
\right)\right)p^{\nu^{\ast}}\left(t-s,x\right)\chi_{\left[\varepsilon,\infty%
\right]}\left(t-s\right),0<s<t,x\in\mathbf{R}^{d}, 
\end{equation*}
where $\nu^{\ast}\left(dy\right)=\nu\left(-dy\right)$. Denote 
\begin{eqnarray*}
Q\left(t,x\right) & = & \int_{0}^{t}\int_{U}\tilde{\Phi}_{\varepsilon}%
\left(s,x,z\right)q\left(ds,dz\right),\left(t,x\right)\in E, \\
\tilde{\Phi}_{\varepsilon}\left(s,x,z\right) & = &
\int\left(L^{\nu}G_{s,t}^{\lambda,\varepsilon}\right)\left(x-y\right)\Phi%
\left(s,y,z\right)dy,\left(s,x\right)\in E.
\end{eqnarray*}

Obviously, with $K_{\lambda}^{\varepsilon}\left(t,x\right)=e^{-\lambda
t}L^{\nu;\frac{1}{2}}p^{\nu^{\ast}}\left(t,x\right)\chi_{\left[%
\varepsilon,\infty\right]}\left(t\right),t>0,x\in\mathbf{R}^{d},$ we have

\begin{eqnarray}
& & \mathbf{E}\left\vert
\int_{0}^{t}\int_{U}\int\left(L^{\nu}G_{s,t}^{\lambda,\varepsilon}\right)%
\left(x-y\right)\Phi\left(s,y,z\right)dyq\left(ds,dz\right)\right\vert
_{L_{p}\left(E\right)}^{p}  \label{ff1} \\
& = & \mathbf{E}\left\vert \int_{0}^{t}\int_{U}\int
L^{\nu;1/2}G_{s,t}^{\lambda,\varepsilon}\left(x-y\right)L^{\nu;1/2}\Phi%
\left(s,y,z\right)dyq\left(ds,dz\right)\right\vert _{L_{p}\left(E\right)}^{p}
\notag \\
& = & \mathbf{E}\left\vert \int_{0}^{t}\int_{U}\int
K_{\lambda}^{\varepsilon}\left(t-s,x-y\right)L^{\nu;1/2}\Phi\left(s,y,z%
\right)dyq\left(ds,dz\right)\right\vert _{L_{p}\left(E\right)}^{p}.  \notag
\end{eqnarray}

If $2\leq p<\infty$, then

\begin{eqnarray*}
& & \mathbf{E}\int_{0}^{T}\left\vert Q\left(t,\cdot\right)\right\vert
_{L_{p}\left(\mathbf{R}^{d}\right)}^{p}dt \\
& \leq & C\mathbf{E}\left\{ \int_{0}^{T}\left\vert \left[\int_{0}^{t}\int_{U}%
\tilde{\Phi}_{\varepsilon}\left(s,\cdot,z\right)^{2}\Pi\left(dz\right)ds%
\right]^{1/2}\right\vert _{L_{p}\left(\mathbf{R}^{d}\right)}^{p}dt\right\} \\
& + & C\mathbf{E}\left\{ \int_{0}^{T}\int_{0}^{t}\int_{U}\left\vert \tilde{%
\Phi}_{\varepsilon}\left(s,\cdot,z\right)\right\vert _{L_{p}\left(\mathbf{R}%
^{d}\right)}^{p}\Pi\left(dz\right)dsdt\right\} =C\left(\mathbf{E}I_{1}+%
\mathbf{E}I_{2}\right).
\end{eqnarray*}

If $1<p<2$, then

\begin{equation*}
\mathbf{E}\int_{0}^{T}\left\vert Q\left(t,\cdot\right)\right\vert
_{L_{p}\left(\mathbf{R}^{d}\right)}^{p}dt\leq C\mathbf{E}I_{2}. 
\end{equation*}

\emph{Estimate of} $\mathbf{E}I_{2}$. Let $B_{t}^{\lambda}g\left(x%
\right)=e^{-\lambda t}\mathbf{E}g\left(x+Z_{t}^{\nu}\right),\left(t,x\right)%
\in E,g\in\tilde{C}_{0,p}^{\infty}\left(\mathbf{R}^{d}\right)$. Then

\begin{eqnarray*}
I_{2} & = & \int_{0}^{T}\int_{0}^{t}\int_{U}\left\vert \tilde{\Phi}%
_{\varepsilon}\left(s,\cdot,z\right)\right\vert _{L_{p}\left(\mathbf{R}%
^{d}\right)}^{p}\Pi\left(dz\right)dsdt \\
& = & \int_{0}^{T}\int_{0}^{t}\int_{U}\left\vert
\left(L^{\nu}G_{s,t}^{\lambda,\varepsilon}\right)\ast\Phi\left(s,\cdot,z%
\right)\right\vert _{L_{p}\left(\mathbf{R}^{d}\right)}^{p}\Pi\left(dz%
\right)dsdt \\
& \leq & \int_{0}^{T}\int_{0}^{t}\int_{U}\left\vert
L^{\nu}B_{t-s}^{\lambda}\Phi\left(s,\cdot,z\right)\right\vert _{L_{p}\left(%
\mathbf{R}^{d}\right)}^{p}\Pi\left(dz\right)dsdt \\
& = & \int_{U}\int_{0}^{T}\int_{s}^{T}\left\vert
L^{\nu}B_{t-s}^{\lambda}\Phi\left(s,\cdot,z\right)\right\vert _{L_{p}\left(%
\mathbf{R}^{d}\right)}^{p}dtds\Pi\left(dz\right)
\end{eqnarray*}

It follows from Proposition \ref{prop-FuncSpace} and Lemma \ref%
{lem:est_initial} that for $p>1,$

\begin{eqnarray*}
\mathbf{E}I_{2} & \leq & \mathbf{E}\int_{U}\int_{0}^{T}\int_{s}^{T}\left%
\vert L^{\nu}B_{t-s}^{\lambda}\Phi\left(s,\cdot,z\right)\right\vert
_{L_{p}\left(\mathbf{R}^{d}\right)}^{p}dtds\Pi\left(dz\right) \\
& \leq & C\mathbf{E}\int_{U}\int_{0}^{T}\sum_{j=0}^{\infty}\left\vert
w\left(N^{-j}\right)^{-\left(1-1/p\right)}\left\vert
\Phi\left(s,\cdot,z\right)\ast\varphi_{j}\right\vert _{L_{p}\left(\mathbf{R}%
^{d}\right)}\right\vert ^{p}ds\Pi\left(dz\right) \\
& = & C\left\vert \Phi\right\vert _{\mathbb{B}_{p,pp}^{\nu;1-1/p}\left(E%
\right)}^{p}.
\end{eqnarray*}

Hence (\ref{f2}) holds for $p\in\left(1,2\right)$. We prove that for $p\geq2,
$ 
\begin{equation}
\mathbf{E}I_{1}\leq C\left\vert \Phi\right\vert _{\mathbb{H}%
_{2,p}^{\nu;1/2}\left(E\right)},\Phi\in\mathbb{\tilde{\mathcal{\mathbb{C}}}}%
_{2,p}^{\infty}\left(E\right),  \label{f3}
\end{equation}
by verifying Hörmander condition (\ref{hc}) in Lemma \ref{lem:cal2} below.

In the following statement we show that a version a stochastic Hörmander
condition holds which implies (\ref{f3})- see theorem 2.5 of \cite{KK2}. Due
to similarity to Lemma \ref{Lem:cal1} (following the same splitting), we
skip some details.

\begin{lemma}
\label{lem:cal2}Let $\nu\in\mathfrak{A}^{\sigma},w=w_{\nu}$ be a continuous
O-RV function and \textbf{A, B} hold. Let $\epsilon>0$ and 
\begin{equation*}
K_{\lambda}^{\epsilon}\left(t,x\right)=e^{-\lambda t}L^{\nu;\frac{1}{2}%
}p^{\nu^{\ast}}\left(t,x\right)\chi_{\left[\epsilon,\infty\right]%
}\left(t\right),t>0,x\in\mathbf{R}^{d}, 
\end{equation*}
where $\nu^{\ast}\left(dy\right)=\nu\left(-dy\right)$. There exists $C_{0}>0$
and $N>0$ such that for all $\left\vert s\right\vert \leq
w\left(\eta\right),\left\vert y\right\vert \leq\eta,\eta>0$, we have 
\begin{equation}
\mathcal{I}=\int\left[\int\chi_{Q_{C_{0}\eta}\left(0\right)^{c}}\left\vert
K_{\lambda}^{\varepsilon}\left(t-s,x-y\right)-K_{\lambda}^{\varepsilon}%
\left(t,x\right)\right\vert dx\right]^{2}dt\leq N,  \label{hc}
\end{equation}
where $Q_{C_{0}\eta}\left(0\right)=\left(-w\left(C_{0}\eta\right),w%
\left(C_{0}\eta\right)\right)\times\left\{ x:\left\vert x\right\vert
<C_{0}\eta\right\} .$
\end{lemma}

\begin{proof}
By Lemma \ref{lem:powerEstratio} we choose $C_{0}>3$ such that $%
w\left(C_{0}\eta\right)>3w\left(\eta\right),\eta>0$. We split 
\begin{equation*}
\mathcal{I}=\int_{-\infty}^{2\left\vert s\right\vert }\left[\int...\right]%
^{2}dt+\int_{2\left\vert s\right\vert }^{\infty}\left[\int...\right]^{2}dt=%
\mathcal{I}_{1}+\mathcal{I}_{2}. 
\end{equation*}
Since $w\left(C_{0}\eta\right)>3w\left(\eta\right),\eta>0$, it follows by
Lemma \ref{lem:fracdensity}, Lemma \ref{lem:powerEstratio} and Corollary \ref%
{cor:aymp_integral} with $k_{0}=C_{0}-1$ and $\beta\in(0,\frac{\alpha_{2}}{2}%
)$,\ 
\begin{eqnarray*}
\left\vert \mathcal{I}_{1}\right\vert & \leq & C\int_{0}^{3\left\vert
s\right\vert }\left[\int_{\left\vert x\right\vert >k_{0}a\left(\left\vert
s\right\vert \right)}\left\vert L^{\nu;\frac{1}{2}}p^{\nu^{\ast}}\left(t,x%
\right)\right\vert dx\right]^{2}dt \\
& \leq & C\int_{0}^{3\left\vert s\right\vert }\left(t^{-\frac{1}{2}%
}a\left(t\right)^{\beta}\left(k_{0}a\left(\left\vert s\right\vert
\right)\right)^{-\beta}\right)^{2}dt\leq Ca\left(\left\vert s\right\vert
\right)^{-2\beta}\int_{0}^{3\left\vert s\right\vert }a\left(t\right)^{2\beta}%
\frac{dt}{t} \\
& \leq & Ca\left(s\right)^{-2\beta}a\left(3\left\vert s\right\vert
\right)^{2\beta}\leq C
\end{eqnarray*}
Now, 
\begin{eqnarray*}
& & \left\vert \mathcal{I}_{2}\right\vert \\
& \leq & 2\int_{2\left\vert s\right\vert }^{\infty}\left[\int\chi_{Q_{C_{0}%
\eta}^{c}\left(0\right)}\left\vert L^{\nu;\frac{1}{2}}p^{\nu^{\ast}}%
\left(t-s,x-y\right)-L^{\nu;\frac{1}{2}}p^{\nu^{\ast}}\left(t-s,x\right)%
\right\vert dx\right]^{2}dt \\
& + & 2\int_{2\left\vert s\right\vert }^{\infty}\left\{
\int\chi_{Q_{C_{0}\eta}^{c}\left(0\right)}|\chi_{\left[\varepsilon,\infty%
\right]}\left(t-s\right)L^{\nu;\frac{1}{2}}p^{\nu^{\ast}}\left(t-s,x\right)-%
\chi_{\left[\varepsilon,\infty\right]}\left(t\right)L^{\nu;\frac{1}{2}%
}p^{\nu^{\ast}}\left(t,x\right)|dx\right\} ^{2}dt \\
& = & \mathcal{I}_{2,1}+\mathcal{I}_{2,2}.
\end{eqnarray*}
We split the estimate of $\mathcal{I}_{2,1}$ into two cases.

\emph{Case 1.} Assume $\left\vert y\right\vert \leq a\left(2\left\vert
s\right\vert +\right)$. Then by Lemma \ref{lem:fracdensity}, Lemma \ref%
{lem:powerEstratio} and Corollary \ref{cor:aymp_integral} 
\begin{eqnarray*}
\mathcal{I}_{2,1} & \leq & C\int_{2\left\vert s\right\vert }^{\infty}\frac{%
\left\vert y\right\vert ^{2}}{(t-s)a\left(t-s\right)^{2}}dt \\
& \leq & C\left\vert y\right\vert ^{2}\int_{2\left\vert s\right\vert
}^{\infty}\frac{1}{a\left(t-s\right)^{2}}\left(t-s\right)^{-1}dt \\
& \leq & C\left\vert y\right\vert ^{2}\int_{\left\vert s\right\vert
}^{\infty}a\left(t\right)^{-2}\frac{dt}{t}\leq C\left\vert y\right\vert
^{2}a\left(\left\vert s\right\vert \right)^{-2}\leq C
\end{eqnarray*}

\emph{Case 2.} Assume $\left\vert y\right\vert >a\left(2\left\vert
s\right\vert +\right)$ i.e. $\eta\geq\left\vert y\right\vert
>a\left(2\left\vert s\right\vert +\right)$ and $a^{-1}\left(\eta\right)\geq
a^{-1}\left(\left\vert y\right\vert \right)\geq2\left\vert s\right\vert $.
We split 
\begin{equation*}
\mathcal{I}_{2,1}=\int_{2\left\vert s\right\vert }^{2\left\vert s\right\vert
+a^{-1}\left(\left\vert y\right\vert \right)}\left[\int...\right]%
^{2}+\int_{2\left\vert s\right\vert +a^{-1}\left(\left\vert y\right\vert
\right)}^{\infty}\left[\int...\right]^{2}=\mathcal{I}_{2,1,1}+\mathcal{I}%
_{2,1,2}. 
\end{equation*}
Hence, with $\beta\in(0,\frac{\alpha_{2}}{2})$, by Lemma \ref%
{lem:fracdensity}, Lemma \ref{lem:powerEstratio} and Corollary \ref%
{cor:aymp_integral}, 
\begin{eqnarray*}
\mathcal{I}_{2,1,1} & \leq & C\int_{2\left\vert s\right\vert }^{2\left\vert
s\right\vert +a^{-1}\left(\left\vert y\right\vert \right)}\left[%
\int_{\left\vert x\right\vert >a\left(2\left\vert s\right\vert
\right)+\left\vert y\right\vert }\left\vert L^{\nu;\frac{1}{2}%
}p^{\nu^{\ast}}\left(t-s,x\right)\right\vert dx\right]^{2}dt \\
& \leq & \frac{C}{\left(a\left(2\left\vert s\right\vert \right)+\left\vert
y\right\vert \right)^{2\beta}}\int_{2\left\vert s\right\vert }^{2\left\vert
s\right\vert +a^{-1}\left(\left\vert y\right\vert
\right)}a\left(t-s\right)^{2\beta}\frac{dt}{\left(t-s\right)} \\
& \leq & \frac{C}{\left(a\left(2\left\vert s\right\vert \right)+\left\vert
y\right\vert \right)^{2\beta}}a\left(3\left\vert s\right\vert
+a^{-1}\left(\left\vert y\right\vert \right)\right)^{2\beta}\leq C
\end{eqnarray*}
Then by Lemma \ref{lem:fracdensity} and Corollary \ref{cor:aymp_integral}, 
\begin{eqnarray*}
\mathcal{I}_{2,1,2} & \leq & C\int_{2\left\vert s\right\vert
+a^{-1}\left(\left\vert y\right\vert \right)}^{\infty}\left[\int_{\mathbf{R}%
^{d}}\left\vert L^{\nu;\frac{1}{2}}p^{\nu^{\ast}}\left(t-s,x-y\right)-L^{\nu;%
\frac{1}{2}}p^{\nu^{\ast}}\left(t-s,x\right)\right\vert dx\right]^{2}dt \\
& \leq & C\int_{2\left\vert s\right\vert +a^{-1}\left(\left\vert
y\right\vert \right)}^{\infty}\frac{\left\vert y\right\vert ^{2}}{%
a\left(t-s\right)^{2}}\frac{dt}{t-s}\leq C\left\vert y\right\vert
^{2}a\left(2\left\vert s\right\vert +a^{-1}\left(\left\vert y\right\vert
\right)-s\right)^{-2} \\
& \leq & C\left\vert y\right\vert ^{2}a\left(\left\vert s\right\vert
+a^{-1}\left(\left\vert y\right\vert \right)\right)^{-2}\leq C
\end{eqnarray*}
Hence, $\mathcal{I}_{2,1}\leq C$. Since 
\begin{eqnarray*}
\mathcal{I}_{2,2} & \leq & \int_{2\left\vert s\right\vert }^{\infty}\left[%
\int\chi_{Q_{C_{0}\eta}^{c}\left(0\right)}\left\vert L^{\nu;\frac{1}{2}%
}p^{\nu^{\ast}}\left(t-s,x\right)-L^{\nu;\frac{1}{2}}p^{\nu^{\ast}}\left(t,x%
\right)\right\vert dx\right]^{2}dt \\
& + & \int_{\varepsilon\vee\left\vert s\right\vert }^{\varepsilon+\left\vert
s\right\vert }\left[\int\left\vert L^{\nu;\frac{1}{2}}p^{\nu^{\ast}}%
\left(t,x\right)\right\vert dx\right]^{2}dt \\
& = & \mathcal{I}_{2,2,1}+\mathcal{I}_{2,2,2},
\end{eqnarray*}
it follows by Lemma \ref{lem:fracdensity} that 
\begin{eqnarray*}
\mathcal{I}_{2,2,1} & \leq & \int_{2\left\vert s\right\vert }^{\infty}\left[%
\int\chi_{Q_{C_{0}\eta}^{c}\left(0\right)}\left\vert L^{\nu;\frac{1}{2}%
}p^{\nu^{\ast}}\left(t-s,x\right)-L^{\nu;\frac{1}{2}}p^{\nu^{\ast}}\left(t,x%
\right)\right\vert dx\right]^{2}dt \\
& \leq & C.
\end{eqnarray*}
Clearly, $\mathcal{I}_{2,2,2}\leq C$ as well.
\end{proof}

\begin{remark}
\label{rem:est_for_smooth}The estimates in Lemmas \ref{Lem:cal1} and \ref%
{lem:cal2} are independent of $\epsilon>0$. Now, Lemmas \ref{lem:est-smooth}%
- \ref{lem:cal2}, (\ref{f0}), (\ref{f1}), and (\ref{f2}) allow us to derive
all the estimates of $\left\vert u\right\vert _{L_{P}\left(E\right)}$ and $%
\left\vert L^{\nu}u\right\vert _{L_{P}\left(E\right)}$ ($s=0$) claimed in
Theorem \ref{thm:main1} for the solution $u$ of (\ref{eq:mainEq}) with
smooth input functions.
\end{remark}

We now prove the main theorem for general imputs.

\subsection{Proof of the main Theorem}

We finish the proof of Theorem \ref{thm:main1} in a standard way. Since by
Proposition \ref{prop-FuncSpace}, $J_{\nu}^{t}:\mathbb{H}_{p}^{\nu;s}\left(%
\mathbf{R}^{d},l_{2}\right)\rightarrow\mathbb{H}_{p}^{\nu;s-t}\left(\mathbf{R%
}^{d},l_{2}\right)$, $J_{\nu}^{t}:\mathbb{H}_{p}^{\nu;s}\left(\mathbf{R}%
^{d}\right)\rightarrow\mathbb{H}_{p}^{\nu;s-t}\left(\mathbf{R}^{d}\right)$
and $J_{\nu}^{t}:\mathbb{B}_{pp}^{\nu;s}\left(\mathbf{R}^{d}\right)%
\rightarrow\mathbb{B}_{pp}^{\nu;s-t}\left(\mathbf{R}^{d}\right)$ is an
isomorphism for any $s,t\in\mathbf{R}$, it is enough to derive the statement
for $s=0.$ Let $f\in\mathbb{L}_{p}\left(E\right),g\in\mathbb{B}%
_{pp}^{\nu;1-1/p}\left(\mathbf{R}^{d}\right),$ and 
\begin{eqnarray*}
\Phi & \in & \mathbb{B}_{p,pp}^{\nu;1-1/p}\left(E\right)\cap\mathbb{H}%
_{2,p}^{\nu;\frac{1}{2}}\left(E\right)\text{ if }p\geq2, \\
\Phi & \in & \mathbb{B}_{p,pp}^{\nu;1-1/p}\left(E\right)\text{ if }%
p\in\left(1,2\right).
\end{eqnarray*}

According to Lemma \ref{lem-denseSubspace}, there are sequences $f_{n}\in%
\tilde{\mathbb{C}}_{0,p}^{\infty}\left(E\right),g_{n}\in\tilde{\mathbb{C}}%
_{0,p}^{\infty}\left(\mathbf{R}^{d}\right),\Phi_{n}\in\tilde{\mathbb{C}}%
_{2,p}^{\infty}\left(E\right)\cap\tilde{\mathbb{C}}_{p,p}^{\infty}\left(E%
\right)$ if $p\geq2$, and $\Phi_{n}\in\tilde{\mathbb{C}}_{p,p}^{\infty}%
\left(E\right)$ if $p\in\left(1,2\right)$, such that 
\begin{equation*}
f_{n}\rightarrow f\text{ in }\mathbb{L}_{p}\left(E\right),g_{n}\rightarrow g%
\text{ in }\mathbb{B}_{pp}^{\nu;1-1/p}\left(\mathbf{R}^{d}\right), 
\end{equation*}
and 
\begin{eqnarray*}
\Phi_{n} & \rightarrow & \Phi\text{ in }\mathbb{B}_{p,pp}^{\nu;1-1/p}\left(E%
\right)\cap\mathbb{H}_{2,p}^{\nu;\frac{1}{2}}\left(E\right)\text{ if }p\geq2,
\\
\Phi_{n} & \rightarrow & \Phi\text{ in }\mathbb{B}_{p,pp}^{\nu;1-1/p}\left(E%
\right)\text{ if }p\in\left(1,2\right).
\end{eqnarray*}
For each $n,$ there is unique $u_{n}\in\tilde{\mathbb{C}}_{0,p}^{\infty}%
\left(E\right)$ solving (\ref{eq:mainEq}). Hence for $u_{n,m}=u_{n}-u_{m,}$
we have 
\begin{eqnarray*}
\partial_{t}u_{n,m} & = &
\left(L^{\nu}-\lambda\right)u_{n,m}+f_{n}-f_{m}+\int_{U}\left(\Phi_{n}-%
\Phi_{m}\right)q\left(dt,dz\right), \\
u_{n,m}\left(0,x\right) & = & g_{n}\left(x\right)-g_{m}\left(x\right),x\in%
\mathbf{R}^{d}.
\end{eqnarray*}
By estimates in Theorem \ref{thm:main1} for smooth inputs(see Remark \ref%
{rem:est_for_smooth}), 
\begin{eqnarray*}
& & \left\vert L^{\nu}u_{n,m}\right\vert _{\mathbb{L}_{p}\left(E\right)}\leq
C[\left\vert f_{n}-f_{m}\right\vert _{\mathbb{L}_{p}\left(E\right)}+\left%
\vert g_{n}-g_{m}\right\vert _{\mathbb{B}_{pp}^{\nu;1-1/p}\left(\mathbf{R}%
^{d}\right)} \\
& & +\left\vert \Phi_{n}-\Phi_{m}\right\vert _{\mathbb{B}_{p,pp}^{\nu;1-1/p}%
\left(E\right)}+\left\vert \Phi_{n}-\Phi_{m}\right\vert _{\mathbb{H}%
_{2,p}^{\nu;1/2}\left(E\right)}],
\end{eqnarray*}
if $p>2$ and 
\begin{eqnarray*}
& & \left\vert L^{\nu}u_{n,m}\right\vert _{\mathbb{L}_{p}\left(E\right)}\leq
C[\left\vert f_{n}-f_{m}\right\vert _{\mathbb{L}_{p}\left(E\right)}+\left%
\vert g_{n}-g_{m}\right\vert _{\mathbb{B}_{pp}^{\nu;1-1/p}\left(\mathbf{R}%
^{d}\right)} \\
& & +\left\vert \Phi_{n}-\Phi_{m}\right\vert _{\mathbb{B}_{p,pp}^{\nu;1-1/p}%
\left(E\right)}.
\end{eqnarray*}
if $p\in\left(1,2\right)$.

By Lemma \ref{lem:est-smooth}, 
\begin{eqnarray*}
\left\vert u_{n,m}\right\vert _{\mathbb{L}_{p}\left(E\right)} & \leq &
C[\rho_{\lambda}\left\vert f_{n}-f_{m}\right\vert _{\mathbb{L}%
_{p}\left(E\right)}+\rho_{\lambda}^{1/p}\left\vert g_{n}-g_{m}\right\vert _{%
\mathbb{L}_{p}\left(\mathbf{R}^{d}\right)} \\
& & +\rho_{\lambda}^{1/p}\left\vert \Phi_{n}-\Phi_{m}\right\vert _{\mathbb{L}%
_{p,p}\left(E\right)}+\rho_{\lambda}^{1/2}\left\vert
\Phi_{n}-\Phi_{m}\right\vert _{\mathbb{L}_{2,p}\left(E\right)}]
\end{eqnarray*}
if $p\geq2$, and 
\begin{eqnarray*}
\left\vert u_{n,m}\right\vert _{\mathbb{L}_{p}\left(E\right)} & \leq &
C[\rho_{\lambda}\left\vert f_{n}-f_{m}\right\vert _{\mathbb{L}%
_{p}\left(E\right)}+\rho_{\lambda}^{1/p}\left\vert g_{n}-g_{m}\right\vert _{%
\mathbb{L}_{p}\left(\mathbf{R}^{d}\right)} \\
& & +\rho_{\lambda}^{1/p}\left\vert \Phi_{n}-\Phi_{m}\right\vert _{\mathbb{L}%
_{p,p}\left(E\right)}]
\end{eqnarray*}
if $p\in\left(1,2\right)$. Hence there is $u\in\mathbb{H}_{p}^{\nu;1}\left(E%
\right)$ so that $u_{n}\rightarrow u$ in $\mathbb{H}_{p}^{\nu;1}\left(E%
\right)$. Moreover, by Lemma \ref{lem:est-smooth}, 
\begin{equation}
\sup_{t\leq T}\left\vert u_{n}\left(t\right)-u\left(t\right)\right\vert _{%
\mathbb{L}_{p}\left(\mathbf{R}^{d}\right)}\rightarrow0,  \label{fo7}
\end{equation}
and $u$ is $\mathbb{L}_{p}\left(\mathbf{R}^{d}\right)$-valued continuous.

By Lemma \ref{lem-stochInt} (see Appendix) and Remark \ref{rem:embedding}, 
\begin{equation}
\sup_{t\leq T}\left\vert
\int_{0}^{t}\int_{U}\Phi_{n}q\left(ds,dz\right)-\int_{0}^{t}\int_{U}\Phi
q\left(ds,dz\right)\right\vert _{L_{p}\left(\mathbf{R}^{d}\right)}%
\rightarrow0  \label{fo80}
\end{equation}
as $n\rightarrow\infty$ in probability.

Hence (see (\ref{fo7})-(\ref{fo80})) we can pass to the limit in the
equation 
\begin{equation}
u_{n}\left(t\right)=g_{n}+\int_{0}^{t}[L^{\nu}u_{n}\left(s\right)-\lambda
u_{n}\left(s\right)+f_{n}\left(s\right)]ds+\int_{0}^{t}\int_{U}\Phi_{n}q%
\left(ds,dz\right),0\leq t\leq T.  \label{fo9}
\end{equation}
Obviously, (\ref{fo9}) holds for $u,g$ and $f,\Phi$. We proved the existence
part of Theorem \ref{thm:main1}.

\emph{Uniqueness. }Assume $u_{1},u_{2}\in\mathbb{H}_{p}^{\nu;1}\left(E\right)
$ solve (\ref{eq:mainEq})$.$ Then $u=u_{1}-u_{2}\in\mathbb{H}%
_{p}^{\nu;1}\left(E\right)$ solves (\ref{eq:mainEq}) with $f=0,g=0,\Phi=0$.
Thus the uniqueness follows from the uniqueness of a deterministic equation
(see \cite{MikPh2}).

Theorem \ref{thm:main1} is proved.

\section{Appendix}

\subsection{A non-degeneracy estimate}

We present a sufficient condition for our assumption \textbf{B}. For the
sake of completeness we add the proof from \cite{MF}, see Corollary 6{,}
with an obvious modification.

\begin{lemma}
\label{cor:Example1}(Corollary 6 of \cite{MF}) Let $\nu\in\mathfrak{A}%
^{\sigma}$, 
\begin{equation*}
\nu\left(\Gamma\right)=-\int_{0}^{\infty}\int_{S_{d-1}}\chi_{\Gamma}\left(rz%
\right)\Pi\left(r,dz\right)d\delta\left(r\right),\Gamma\in\mathcal{B}\left(%
\mathbf{R}_{0}^{d}\right), 
\end{equation*}
where $\delta=\delta_{\nu},\Pi\left(r,dz\right),r>0$ is a measurable family
of measures on $S_{d-1}$ with $\Pi\left(r,S_{d-1}\right)=1,r>0$. Assume that 
$w=w_{\nu}=\delta_{\nu}^{-1}$ is an O-RV function at zero and infinity
satisfying 
\begin{equation*}
\inf_{\left\vert \hat{\xi}\right\vert =1}\int_{S_{d-1}}\left\vert \hat{\xi}%
\cdot z\right\vert ^{2}\Pi\left(r,dz\right)\geq c_{0}>0,\hspace{1em}r>0 
\end{equation*}
and $\left\vert \left\{ s\in\left[0,1\right]:r_{i}\left(s\right)<1\right\}
\right\vert >0,i=1,2.$ Then assumption $\mathbf{B}$ holds. That is 
\begin{equation*}
\inf_{R\in\left(0,\infty\right),\left\vert \hat{\xi}\right\vert
=1}\int_{\left\vert y\right\vert \leq1}\left\vert \hat{\xi}\cdot
y\right\vert ^{2}\tilde{\nu}_{R}\left(dy\right)>0 
\end{equation*}
\end{lemma}

\begin{proof}
Indeed, for $\left\vert \hat{\xi}\right\vert =1,R\mathbf{>}0,$ with $C>0$, 
\begin{align*}
& \int_{\left\vert y\right\vert \leq1}\left\vert \hat{\xi}\cdot y\right\vert
^{2}\nu_{R}\left(dy\right) \\
= & R^{-2}\int_{\left\vert y\right\vert \leq R}\left\vert \hat{\xi}\cdot
y\right\vert
^{2}\nu\left(dy\right)=-R^{-2}\int_{0}^{R}\int_{S_{d-1}}\left\vert \hat{\xi}%
\cdot z\right\vert ^{2}\Pi\left(r,dz\right)r^{2}d\delta\left(r\right) \\
\geq &
-R^{-2}c_{0}\int_{0}^{R}r^{2}d\delta\left(r\right)=c_{0}R^{-2}\int_{\left%
\vert y\right\vert \leq R}\left\vert y\right\vert
^{2}\nu\left(dy\right)=c_{0}\int_{\left\vert y\right\vert \leq1}\left\vert
y\right\vert ^{2}\nu_{R}\left(dy\right)
\end{align*}

Therefore, $\int_{\left\vert y\right\vert \leq1}\left\vert \hat{\xi}\cdot
y\right\vert ^{2}\tilde{\nu}_{R}\left(dy\right)\geq c_{0}\int_{\left\vert
y\right\vert \leq1}\left\vert y\right\vert ^{2}\tilde{\nu}%
_{R}\left(dy\right).$

Using computation similar to the proof of Lemma \ref{lem:alpha12}, 
\begin{align*}
\int_{\left\vert y\right\vert \leq1}\left\vert y\right\vert ^{2}\tilde{\nu}%
_{R}\left(dy\right) & =w\left(R\right)\int_{\left\vert y\right\vert
\leq1}\left\vert y\right\vert ^{2}\nu_{R}\left(dy\right) \\
& =2R^{-2}\int_{0}^{R}s^{2}\left[\frac{w\left(R\right)}{w\left(s\right)}-1%
\right]\frac{ds}{s} \\
& =2\int_{0}^{1}s^{2}\left[\frac{w\left(R\right)}{w\left(Rs\right)}-1\right]%
\frac{ds}{s}
\end{align*}

By Fatou's Lemma, 
\begin{eqnarray*}
\liminf_{R\rightarrow0}\int_{\left\vert y\right\vert \leq1}\left\vert
y\right\vert ^{2}\tilde{\nu}_{R}\left(dy\right) & \geq & 2\int_{0}^{1}s^{2}%
\left[\frac{1}{r_{1}\left(s\right)}-1\right]\frac{ds}{s}=c_{1}>0, \\
\liminf_{R\rightarrow\infty}\int_{\left\vert y\right\vert \leq1}\left\vert
y\right\vert ^{2}\tilde{\nu}_{R}\left(dy\right) & \geq & 2\int_{0}^{1}s^{2}%
\left[\frac{1}{r_{2}\left(s\right)}-1\right]\frac{ds}{s}=c_{2}>0,
\end{eqnarray*}

if $\left\vert \left\{ s\in\left[0,1\right]:r_{i}\left(s\right)<1\right\}
\right\vert >0,i=1,2,$ completing the proof.
\end{proof}

\subsection{Stochastic integral}

We discuss here the definition of stochastic integrals with respect to a
martingale measure. Let $\left(\Omega,\mathcal{F},\mathbf{P}\right)$ be a
complete probability space with a filtration of $\sigma-$algebras on $%
\mathbb{F}=\left(\mathcal{F}_{t},t\geq0\right)$ satisfying the usual
conditions. Let $\left(U,\mathcal{U},\Pi\right)$ be a measurable space with $%
\sigma-$finite measure $\Pi$, $\mathbf{R}_{0}^{d}=\mathbf{R}%
^{d}\backslash\left\{ 0\right\} $. Let $p\left(dt,dz\right)$ be $\mathbb{F}-$%
adapted point measures on $\left(\left[0,\infty\right)\times U,\mathcal{B}%
\left(\left[0,\infty\right)\right)\otimes\mathcal{U}\right)$ with
compensator $\Pi\left(dz\right)dt.$ We denote the martingale measure $%
q\left(dt,dz\right)=p\left(dt,dz\right)-\Pi\left(dz\right)dt$.

We prove the following based on Lemma 12 from \cite{MP1}.

\begin{lemma}
\label{lem-stochInt}Let $s\in\mathbf{R}$, $\Phi\in\mathbb{H}%
_{2,p}^{s}\left(E\right)\cap\mathbb{H}_{p,p}^{s}\left(E\right)$ for $p\in%
\left[2,\infty\right)$ and $\Phi\in\mathbb{H}_{p,p}^{s}\left(E\right)$ for $%
p\in\left[1,2\right)$. There is a unique cadlag $H_{p}^{s}\left(\mathbf{R}%
^{d}\right)-$valued process 
\begin{equation*}
M\left(t\right)=\int_{0}^{t}\int\Phi\left(r,x,z\right)q\left(dr,dz\right),0%
\leq t\leq T,x\in\mathbf{R}^{d}, 
\end{equation*}
such that for every $\varphi\in\mathcal{S}\left(\mathbf{R}^{d}\right)$, 
\begin{equation}
\left\langle M\left(t\right),\varphi\right\rangle
=\int_{0}^{t}\int\left(\int J^{s}\Phi\left(r,\cdot,z\right)J^{-s}\varphi
dx\right)q\left(dr,dz\right),0\leq t\leq T.  \label{eq:unique}
\end{equation}
Moreover, there is a constant $C$ independent of $\Phi$ such that 
\begin{eqnarray}
& & \mathbf{E}\sup_{t\leq T}\left\vert
\int_{0}^{t}\int\Phi\left(r,\cdot,z\right)q\left(dr,dz\right)\right\vert
_{H_{p}^{s}\left(\mathbf{R}^{d}\right)}  \label{eq:normEst} \\
& \leq & C\sum_{j=2,p}\left\vert \Phi\right\vert _{\mathbb{H}%
_{j,p}^{s}\left(E\right)},p\geq2,  \notag \\
& & \mathbf{E}\sup_{t\leq T}\left\vert
\int_{0}^{t}\int\Phi\left(r,\cdot,z\right)q\left(dr,dz\right)\right\vert
_{H_{p}^{s}\left(\mathbf{R}^{d}\right)}  \notag \\
& \leq & C\left\vert \Phi\right\vert _{\mathbb{H}_{p,p}^{s}\left(E\right)},p%
\in\left(1,2\right).  \notag
\end{eqnarray}
\end{lemma}

\begin{proof}
According to Proposition \ref{prop-FuncSpace}, it is enough to consider the
case $s=0$. Let $\Phi_{n}$ be a sequence defined in Lemma \ref%
{lem-approximation1} that approximates $\Phi$. Note first that by Lemma \ref%
{lem-approximation1}, for all $x$, 
\begin{equation*}
\mathbf{E}\int_{0}^{T}\sup_{x}\int\left\vert
D^{\gamma}\Phi_{n}\left(r,x,z\right)\right\vert
^{p}\Pi\left(dz\right)dr<\infty. 
\end{equation*}
Recall for each $n$, $\ \,$we have $\Phi_{n}=\Phi_{n}\raisebox{2pt}{%
\ensuremath{\chi}}_{U_{n}}$ for some $U_{n}\in\mathcal{U}$ with $%
\Pi\left(U_{n}\right)<\infty$. Consequently, we define for each $x\in\mathbf{%
R}^{d}$ and $\mathbf{P}$-a.s. for all $\left(t,x\right)\in E,$ 
\begin{eqnarray*}
M_{n}\left(t,x\right) & = &
\int_{0}^{t}\int\Phi_{n}\left(r,x,z\right)q\left(dr,dz\right) \\
& = &
\int_{0}^{t}\int\Phi_{n}\left(r,x,z\right)p\left(dr,dz\right)-\int_{0}^{t}%
\int\Phi_{n}\left(r,x,z\right)\Pi\left(dz\right)dr.
\end{eqnarray*}
Obviously, $M_{n}\left(t,x\right)$ is cadlag in $t$ and infinitely
differentiable in $x$. Obviously,$M_{n}\left(t\right)=M_{n}\left(t,\cdot%
\right)$ is $\mathbb{L}_{p}\left(\mathbf{R}^{d}\right)$-valued cadlag and,
by Kunita's inequality (Theorem 2.11 of \cite{ku}), there is a constant $C$
independent of $\Phi_{n}$ such that 
\begin{equation}
\mathbf{E}\sup_{t\leq T}\left\vert M_{n}\left(t\right)\right\vert
_{L_{p}\left(\mathbf{R}^{d}\right)}^{p}\leq C\mathbf{E}\sum_{j=2,p}\left%
\vert \Phi_{n}\right\vert _{L_{j,p}\left(E\right)}^{p},2\leq p<\infty.
\label{f1-1}
\end{equation}
By BDG inequality, for $1<p<2$, 
\begin{align*}
\mathbf{E}\sup_{t\leq T}\left\vert M_{n}\left(t,x\right)\right\vert ^{p} &
\leq C\mathbf{E}\left[\text{$\left(\int_{0}^{T}\int\text{$%
\left|\Phi_{n}\left(r,x,z\right)\right|$}^{2}p\text{$\left(dr,dz\right)$}%
\right)$}^{\frac{p}{2}}\right] \\
& \leq C\mathbf{E}\left[\text{$\int_{0}^{T}\int\text{$\left|\Phi_{n}%
\left(r,x,z\right)\right|$}^{p}$}p\left(dr,dz\right)\right] \\
& =C\mathbf{E}\left[\text{$\int_{0}^{T}\int\text{$\left|\Phi_{n}\left(r,x,z%
\right)\right|$}^{p}$}dr\Pi\left(dz\right)\right]
\end{align*}

Hence, by integrating in $x$,

\begin{equation}
\mathbf{E}\sup_{t\leq T}\left\vert M_{n}\left(t\right)\right\vert
_{L_{p}\left(\mathbf{R}^{d}\right)}^{p}\leq C\mathbf{E}\left\vert
\Phi_{n}\right\vert _{L_{p,p}\left(E\right)}^{p}.  \label{f2-1}
\end{equation}
In addition, by Fubini theorem, $\mathbf{P}$-a.s. for $0\leq t\leq
T,\varphi\in\mathcal{S}\left(\mathbf{R}^{d}\right),$ 
\begin{eqnarray}
\left\langle M_{n}\left(t\right),\varphi\right\rangle & = & \int
M_{n}\left(t,x\right)\varphi\left(x\right)dx  \label{eq:stoch3} \\
& = &
\int_{0}^{t}\int\left(\int\Phi_{n}\left(r,x,z\right)\varphi\left(x\right)dx%
\right)q\left(dr,dz\right).  \notag
\end{eqnarray}
By Lemma \ref{lem-approximation1}, 
\begin{eqnarray*}
\mathbf{E}\sum_{j=2,p}\left\vert \Phi_{n}-\Phi\right\vert _{L_{j,p}\left(%
\mathbf{R}^{d}\right)}^{p} & \rightarrow & 0,2\leq p<\infty, \\
\mathbf{E}\left\vert \Phi_{n}-\Phi\right\vert _{L_{p,p}\left(E\right)}^{p} &
\rightarrow & 0,p\in(1,2).
\end{eqnarray*}
Similarly, for each $p\in\lbrack2,\infty)$, 
\begin{equation*}
\mathbf{E}\sup_{t\leq T}\left\vert
M_{n}\left(t\right)-M_{m}\left(t\right)\right\vert _{L_{p}\left(\mathbf{R}%
^{d}\right)}^{p}\leq C\mathbf{E}\sum_{j=2,p}\left\vert
\Phi_{n}-\Phi_{m}\right\vert _{L_{j,p}\left(E\right)}^{p}\rightarrow0, 
\end{equation*}
and for each $p\in\left(1,2\right)$ 
\begin{equation*}
\mathbf{E}\sup_{t\leq T}\left\vert
M_{n}\left(t\right)-M_{m}\left(t\right)\right\vert _{L_{p}\left(\mathbf{R}%
^{d}\right)}^{p}\leq C\mathbf{E}\left\vert \Phi_{n}-\Phi_{m}\right\vert
_{L_{p,p}\left(E\right)}^{p}\rightarrow0, 
\end{equation*}
as $n,m\rightarrow\infty$. Therefore there is an adapted cadlag $L_{p}\left(%
\mathbf{R}^{d}\right)$ -valued process $M\left(t\right)$ so that 
\begin{equation*}
\mathbf{E}\sup_{t\leq T}\left\vert
M_{n}\left(t\right)-M\left(t\right)\right\vert _{L_{p}\left(\mathbf{R}%
^{d}\right)}^{p}\rightarrow0 
\end{equation*}
as $n\rightarrow\infty.$ Passing to the limit as $n\rightarrow\infty$ in (%
\ref{f1-1}), (\ref{f2-1}) and (\ref{eq:stoch3}) we derive (\ref{eq:normEst}%
), (\ref{eq:unique}). Henceforth, we define $\int_{0}^{t}\int\Phi\left(r,x,z%
\right)q\left(dr,dz\right)$ to be $M\left(t\right)$ in this lemma.
\end{proof}


\begin{thebibliography}{99}
\bibitem{ALAR} Aljan\v{c}i\'{c}, S. and Arandelovi\.{c}, D., O -regularly
varying functions, PIMB (NS), 22(36), 1977, pp 5-22.

\bibitem{lofs} Bergh, J. and Löfstrom, J., Interpolation Spaces. An
Introduction, Springer, Berlin, 1976.

\bibitem{bgt} Bingham, N.H., Goldie, C.M. and Teugels, J.L., Regular
Variation, Cambridge University Press, 1987.

\bibitem{DM} Dellacherie, C. and Meyer, P.-A., Probabilities and
Potential-A, North-Holland, 1978.

\bibitem{KimDong} Dong H. and Kim D., On $L_{p}$- estimates of non-local
elliptic equations, Journal of Functional Analysis, 262, 2012, pp. 1166-1199.

\bibitem{fjs} Farkas, W., Jacob, N. and Schilling, R.L., Function spaces
related to continuous negative definite functions: $\psi$ -Bessel potential
spaces, Dissertationes Mathematicae, 2001, p. 1-60.

\bibitem{fw2} Farkas, W. and Leopold, H.G., Characterization of function
spaces of generalized smoothness, Annali di Matematica, 185, 2006, pp. 1-62.

\bibitem{GK} Grzywny, T. and Kwa\'{s}nicki, M., Potential kernels,
probabilities of hitting a ball, harmonic functions and the boundary Harnack
inequality for unimodal Lévy processes. https://arxiv.org/abs/1611.10304,
2017.

\bibitem{k} Karamata, J., Sur un mode de croissance regulière des fonctions,
Mathematica (Cluj), 4, 1930, pp. 38-53.

\bibitem{KK1} Kim, I. and Kim, K.-H., An $L_{p}$-theory for a class of
non-local elliptic equations related to nonsymmetric measurable kernels, J.
Math. Anal. Appl, 434, 2016, pp. 1302-1335.

\bibitem{KK2} Kim, I. and Kim, K.-H., An $L_{p}$-boundedness of stochastic
singular integral operators and its application to SPDEs,
arXiv:arXiv:1608.08728, 2017.

\bibitem{kry} Krylov, N.V., A generalization of Littlewood-Paley inequality
and some other results related to SPDEs, Ulam Quart., 2, 1994, pp. 16--26.

\bibitem{MF} Mikulevicius, R., Xu, F., 2018. On the Cauchy problem for
integro-differential equations in the scale of generalized Hölder spaces.
arXiv:1810.00262.

\bibitem{MikPh1} Mikulevicius, R., Phonsom, C., On $L_{p}-$theory for
parabolic and elliptic integro-differential equations with scalable
operators in the whole space, Stochastics and PDEs: Anal Comp, 2017, DOI
10.1007/s40072-017-0095-4; arXiv:1605.07086, 2016.

\bibitem{MikPh2} Mikulevicius, R., Phonsom, C., On the Cauchy problem for
integro-differential equations in the scale of spaces of generalized
smoothness, Potential Analysis, 2018, DOI 10.1007/s11118-018-9690-x;
arXiv:1705.09256, 2017.

\bibitem{MP1} Mikulevicius, R., Pragarauskas, H., On $L_{p}$-theory for
stochastic parabolic integro-differential equations, Stochastics and PDEs:
Anal. Comp., 1(2), 2013, pp 282--324.

\bibitem{sato} Sato, K., Levy Processes and Infinitely Divisible
Distributions, Cambridge University Press, 1999.

\bibitem{stein3} Stein, E., Harmonic Analysis, Princeton University Press,
1993.

\bibitem{ku} Kunita, H.: Stochastic differential equations based on Lévy
processes and stochastic flows of diffeomorphisms. Real and Stochastic
Analysis, Trends in Mathematics, Birkhäuser Boston, 2004; pp. 305--373.

\bibitem{stein1} Stein, E., Harmonic Analysis, Princeton University Press,
1993.

\bibitem{zh} Zhang, X., $L^{p}$ maximal regularity of nonlocal parabolic
equations and applications, Ann. I. H. Poincaré , 2013, pp. 573-614.
\end{thebibliography}
\end{document}